\documentclass[11pt]{article}
\usepackage{amsfonts, amsmath, amssymb, latexsym, eucal, amscd,cite} 
\newtheorem{theorem}{Theorem}[section]
\newtheorem{lemma}[theorem]{Lemma}
\newtheorem{proposition}[theorem]{Proposition}
\newtheorem{corollary}[theorem]{Corollary}
\newtheorem{exAux}[theorem]{Example}
\newenvironment{example}{\begin{exAux} \rm}{\end{exAux}}
\newtheorem{Def}[theorem]{Definition}
\newenvironment{definition}{\begin{Def} \rm}{\end{Def}}
\newtheorem{Note}[theorem]{Note}
\newenvironment{note}{\begin{Note} \rm}{\end{Note}}
\newtheorem{Problem}[theorem]{Problem}

\newtheorem{Rem}[theorem]{Remark}

\newtheorem{Not}[theorem]{Notation}

\newtheorem{Conj}[theorem]{Conjecture}

\newtheorem{Ass}[theorem]{Assumption}

\newenvironment{proof}{\medskip\noindent{\bf Proof.\ }}{\qed\medskip}

\newcommand{\qed}{\hfill\mbox{$\Box$\qquad\qquad}}

\newcommand{\F}{\mathbb{F}}
\newcommand{\Mat}{\text{\rm Mat}}

\renewcommand{\b}[1]{\langle #1 \rangle}

\newcommand{\vphi}{\varphi}
\renewcommand{\th}{\theta}
\newcommand{\R}{\mathcal R}

\newif\ifDRAFT

\DRAFTfalse

\begin{document}

\title{Totally bipartite tridiagonal pairs}

\author{Kazumasa Nomura and Paul Terwilliger}

\maketitle

\medskip

\begin{quote}
\small 
\begin{center}
\bf Abstract
\end{center}
There is a concept in linear algebra called a tridiagonal pair.
The concept was motivated by the theory of $Q$-polynomial
distance-regular graphs. 
We give a tutorial introduction to tridiagonal pairs, working with a 
special case as a concrete example. 
The special  case is called totally bipartite, or TB. 
Starting from first principles, we give an elementary but
comprehensive account of TB tridiagonal pairs. 
The following topics are discussed: 
(i) the notion of a TB tridiagonal system;
(ii) the eigenvalue array;
(iii) the standard basis and matrix representations;
(iv) the intersection numbers;
(v) the Askey-Wilson relations;
(vi) a recurrence involving the eigenvalue array;
(vii) the classification of TB tridiagonal systems;
(viii) self-dual TB tridiagonal pairs and systems;
(ix) the $\mathbb{Z}_3$-symmetric Askey-Wilson relations;
(x) some automorphisms and antiautomorphisms associated with a TB tridiagonal pair;
(xi) an action of the modular group ${\rm PSL}_2(\mathbb{Z})$ associated with a TB tridiagonal pair.
\end{quote}

\section{Introduction}
\label{sec:intro}

This paper is about a linear algebraic object called a tridiagonal pair \cite[Definition 1.1]{ITT}.
Before describing our purpose in detail,
we give the definition of a tridiagonal pair.
Let $\F$ denote a field and let $V$ denote a vector space over $\F$ with finite positive dimension.
A {\em tridiagonal pair} on $V$ is an ordered pair of  $\F$-linear maps
$A : V \to V$ and $A^* : V \to V$ that satisfy the following conditions {\rm (i)--(iv)}.
\begin{itemize}
\item[\rm (i)]
Each of $A$, $A^*$ is diagonalizable.
\item[\rm (ii)]
There exists an ordering $\{V_i\}_{i=0}^d$ of the eigenspaces of $A$
such that
\begin{align}
   A^* V_i  &\subseteq V_{i-1} + V_i + V_{i+1}  \qquad\qquad (0 \leq i \leq d),    \label{eq:AsVipre}
\end{align}
where $V_{-1}=0$ and $V_{d+1}=0$.
\item[\rm (iii)]
There exists an ordering $\{V^*_i\}_{i=0}^\delta$ of the eigenspaces of $A^*$
such that
\begin{align}
  A V^*_i  &\subseteq V^*_{i-1}  + V^*_i + V^*_{i+1}  \qquad\qquad (0 \leq i \leq \delta),   \label{eq:AVsipre}
\end{align}
where $V^*_{-1}=0$ and $V^*_{\delta+1}=0$.
\item[\rm (iv)]
There does not exist a subspace $W \subseteq V$ such that
$AW \subseteq W$, $A^* W \subseteq W$, $W \neq 0$, $W \neq V$.
\end{itemize}

We summarize the history behind the tridiagonal pair concept.
There is a type of finite undirected graph said to be distance-regular (see \cite{BCN}).
In \cite{Del} Delsarte investigated a class of distance-regular graphs said to be $Q$-polynomial.
Given a $Q$-polynomial distance-regular graph, 
Delsarte obtained two sequences of orthogonal polynomials 
that are related by what is now called Askey-Wilson duality. 
In \cite{Leonard} Leonard classified  the pairs of orthogonal polynomial
sequences that obey this duality. 
He found that all  the solutions come from the terminating branch of the Askey scheme,
which consists of the $q$-Racah polynomials and their limits (see \cite{AW}).
In \cite{BI} Bannai and Ito gave a thorough study of $Q$-polynomial 
distance-regular graphs, including a detailed version of Leonard's theorem (see \cite[Theorem 5.1]{BI}).
In  \cite{T:subconst1, T:subconst2, T:subconst3} the second author introduced an algebra
$T=T(x)$,  called  the {\em subconstituent algebra} (or {\em Terwilliger algebra}),
to analyze the subconstituents of a $Q$-polynomial distance-regular graph $\Gamma$ with respect to 
a fixed vertex $x$.
The algebra $T$ is generated by the adjacency matrix $A$ of $\Gamma$  and a certain
diagonal matrix $A^*$, called the dual adjacency matrix of $\Gamma$ with respect to $x$.
In \cite[Lemmas 3.9, 3.12]{T:subconst1} it was observed that
each of $A, A^*$ act on the eigenspaces of the other one in a block-tridiagonal fashion. 
This observation motivated the definition of a tridiagonal pair (see \cite[Definition 1.1]{ITT}).

This paper is meant for graduate students and researchers
who seek an introduction to tridiagonal pairs.
The theory of tridiagonal pairs is extensive, and some proofs are rather intricate.
So for the beginner it is best to start with a special case.
In this paper, we consider a special case,
said to be totally bipartite (TB).
A pair of $\F$-linear maps $A: V \to V$ and $A^*: V \to V$ is called a TB tridiagonal pair
whenever it satisfies the above conditions (i)--(iv),
with \eqref{eq:AsVipre} replaced by
\begin{align}
   A^* V_i  &\subseteq V_{i-1} + V_{i+1}  \qquad\qquad (0 \leq i \leq d)    \label{eq:AsVipre2}
\end{align}
and \eqref{eq:AVsipre} replaced by
\begin{align}
  A V^*_i  &\subseteq V^*_{i-1} + V^*_{i+1}  \qquad\qquad (0 \leq i \leq \delta).  \label{eq:AVsipre2}
\end{align}

We are not the first authors to consider the TB tridiagonal pairs.
A number of previous articles are effectively about this topic,
although they may not use the term.
We now summarize these articles.
In \cite{Brown1} Brown classified up to isomorphism a certain class of TB tridiagonal pairs,
said to have Bannai/Ito type.
Later in \cite{HouWangGao2} Gao, Hou, and Wang classified up to isomorphism
all the TB tridiagonal pairs.
This classification separates the TB tridiagonal pairs into three infinite families,
called  Krawtchouk, Bannai/Ito, and $q$-Racah.
It has been shown that a TB tridiagonal pair can be extended to a Leonard triple
in the sense of Curtin \cite{Cur}.
This result is due to
Balmaceda and Maralit for Krawtchouk type (see \cite{BalMa}),
Brown for Bannai/Ito type (see \cite{Brown1}),
and Gao, Hou, Wang  for $q$-Racah type (see \cite{HouWangGao2}).
In \cite{T:intro} the second author showed that a 
finite-dimensional irreducible module for the Lie algebra $\mathfrak{sl}_2$
gives a TB tridiagonal pair of Krawtchouk type.
In \cite{AC5} Alnajjar and Curtin obtained a similar result
using the equitable basis for $\mathfrak{sl}_2$.
The anticommutator spin algebra was introduced by Arik and Kayserilioglu \cite{Arik}.
In \cite{Brown1} Brown showed that certain irreducible modules for
this algebra give TB tridiagonal pairs of Bannai/Ito type.
In \cite{havlicek2, Huang2} Havli\v cek, Po\v sta and Huang 
showed that certain irreducible $U_q(\mathfrak{so}_3)$-modules give TB tridiagonal pairs
of $q$-Racah type.
We remark that in all the above articles except \cite{T:intro, AC5}
the field $\F$ is assumed to be algebraically closed.

We just summarized the previous articles that are effectively about TB tridiagonal pairs.
Among these,
the articles \cite{Brown1, BalMa, HouWangGao2, HouWangGao, Cur, Huang2, AC5}
explicitly refer to the concept of a tridiagonal pair.
What these cited papers have in common is that they invoke results about general 
tridiagonal pairs, and then specialize to the TB case.
In our view, this makes the theory more complicated than necessary.
In our approach,
we examine TB tridiagonal pairs from first principles, and
do not invoke any results from the literature about general tridiagonal pairs.
For most of our results we do not assume that $\F$ is algebraically closed.

In this paper we discuss the following topics:
(i) the notion of a TB tridiagonal system;
(ii) the eigenvalue array;
(iii) the standard basis and matrix representations;
(iv) the intersection numbers;
(v) the Askey-Wilson relations;
(vi) a recurrence involving the eigenvalue array;
(vii) the classification of TB tridiagonal systems;
(viii) self-dual TB tridiagonal pairs and systems;
(ix) the $\mathbb{Z}_3$-symmetric Askey-Wilson relations;
(x) some automorphisms and antiautomorphisms associated with a TB tridiagonal pair;
(xi) an action of the modular group ${\rm PSL}_2(\mathbb{Z})$ associated with a TB tridiagonal pair.

We now summarize our results in detail.
Let $V$ denote a vector space over $\F$ with finite positive dimension.
Let $\text{\rm End}(V)$ denote the $\F$-algebra consisting of the $\F$-linear maps from $V$ to $V$.
Let $A,A^*$ denote a TB tridiagonal pair on $V$.
Fix an ordering $\{V_i\}_{i=0}^d$ (resp.\ $\{V^*_i\}_{i=0}^\delta$) of the eigenspaces of $A$
(resp.\ $A^*$) that satisfies \eqref{eq:AsVipre2} (resp.\ \eqref{eq:AVsipre2}).
For $0 \leq i \leq d$ let $E_i \in \text{\rm End}(V)$ denote the projection onto $V_i$.
The elements $\{E_i\}_{i=0}^d$ are called the primitive idempotents of $A$.
The primitive idempotents $\{E^*_i\}_{i=0}^\delta$ of $A^*$ are similarly defined.
We call the sequence $\Phi = (A; \{E_i\}_{i=0}^d; A^*; \{E^*_i\}_{i=0}^\delta)$ a TB tridiagonal system.
For $0 \leq i \leq d$ let $\th_i$ denote the eigenvalue of $A$ for the eigenspace $E_i V$,
and for $0 \leq i \leq \delta$ let $\th^*_i$ denote the eigenvalue of $A^*$ for the eigenspace $E^*_i V$.
We call the sequence $(\{\th_i\}_{i=0}^d; \{\th^*_i\}_{i=0}^\delta)$ the eigenvalue array of $\Phi$.
We show that $d=\delta$ (see Corollary \ref{cor:coincide}).
We show that the eigenspaces $E_i V$, $E^*_i V$ have dimension $1$ for $0 \leq i \leq d$
(see Corollary \ref{cor:coincide}).
Fix a nonzero $v \in E_0 V$ and define $v_i = E^*_i v$ $(0 \leq i \leq d)$.
We show that $\{v_i\}_{i=0}^d$ form a basis for $V$ (see Lemma \ref{lem:vibasis}).
With respect to this basis, the matrix representing $A^*$ is diagonal and the
matrix representing $A$ is tridiagonal with diagonal entries all zero.
Let $\{c_i\}_{i=1}^{d}$ (resp. $\{b_i\}_{i=0}^{d-1}$) denote the subdiagonal entries
(resp.\ superdiagonal entries) of this tridiagonal matrix.
The scalars $\{c_i\}_{i=1}^d$ and $\{b_i\}_{i=0}^{d-1}$ are called the intersection numbers of $\Phi$.
We represent these intersection numbers in terms of the eigenvalue array
(see Lemma \ref{lem:cibi}).
Using this we show that a TB tridiagonal system is uniquely determined up to isomorphism
by its eigenvalue array (see Corollary \ref{cor:unique}).
We show that $A,A^*$ satisfy  a pair of equations called the Askey-Wilson relations
(see lines \eqref{eq:AW1}, \eqref{eq:AW2}).
We show that  there exists $\beta \in \F$ such that
\[
 \th_{i-1} - \beta \th_i + \th_{i+1}=0,  \qquad \qquad
 \th^*_{i-1} - \beta \th^*_i + \th^*_{i+1}=0
\]
for $1 \leq i \leq d-1$ (see Proposition \ref{prop:beta}).
We classify the TB tridiagonal systems up to isomorphism in terms of the eigenvalue array
(see Theorem \ref{thm:main}).
We represent the eigenvalue array in closed form  (see Examples \ref{ex:1}--\ref{ex:4}).
The TB tridiagonal pair $A,A^*$ is said to be self-dual whenever it is isomorphic to
$A^*,A$.
We show that there exists $0 \neq \zeta \in \F$ such that $\zeta A,A^*$ is self-dual
(see Lemma \ref{lem:selfdualpair2}).
In Theorem \ref{thm:selfdual} we show that if $A,A^*$ is self-dual then the
following TB tridiagonal pairs are mutually isomorphic:
\begin{equation}
\begin{array}{llll}
   A,A^*,  \quad & A, -A^*,  \quad &  -A, A^*,  \quad & -A,-A^*,
\\
  A^*,A,    & A^*,-A,  & -A^*,A,  & -A^*, -A.                            
\end{array}                                                       \label{eq:TBpairs}
\end{equation}
The explicit isomorphisms are given in Section \ref{sec:selfdual}.
We put the Askey-Wilson relations in a form said to be $\mathbb{Z}_3$-symmetric
(see Propositions \ref{prop:Z3AW1pre}--\ref{prop:Z3AW3pre}).
For our remaining results we assume that $\F$ is algebraically closed,
and that $A,B$ is self-dual, where $B=A^*$.
We show that there exists $C \in \text{\rm End}(V)$
such that $A,B,C$ is a modular Leonard triple in the sense of Curtin (see \cite{Cur}).
We display some automorphisms and anti-automorphisms of $\text{\rm End}(V)$ that act on
$A$, $B$, $C$ in an attractive manner.
Recall that the modular group  ${\rm PSL}_2(\mathbb{Z})$ has a presentation by
generators $r$, $s$ and relations $r^3=1$, $s^2=1$.
We show that ${\rm PSL}_2(\mathbb{Z})$ acts on $\text{\rm End}(V)$ 
as a group of automorphisms such that
$r$ sends $A \mapsto B \mapsto C \mapsto A$ and
$s$ sends $A \leftrightarrow B$
(see Corollary \ref{cor:mu} and Proposition \ref{prop:ddagger}).
In the last part of the paper,
we describe general tridiagonal pairs, using the TB case as a guide.

The paper is organized as follows.
In Section \ref{sec:pre} we recall some materials from linear algebra.
In Section \ref{sec:TBTDpair} we define a TB tridiagonal system and its eigenvalue array.
In Section \ref{sec:RL} we discuss the diameter and  eigenspace dimensions of a TB tridiagonal system. 
In Section \ref{sec:intersection} we discuss the intersection numbers of a TB tridiagonal system.
In Sections \ref{sec:trid} and \ref{sec:rec} we obtain some linear algebra facts that will be needed later. 
In Section \ref{sec:AW} we discuss the Askey-Wilson relations. 
In Sections \ref{sec:AWseq} and \ref{sec:sym} we investigate a recurrence satisfied by 
the eigenvalue array of a TB tridiagonal system. 
In Section  \ref{sec:classify} we classify the TB tridiagonal systems up to isomorphism. 
In Sections \ref{sec:array} and \ref{sec:closedform} we examine in detail the eigenvalue array of a TB tridiagonal system. 
In Section \ref{sec:selfdual} we show that in the self-dual case the TB tridiagonal pairs 
\eqref{eq:TBpairs} are mutually isomorphic. 
In Section \ref{sec:Z3} we put the Askey-Wilson relations in $\mathbb{Z}_3$-symmetric form. 
In Sections \ref{sec:W} and \ref{sec:anti}  we obtain some automorphisms and antiautomorphisms
of $\text{\rm End}(V)$ that act on a given TB tridiagonal pair in an attractive manner.
In Section \ref{sec:auto} we discuss an action of ${\rm PSL}_2(\mathbb{Z})$.
In Section \ref{sec:conc} we describe the general tridiagonal pairs.

\section{Preliminaries}
\label{sec:pre}

In this section we recall some materials from linear algebra.
Throughout the paper we use the following notation.
Let $\F$ denote a field and let $\overline{\F}$ denote the algebraic closure of $\F$.
For an integer $n \geq 0$,
let $\Mat_{n+1}(\F)$ denote the $\F$-algebra consisting of the $n+1$ by $n+1$ matrices that have all entries in $\F$.
We index the rows and columns by $0,1,\ldots,n$.
Let $\F^{n+1}$ denote the vector space over $\F$ consisting of the column vectors of length $n+1$
that have all entries in $\F$.
We index the rows by $0,1,\ldots,n$.
Let $V$ denote a vector space over $\F$ with finite positive dimension.
Let $\text{End}(V)$ denote the $\F$-algebra consisting 
of the $\F$-linear maps from $V$ to $V$.
For the $\F$-algebras $\Mat_{n+1}(\F)$ and $\text{End}(V)$ the identity element is denoted by $I$.
For $X \in \Mat_{n+1}(\F)$ let $X^{\rm t}$ denote the transpose of $X$.

Let $A$ denote an element of $\text{End}(V)$.
For $\th \in \F$ define 
$V(\th) = \{ v \in V \,|\, A v = \th v\}$.
Observe that $V(\th)$ is a subspace of $V$.
The scalar $\th$ is called an {\em eigenvalue} of $A$ whenever $V(\th) \neq 0$.
In this case, $V(\th)$ is called the {\em eigenspace} of $A$ corresponding to $\th$.
We say that
$A$ is {\em diagonalizable}  whenever $V$ is spanned by the eigenspaces of $A$.
Assume that $A$ is diagonalizable.
Let $\{V_i\}_{i=0}^d$ denote an ordering of the eigenspaces of $A$.
So $V = \sum_{i=0}^d V_i$ (direct sum).
For $0 \leq i \leq d$ let $\th_i$ denote the eigenvalue of $A$ corresponding to $V_i$.
For $0 \leq i \leq d$ define $E_i \in \text{End}(V)$ such that
$(E_i - I)V_i = 0$ and $E_i V_j = 0$ if $j \neq i$ $(0   \leq j \leq d)$.
Thus $E_i$ is the projection onto $V_i$.
Observe that
(i) $V_i = E_i V$ $(0 \leq i \leq d)$;
(ii) $E_i E_j = \delta_{i,j} E_i$ $(0 \leq i,j \leq d)$;
(iii) $I = \sum_{i=0}^d E_i$;
(iv) $A = \sum_{i=0}^d \th_i E_i$.
Also 
\begin{align}
   E_i &= \prod_{\stackrel{0 \leq j \leq d}{j \neq i}}
           \frac{A-\th_j I}{\th_i - \th_j}    
   \qquad\qquad  (0 \leq i \leq d).                               \label{eq:Ei}
\end{align}
We call $E_i$ the {\em primitive idempotent of $A$ for $\th_i$} $(0 \leq i \leq d)$.
Let $M$ denote the subalgebra of $\text{End}(V)$ generated by $A$.
Observe that $\{A^i\}_{i=0}^d$ is a basis for the $\F$-vector space $M$,
and $\prod_{i=0}^d (A-\th_i I)=0$.
Also observe that $\{E_i\}_{i=0}^d$ is a basis for the $\F$-vector space $M$.

Let $\{v_i\}_{i=0}^n$ denote a basis for $V$.
For $A \in \text{End}(V)$ and $X \in \Mat_{n+1}(\F)$,
we say that  {\em $X$ represents $A$ with respect to $\{v_i\}_{i=0}^n$} whenever
$A v_j = \sum_{i=0}^n X_{i,j} v_i$ for $0 \leq j \leq n$.
For $A \in \text{End}(V)$ let $A^\natural$ denote the matrix in $\Mat_{n+1}(\F)$
that represents $A$ with respect to $\{v_i\}_{i=0}^n$.
Then the map $\natural : \text{End}(V) \to \Mat_{n+1}(\F)$,
$A \mapsto A^\natural$ is an isomorphism of $\F$-algebras.

For an $\F$-algebra $\cal A$,
by an {\em automorphism of $\cal A$} we mean an isomorphism of $\F$-algebras
from $\cal A$ to $\cal A$.
For an invertible $T \in {\cal A}$,
the map $X \mapsto T^{-1} X T$ is an automorphism of $\cal A$,
said to be {\em inner}.
By the Skolem-Noether theorem \cite[Corollary 7.125]{Rot},
every automorphism of $\text{\rm End}(V)$ is inner.

At several places in the paper we will discuss polynomials.
Let $x$ denote an indeterminate.
Let $\F[x]$ denote the $\F$-algebra consisting of the
polynomials in $x$ that have all coefficients in $\F$.

\section{Totally bipartite tridiagonal pairs}
\label{sec:TBTDpair}

In this section we define a totally bipartite tridiagonal pair,
and obtain some basic facts about this object.

\begin{definition}   {\rm (See \cite[Definition 1.2]{Brown1}.)}  \label{def:TBTDpair}    \samepage
\ifDRAFT {\rm def:TBTDpair}. \fi
Let $V$ denote a vector space over $\F$ with finite positive dimension.
A {\em totally bipartite tridiagonal pair} (or {\em TB tridiagonal pair}) on $V$ is an ordered pair $A,A^*$ of
elements in $\text{End}(V)$ that satisfy the following {\rm (i)--(iv)}.
\begin{itemize}
\item[\rm (i)]
Each of $A$, $A^*$ is diagonalizable.
\item[\rm (ii)]
There exists an ordering $\{V_i\}_{i=0}^d$ of the eigenspaces of $A$
such that
\begin{align}
   A^* V_i  &\subseteq V_{i-1} + V_{i+1}  \qquad\qquad (0 \leq i \leq d),      \label{eq:AsVi}
\end{align}
where $V_{-1}=0$ and $V_{d+1}=0$.
\item[\rm (iii)]
There exists an ordering $\{V^*_i\}_{i=0}^\delta$ of the eigenspaces of $A^*$
such that
\begin{align}
  A V^*_i  &\subseteq V^*_{i-1}  + V^*_{i+1}  \qquad\qquad (0 \leq i \leq \delta),     \label{eq:AVsi}
\end{align}
where $V^*_{-1}=0$ and $V^*_{\delta+1}=0$.
\item[\rm (iv)]
There does not exist a subspace $W \subseteq V$ such that
$AW \subseteq W$, $A^* W \subseteq W$, $W \neq 0$, $W \neq V$.
\end{itemize}
We say that $A,A^*$ is {\em over} $\F$.
We call $V$ the {\em underlying vector space}.
\end{definition}

\begin{note}
According to a common notational convention,
$A^*$ denotes the conjugate-transpose of $A$.
We are not using this convention.
In a TB tridiagonal pair the elements $A$ and $A^*$ are arbitrary subject to (i)--(iv) above.
\end{note}

\begin{note}
If $A,A^*$ is a TB tridiagonal pair on $V$, then so is $A^*,A$.
\end{note}

\begin{note}
Let $A,A^*$ denote a TB tridiagonal pair on $V$.
Pick nonzero $\zeta$, $\zeta^* \in \F$.
Then $\zeta A$, $\zeta^* A^*$ is a TB tridiagonal pair on $V$.
\end{note}

We mention a special case of a TB tridiagonal pair.

\begin{example}   \label{ex:deftrivial}    \samepage
\ifDRAFT {\rm ex:deftrivial}. \fi
Referring to Definition \ref{def:TBTDpair},
assume that $\dim V=1$
and $A=0$, $A^*=0$.
Then $A,A^*$ is a TB tridiagonal pair on $V$,
said to be {\em trivial}.
\end{example}

\begin{lemma}    \label{lem:trivial}    \samepage
\ifDRAFT {\rm lem:trivial}. \fi
With reference to Definition \ref{def:TBTDpair},
assume that $A,A^*$ is a TB tridiagonal pair on $V$.
Then the following {\rm (i)--(vi)} are equivalent:
{\rm (i)} $A,A^*$ is trivial;
{\rm (ii)} $d=0$;
{\rm (iii)} $\delta=0$;
{\rm (iv)} $A = 0$;
{\rm (v)} $A^* = 0$;
{\rm (vi)} $\dim V = 1$.
\end{lemma}

\begin{proof}
Routine.
\end{proof}

For the rest of this section, let $V$ denote a vector space over $\F$
with finite positive dimension.

\smallskip

Let $A,A^*$ denote a TB tridiagonal pair on $V$.
Let $V'$ denote a vector space over $\F$ with finite positive dimension,
and let $A',A^{*\prime}$ denote a TB tridiagonal pair on $V'$.
By an {\em isomorphism of TB tridiagonal pairs} from $A,A^*$ to $A',A^{* \prime}$
we mean an $\F$-linear bijection $\psi : V \to V'$
such that $\psi A = A' \psi$ and $\psi A^* = A^{*\prime} \psi$.
We say that $A,A^*$ and $A',A^{* \prime}$ are {\em isomorphic}
whenever there exists an isomorphism of TB tridiagonal pairs from $A,A^*$ to $A',A^{* \prime}$.

Let $A,A^*$ denote a TB tridiagonal pair on $V$.
An ordering $\{V_i\}_{i=0}^d$ 
of the eigenspaces of $A$ is said to be {\em standard}
whenever it satisfies \eqref{eq:AsVi}. 
Let $\{V_i\}_{i=0}^d$ denote a standard ordering of 
the eigenspaces of $A$.
Then the ordering $\{V_{d-i}\}_{i=0}^d$ is 
also standard and no further ordering is standard.
An ordering of the eigenvalues or primitive idempotents for $A$
is said to be {\em standard} whenever  the corresponding ordering of the
eigenspaces is standard.
Similar comments apply to $A^*$.

\begin{definition}    \label{def:TBTDsystem}   \samepage
\ifDRAFT {\rm def:TBTDsystem}. \fi
By a {\em TB tridiagonal system} on $V$  we mean a sequence
\begin{equation}
  \Phi = (A; \{E_i\}_{i=0}^d; A^*; \{E^*_i\}_{i=0}^\delta)     \label{eq:Phi}
\end{equation}
of elements in $\text{\rm End}(V)$
that satisfy the following (i)--(iii):
\begin{itemize}
\item[\rm (i)]
$A,A^*$ is a TB tridiagonal pair on $V$;
\item[\rm (ii)]
$\{E_i\}_{i=0}^d$ is a standard ordering of the primitive idempotents of $A$;
\item[\rm (iii)]
$\{E^*_i\}_{i=0}^\delta$ is a standard ordering of the primitive idempotents of $A^*$.
\end{itemize}
We say that $\Phi$ is {\em over} $\F$.
We call $V$ the {\em underlying vector space}.
\end{definition}

\begin{definition}     \label{def:convenience}    \samepage
\ifDRAFT {\rm def:convenience}. \fi
Referring to Definition \ref{def:TBTDsystem},
for notational convenience define $E_{-1}=0$, $E_{d+1}=0$,
$E^*_{-1}=0$, $E^*_{\delta+1}=0$.
\end{definition}

\begin{definition}     \label{def:associated}    \samepage
\ifDRAFT {\rm def:associated}. \fi
Referring to Definition \ref{def:TBTDsystem},
we say that the TB tridiagonal pair $A,A^*$ and the TB tridiagonal system $\Phi$ are
{\em associated}.
\end{definition}

\begin{definition}    \label{def:trivial}    \samepage
\ifDRAFT {\rm def:trivial}. \fi
A TB tridiagonal system is said to be {\em trivial}
whenever the associated TB tridiagonal pair is trivial.
\end{definition}

\begin{definition}   \label{def:eigenseq}   \samepage
\ifDRAFT {\rm def:eigenseq}. \fi
Consider the TB tridiagonal system $\Phi$ from \eqref{eq:Phi}.
For $0 \leq i \leq d$ let $\th_i$ denote the eigenvalue of $A$ corresponding to $E_i$,
and for $0 \leq i \leq \delta$ let $\th^*_i$ denote the eigenvalue of $A^*$
corresponding to $E^*_i$.
We call $\{\th_i\}_{i=0}^d$ (resp.\ $\{\th^*_i\}_{i=0}^\delta$) the {\em eigenvalue sequence}
(resp. {\em dual eigenvalue sequence}) of $\Phi$.
We call $(\{\th_i\}_{i=0}^d; \{\th^*_i\}_{i=0}^\delta)$ the
{\em eigenvalue array} of $\Phi$.
\end{definition}

\begin{note}    \label{note:distinct}    \samepage
\ifDRAFT {\rm note:distinct}. \fi
Referring to Definition \ref{def:eigenseq},
the scalars $\{\th_i\}_{i=0}^d$ are mutually distinct and contained in $\F$.
Moreover,
the scalars $\{\th^*_i\}_{i=0}^\delta$ are mutually distinct and contained in $\F$.
\end{note}

\begin{note}    \label{note:trivial}    \samepage
\ifDRAFT {\rm note:trivial}. \fi
Referring to Definition \ref{def:eigenseq},
assume that $\Phi$ is trivial.
Then $\th_0 = 0$ and $\th^*_0 = 0$.
\end{note}

\begin{lemma}   \label{lem:hhs}    \samepage
\ifDRAFT {\rm lem:hhs}. \fi
Pick nonzero $\zeta$, $\zeta^* \in \F$.
Let $(A; \{E_i\}_{i=0}^d; A^*; \{E^*_i\}_{i=0}^\delta)$ denote a TB tridiagonal system on $V$,
with eigenvalue array $(\{\th_i\}_{i=0}^d; \{\th^*_i\}_{i=0}^\delta)$.
Then 
\[
(\zeta A; \{E_i\}_{i=0}^d; \zeta^* A^*; \{E^*_i\}_{i=0}^\delta)
\]
is a TB tridiagonal system on $V$,
with eigenvalue array $(\{ \zeta \th_i\}_{i=0}^d; \{\zeta^* \th^*_i\}_{i=0}^\delta)$.
\end{lemma}

\begin{proof}
By Definitions \ref{def:TBTDsystem}, \ref{def:eigenseq} and linear algebra.
\end{proof}

Consider the TB tridiagonal system $\Phi$ from \eqref{eq:Phi}.
Let $V'$ denote a vector space over $\F$ with finite positive dimension,
and let
\[
   \Phi'= (A'; \{E'_i\}_{i=0}^d; A^{* \prime}; \{E^{* \prime}_i\}_{i=0}^\delta)
\]
denote a TB tridiagonal system on $V'$.
By an {\em isomorphism of TB tridiagonal systems} from $\Phi$ to $\Phi'$
we mean an $\F$-linear bijection $\psi : V \to V'$ 
such that $\psi A = A' \psi$, $\psi A^* = A^{*\prime} \psi$, 
$\psi E_i = E'_i \psi$ $(0 \leq i \leq d)$,  $\psi E^*_i  = E^{* \prime}_i \psi$ 
$(0 \leq i \leq \delta)$.
We say that $\Phi$ and $\Phi'$ are {\em isomorphic}
whenever there exists an isomorphism of TB tridiagonal systems from $\Phi$ to $\Phi'$.
Note that two isomorphic TB tridiagonal systems have the same eigenvalue array.

Given a TB tridiagonal system $\Phi = (A; \{E_i\}_{i=0}^d; A^*; \{E^*_i\}_{i=0}^\delta)$
on $V$,
each of the following is a TB tridiagonal system on $V$:
\begin{align*}
  \Phi^* &= (A^*; \{E^*_i\}_{i=0}^\delta; A; \{E_i\}_{i=0}^d),
\\
  \Phi^{\downarrow} &= (A; \{E_i\}_{i=0}^d; A^*; \{E^*_{\delta-i}\}_{i=0}^\delta),
\\
  \Phi^{\Downarrow} &= (A; \{E_{d-i}\}_{i=0}^d; A^*; \{E^*_i\}_{i=0}^\delta).
\end{align*}
Viewing $*$, $\downarrow$, $\Downarrow$ as permutations on the set of
all TB tridiagonal systems,
\begin{align} 
        *^2 \,=\, \downarrow^2 \,&=\, \Downarrow^2 = 1,      \label{eq:rel1}
\\
  \Downarrow * \,=\, * \downarrow, \qquad
  \downarrow * &\,=\, * \Downarrow, \qquad
  \downarrow \Downarrow \,=\, \Downarrow \downarrow.    \label{eq:rel2}
\end{align}
The group generated by the symbols $*$, $\downarrow$,  $\Downarrow$ 
subject to the relations \eqref{eq:rel1}, \eqref{eq:rel2} is the dihedral group
$D_4$.
Recall that $D_4$ is the group of symmetries of a square, and has $8$ elements.
The elements $*$, $\downarrow$, $\Downarrow$ induce an action of $D_4$ on the
set of all TB tridiagonal systems over $\F$.
Two TB tridiagonal systems will be called {\em relatives} whenever they are in the same
orbit of this $D_4$ action.

\begin{definition}   \label{def:fg}    \samepage
\ifDRAFT {\rm def:fg}. \fi
Let $\Phi$ denote a TB tridiagonal system, and let $g \in D_4$.
For any object $f$ attached to $\Phi$,
let $f^g$ denote the corresponding object attached to $\Phi^{g^{-1}}$.
\end{definition}

\begin{lemma}    \label{lem:isorelative}    \samepage
\ifDRAFT {\rm lem:isorelative}. \fi
Let $\Phi$ and $\Phi'$ denote TB tridiagonal systems over $\F$.
Assume that $\Phi$ and $\Phi'$ are isomorphic,
and let $\psi$ denote an isomorphism of TB tridiagonal systems
from $\Phi$ to $\Phi'$.
Then for $g \in D_4$ the map $\psi$ is an isomorphism of TB tridiagonal systems
from $\Phi^g$ to ${\Phi'}^g$.
\end{lemma}

\begin{proof}
By the construction.
\end{proof}

\begin{lemma}    \label{lem:associated}    \samepage
\ifDRAFT {\rm lem:associated}. \fi
Let $A,A^*$ denote a TB tridiagonal pair over $\F$,
and let $\Phi$ denote an associated TB tridiagonal system.
Then the TB tridiagonal systems associated with $A,A^*$ are
$\Phi$, $\Phi^\downarrow$, $\Phi^\Downarrow$, $\Phi^{\downarrow\Downarrow}$.
\end{lemma}

\begin{proof}
By the comments above Definition \ref{def:TBTDsystem}.
\end{proof}

\begin{definition}    \label{def:eigenseq2}    \samepage
\ifDRAFT {\rm def:eigenseq2}. \fi
Let $A,A^*$ denote a TB tridiagonal pair over $\F$.
By an {\em eigenvalue sequence} (resp.\ {\em dual eigenvalue sequence})
(resp.\ {\em eigenvalue array})
 {\em of  $A,A^*$}
we mean the eigenvalue sequence
 (resp.\ dual eigenvalue sequence) (resp.\ eigenvalue array) 
of a TB tridiagonal system associated with $A,A^*$.
\end{definition}

\begin{lemma}   \label{lem:relative}    \samepage
\ifDRAFT {\rm lem:relative}. \fi
Let $\Phi$ denote a TB tridiagonal system over $\F$
with eigenvalue array $(\{\th_i\}_{i=0}^d; \{\th^*_i\}_{i=0}^\delta)$.
Then for $g \in \{*, \downarrow, \Downarrow\}$ the eigenvalue array
of $\Phi^g$ is as follows:
\[
\begin{array}{c|c}
g    &  \text{\rm Eigenvalue array of $\Phi^g$}
\\ \hline
* &  ( \{\th^*_i\}_{i=0}^\delta; \{\th_i\}_{i=0}^d)    \rule{0mm}{2.5ex}
\\
\downarrow &  (\{\th_i\}_{i=0}^d; \{\th^*_{\delta-i}\}_{i=0}^\delta)    \rule{0mm}{2.5ex}
\\
\Downarrow &  (\{\th_{d-i}\}_{i=0}^d; \{\th^*_i\}_{i=0}^\delta)    \rule{0mm}{2.5ex}
\end{array}
\]
\end{lemma}

\begin{proof}
By the construction.
\end{proof}

\begin{lemma}    \label{lem:AEsi}   \samepage
\ifDRAFT {\rm lem:AEsi}. \fi
For the TB tridiagonal system in \eqref{eq:Phi} we have
\begin{align*}
A E^*_i V &\subseteq E^*_{i-1} V + E^*_{i+1} V \qquad\qquad (0 \leq i \leq \delta),
\\
A^* E_i V &\subseteq E_{i-1} V + E_{i+1} V \qquad\qquad (0 \leq i \leq d).
\end{align*}
\end{lemma}

\begin{proof}
By Definition \ref{def:TBTDpair}(ii), (iii).
\end{proof}

\begin{lemma}    \label{lem:EsiAEsj}    \samepage
\ifDRAFT {\rm lem:EsiAEsj}. \fi
For the TB tridiagonal system in \eqref{eq:Phi} we have
\begin{align}
   E^*_i A E^*_j &=
     \begin{cases}
        0  &  \text{ if $\;\;|i-j| \neq 1$},
     \\
       \neq 0 & \text{ if $\;\;|i-j|=1$}
    \end{cases}     
            \qquad\qquad (0 \leq i,j \leq \delta),           \label{eq:EsiAEsj}
\\
   E_i A^* E_j &=
     \begin{cases}
        0  &  \text{ if $\;\;|i-j| \neq 1$},
     \\
       \neq 0 & \text{ if $\;\;|i-j|=1$}
    \end{cases}  
             \qquad\qquad (0 \leq i ,j \leq d).               \label{eq:EiAsEj}
\end{align}
\end{lemma}

\begin{proof}
We first show that
\begin{align}
   E^*_i A E^*_j &= 0  \quad \text{ if $\;\;|i-j| \neq 1$} \qquad\qquad (0 \leq i,j \leq \delta).
                         \label{eq:EsiAEsjaux2}
\end{align}
By Lemma \ref{lem:AEsi},
\begin{equation}
A E^*_j V \subseteq E^*_{j-1} V + E^*_{j+1} V.     \label{eq:EsiAEsjaux3}
\end{equation}
In the above line, apply $E^*_i$ to each side.
The right-hand side becomes $0$ since $i  \neq j-1$ and $i \neq j+1$.
Thus $E^*_i A E^*_j V= 0$ and so \eqref{eq:EsiAEsjaux2} holds.
Next we show that 
\begin{align}
      E^*_i A E^*_j &\neq 0  \quad \text{ if $\;\;|i-j|=1$}  \qquad\qquad  (1 \leq i,j \leq \delta).    \label{eq:EsiAEsjaux1}
\end{align}
Consider the case $i=j-1$.
Assume by way of contradiction that $E^*_i A E^*_j = 0$.
For $v \in A E^*_j V$ we have $E^*_{j-1} v=0$.
By this and \eqref{eq:EsiAEsjaux3} we find $v \in E^*_{j+1} V$.
By these comments 
\begin{equation}
  A E^*_j V \subseteq E^*_{j+1} V.                  \label{eq:EsiAEsjaux4}
\end{equation}
Define
$W = E^*_j V + \cdots + E^*_\delta V$.
We have $A W \subseteq W$ by Lemma \ref{lem:AEsi} and \eqref{eq:EsiAEsjaux4}.
By construction $A^* W \subseteq W$, $W \neq 0$, $W \neq V$.
This contradicts Definition \ref{def:TBTDpair}(iv),
and therefore \eqref{eq:EsiAEsjaux1} holds for $i=j-1$.
A similar argument shows that \eqref{eq:EsiAEsjaux1} holds for $j=i-1$.
We have shown \eqref{eq:EsiAEsj}.
To get \eqref{eq:EiAsEj},
apply \eqref{eq:EsiAEsj} to $\Phi^*$.
\end{proof}

\section{The raising and lowering maps}
\label{sec:RL}

Throughout this section,  
let $V$ denote a vector space over $\F$ with finite positive dimension,
and let 
\[
 \Phi = (A; \{E_i\}_{i=0}^d; A^*; \{E^*_i\}_{i=0}^\delta)
\]
denote a TB tridiagonal system on $V$.
Let $(\{\th_i\}_{i=0}^d; \{\th^*_i\}_{i=0}^\delta)$ denote the eigenvalue array of $\Phi$.

\begin{definition}    \label{def:RL}    \samepage
\ifDRAFT {\rm def:RL}. \fi
Define $R$, $L \in \text{\rm End}(V)$ by
\begin{align*}
 R &= \sum_{i=1}^\delta E^*_i A E^*_{i-1}, \qquad\qquad\qquad
 L = \sum_{i=1}^\delta E^*_{i-1} A E^*_i.
\end{align*}
We call $R$ (resp.\ $L$) the {\em raising map} (resp.\ {\em lowering map}) for $\Phi$.
\end{definition}

\begin{lemma}    \label{lem:ARL}    \samepage
\ifDRAFT {\rm lem:ARL}. \fi
We have $A = R + L$.
\end{lemma}

\begin{proof}
Evaluate $A = I A I$ using $I = \sum_{i=0}^\delta E^*_i$
to get $A =  \sum_{i=0}^\delta \sum_{j=0}^\delta E^*_i A E^*_j$.
In this equation evaluate the right-hand side
using Lemma \ref{lem:EsiAEsj} and Definition \ref{def:RL}.
\end{proof}

\begin{lemma}    \label{lem:EsiR}    \samepage
\ifDRAFT {\rm lem:EsiR}. \fi
The following {\rm (i)--(iv)} hold:
\begin{itemize}
\item[\rm (i)]
$E^*_i R = E^*_i A E^*_{i-1} = R E^*_{i-1}  \qquad (1 \leq i \leq \delta)$;
\item[\rm (ii)]
$E^*_0 R = 0$,  $\;\; R E^*_\delta = 0$;
\item[\rm (iii)]
$E^*_{i-1} L = E^*_{i-1} A E^*_i = L E^*_i  \qquad (1 \leq i \leq \delta)$;
\item[\rm (iv)]
$E^*_\delta L = 0$, $\;\; L E^*_0 = 0$.
\end{itemize}
\end{lemma}

\begin{proof}
Recall that $E^*_r E^*_s = \delta_{r,s} E^*_r$ for $0 \leq r,s \leq \delta$.
Using this and  Definition \ref{def:RL}
we routinely obtain the results.
\end{proof}

\begin{lemma}    \label{lem:REsiV}    \samepage
\ifDRAFT {\rm lem:REsiV}. \fi
The following hold:
\begin{itemize}
\item[\rm (i)]
$R E^*_i V \subseteq E^*_{i+1} V \qquad (0 \leq i \leq \delta)$; 
\item[\rm (ii)]
$L E^*_i V \subseteq E^*_{i-1} V \qquad (0 \leq i \leq \delta)$.
\end{itemize}
\end{lemma}

\begin{proof}
Use Lemma \ref{lem:EsiR}.
\end{proof}

\begin{definition}   \label{def:vi}   \samepage
\ifDRAFT {\rm def:vi}. \fi
Fix $0 \neq v \in E_0 V$.
For $0 \leq i \leq \delta$ define
$v_i = E^*_i v$.
For notational convenience, define $v_{-1}=0$ and $v_{\delta+1}=0$.
\end{definition}

\begin{lemma}   \label{lem:vi}   \samepage
\ifDRAFT {\rm lem:vi}. \fi
With reference to Definition \ref{def:vi}, the following hold:
\begin{itemize}
\item[\rm (i)]
$v_i \in E^*_i V \qquad (0 \leq i \leq \delta)$; 
\item[\rm (ii)]
$v = \sum_{i=0}^\delta v_i$.
\end{itemize} 
\end{lemma}

\begin{proof}
(i)
By Definition \ref{def:vi}.

(ii)
Use $I = \sum_{i=0}^\delta E^*_i$.
\end{proof}

We now describe the action of $R$ and $L$ on $\{v_i\}_{i=0}^\delta$.
Note by Lemma \ref{lem:REsiV} that $R v_\delta = 0$ and $L v_0 = 0$.

\begin{lemma}   \label{lem:RL}   \samepage
\ifDRAFT {\rm lem:RL}. \fi
With reference to Definition \ref{def:vi},
\begin{align*}
 \th_0 v_i &= R v_{i-1} + L v_{i+1}  \qquad\qquad (0 \leq i \leq \delta).
\end{align*}
\end{lemma}

\begin{proof}
We have $v \in E_0 V$ so $A v = \th_0 v$.
Using Lemmas \ref{lem:ARL} and \ref{lem:EsiR},
\begin{align*}
  0 &= E^*_i (A - \th_0 I) v  \\
   &= E^*_i (R + L - \th_0 I) v \\
   &= (R E^*_{i-1} + L E^*_{i+1}  - \th_0 E^*_i) v  \\
   &= R v_{i-1} + L v_{i+1} - \th_0 v_i.
\end{align*}
The result follows.
\end{proof}

\begin{lemma}  \label{lem:RL2}   \samepage
\ifDRAFT {\rm lem:RL2}. \fi
Assume that $\Phi$ is nontrivial.
Then with reference to Definition \ref{def:vi},
the following {\rm (i)--(iii)} hold:
\begin{itemize}
\item[\rm (i)]
$\th_1 \th^*_0 v_0 = \th^*_1 L v_1$;
\item[\rm (ii)]
$\th_1 \th^*_i v_i = \th^*_{i-1} R v_{i-1} + \th^*_{i+1} L v_{i+1} \qquad (1 \leq i \leq \delta-1)$;
\item[\rm (iii)]
$\th_1 \th^*_\delta v_\delta = \th^*_{\delta-1} R v_{\delta-1}$.
\end{itemize}
\end{lemma}

\begin{proof}
(ii)
We have $v \in E_0 V$ and $A^* E_0 V \subseteq E_1 V$, so $A^* v \in E_1 V$.
Therefore $A A^* v = \th_1 A^* v$.
Using Lemmas \ref{lem:ARL}, \ref{lem:EsiR} and
$E^*_r A^* = \th^*_r E^*_r$ $(0 \leq r \leq d)$,
\begin{align*}
  0 &= E^*_i (A-\th_1 I) A^* v \\
   &= E^*_i (R + L - \th_1 I) A^* v \\
  &= (R E^*_{i-1} + L E^*_{i+1} - \th_1 E^*_i) A^* v  \\
  &= \th^*_{i-1} R v_{i-1} + \th^*_{i+1} L v_{i+1} - \th_1 \th^*_i v_i.
\end{align*}
The result follows.

(i), (iii)
Similar to the proof of (ii) above.
\end{proof}

\begin{lemma}   \label{lem:RL3}   \samepage
\ifDRAFT {\rm lem:RL3}. \fi
Assume that $\Phi$ is nontrivial.
Then with reference to Definition \ref{def:vi},
\begin{align*}
R v_{i-1} &= \frac{\th_1 \th^*_i - \th_0 \th^*_{i+1}}{\th^*_{i-1}-\th^*_{i+1}} \, v_i 
           \qquad (1 \leq i \leq \delta-1),
&  R v_{\delta-1} &= \th_0 v_{\delta},
\\
L v_{i+1} &= \frac{\th_1 \th^*_i - \th_0 \th^*_{i-1}}{\th^*_{i+1}-\th^*_{i-1}} \, v_i
          \qquad (1 \leq i \leq \delta -1),
&  L v_1 &= \th_0 v_0.
\end{align*}
Also
\begin{align*}
  (\th_0 \th^*_1 - \th_1 \th^*_0) v_0  &= 0,
  \qquad\qquad (\th_0 \th^*_{\delta-1} - \th_1 \th^*_\delta) v_{\delta} = 0.
\end{align*}
\end{lemma}

\begin{proof}
Solve the linear equations in Lemmas \ref{lem:RL} and \ref{lem:RL2}.
\end{proof}

\begin{lemma}  \label{lem:vibasis}   \samepage
\ifDRAFT {\rm lem:vibasis}. \fi
With reference to Definition \ref{def:vi},
the following hold:
\begin{itemize}
\item[\rm (i)]
$v_i$ is a basis for $E^*_i V  \;\; (0 \leq i \leq \delta)$;
\item[\rm (ii)]
$\{v_i\}_{i=0}^\delta$ is a basis for $V$.
\end{itemize}
\end{lemma}

\begin{proof}
Assume that $\Phi$ is nontrivial; otherwise the result holds by Lemma \ref{lem:trivial}.
The sum $V = \sum_{i=0}^\delta E^*_i V$ is direct.
For $0 \leq i \leq \delta$
the subspace $E^*_i V$ is nonzero and contains $v_i$.
So it suffices to show that the vectors $\{v_i\}_{i=0}^\delta$ span $V$.
Let $W$ denote the subspace of $V$ spanned by $\{v_i\}_{i=0}^\delta$.
Note that $W \neq 0$ since $0 \neq v \in W$.
We have $A^* W \subseteq W$ by construction.
Using Lemma \ref{lem:RL3}
we find that  $RW \subseteq W$ and $LW \subseteq W$.
So $A W \subseteq W$ in view of Lemma \ref{lem:ARL}.
Now $W = V$ by Definition \ref{def:TBTDpair}(iv).
Therefore $\{v_i\}_{i=0}^\delta$ span $V$.
The result follows.
\end{proof}

\begin{corollary}   \label{cor:coincide}  \samepage
\ifDRAFT {\rm cor:coincide}. \fi
We have $d = \delta$ and
\begin{align*}
  \dim E_i V &=1,  \qquad\qquad  \dim E^*_i V = 1 && (0 \leq i \leq d).
\end{align*}
Moreover $\dim V = d+1$.
\end{corollary}

\begin{proof}
By Lemma \ref{lem:vibasis}, $\dim E^*_i V = 1$ $(0 \leq i \leq \delta)$ and $\dim V = \delta + 1$.
Applying Lemma \ref{lem:vibasis} to $\Phi^*$,
we obtain $\dim E_i V = 1$ $(0 \leq i \leq d)$ and $\dim V = d+1$.
The result follows.
\end{proof}

\begin{note}
A tridiagonal pair $A,A^*$  is often called a Leonard pair
if all the eigenspaces of $A$ and $A^*$ have dimension one.
\end{note}

\begin{definition}    \label{def:diameter}   \samepage
\ifDRAFT {\rm def:diameter}. \fi
Recall from Corollary \ref{cor:coincide} that $d=\delta$.
We call this common value the {\em diameter} of $\Phi$.
\end{definition}

\begin{definition}   \label{def:Phistandard}    \samepage
\ifDRAFT {\rm def:Phistandard}. \fi
A basis $\{v_i\}_{i=0}^d$ for $V$ is said to be {\em $\Phi$-standard} 
whenever there exists a nonzero $v \in E_0 V$
such that $v_i = E^*_i v$ for $0 \leq i \leq d$.
\end{definition}

A $\Phi$-standard basis is characterized as follows.

\begin{lemma}   \label{lem:standard}     \samepage
\ifDRAFT {\rm lem:standard}. \fi
Given vectors $\{v_i\}_{i=0}^d$ in $V$, not all zero.
Then the following {\rm (i), (ii)} are equivalent:
\begin{itemize}
\item[\rm (i)]
$v_i \in E^*_i V$ for $0 \leq i \leq d$, and $\sum_{i=0}^d v_i  \in E_0 V$;
\item[\rm (ii)]
$\{v_i\}_{i=0}^d$ is a $\Phi$-standard basis for $V$.
\end{itemize}
\end{lemma}

\begin{proof}
${\rm (i)} \Rightarrow {\rm (ii)}$
Define $v = \sum_{j=0}^d v_j$.
By construction $v \in E_0 V$.
For $0 \leq i \leq d$, 
in the equation $v = \sum_{j=0}^d v_j$
apply $E^*_i$ to each side to obtain $E^*_i v = v_i$.
Note that $v \neq 0$ since the vectors $\{v_i\}_{i=0}^d$ are not all zero.
The vectors $\{v_i\}_{i=0}^d$ form a basis for $V$ by Lemma \ref{lem:vibasis}.
This basis is  $\Phi$-standard by Definition \ref{def:Phistandard}.

(ii) $\Rightarrow$ (i)
By Definition \ref{def:Phistandard} 
there exists $0 \neq v \in E_0 V$ such that $v_i = E^*_i v$ for $0 \leq i \leq d$.
By construction $v_i \in E^*_i V$ for $0 \leq i \leq d$.
Also $\sum_{i=0}^d v_i = v \in E_0 V$.
\end{proof}

We mention a lemma for later use.

\begin{lemma}    \label{lem:th0thd}    \samepage
\ifDRAFT {\rm lem:th0thd}. \fi
Assume that $\Phi$ is nontrivial. Then the following {\rm (i)--(iii)} hold.
\begin{itemize}
\item[\rm (i)]
Each of $\th_0$, $\th_d$, $\th^*_0$, $\th^*_d$ is nonzero.
\item[\rm (ii)]
$\displaystyle
  \th_1 /\th_0 = \th_{d-1} /\th_d = \th^*_1 / \th^*_0 = \th^*_{d-1} / \th^*_d$.
\item[\rm (iii)]
$\th_0 \th_{d-1} = \th_1 \th_d\;$ and $\; \th^*_0 \th^*_{d-1} = \th^*_1 \th^*_d$. 
\end{itemize}
\end{lemma}

\begin{proof}
By the last assertion of Lemma \ref{lem:RL3} and since 
$d=\delta$, $v_0 \neq 0$, $v_d \neq 0$,
\begin{align}
   \th_0 \th^*_1 &=  \th_1 \th^*_0,                   \label{eq:th0thdaux1}
\\
  \th_0 \th^*_{d-1} &= \th_1 \th^*_d.     \label{eq:th0thdaux2}
\end{align}
Applying  \eqref{eq:th0thdaux2} to $\Phi^*$,
\begin{equation}
   \th_{d-1} \th^*_0  = \th_d \th^*_1.            \label{eq:th0thdaux3}
\end{equation}
Suppose that $\th_0=0$.
Then $\th_1 \neq 0$ since $\{\th_i\}_{i=0}^d$ are mutually distinct.
Now $\th^*_0 =0$ by \eqref{eq:th0thdaux1} and
$\th^*_d = 0$ by \eqref{eq:th0thdaux2}.
This is a contradiction since $\{\th^*_i\}_{i=0}^d$ are mutually distinct
and $d \geq 1$.
Therefore $\th_0 \neq 0$.
Applying this to $\Phi^*$, $\Phi^{\downarrow}$, $\Phi^{\Downarrow}$,
we obtain $\th_d \neq 0$, $\th^*_0 \neq 0$, $\th^*_d \neq 0$.
We have obtained assertion (i).
Assertion (ii) follows in view of \eqref{eq:th0thdaux1}--\eqref{eq:th0thdaux3}.
Assertion (iii) follows from assertion (ii).
\end{proof}

\section{The intersection numbers}
\label{sec:intersection}

In this section we discuss the intersection numbers of a TB tridiagonal system.
Let $V$ denote a vector space over $\F$ with finite positive dimension,
and let
\[
   \Phi = (A; \{E_i\}_{i=0}^d; A^*; \{E^*_i\}_{i=0}^d)
\]
denote a TB tridiagonal system on $V$.
Let $(\{\th_i\}_{i=0}^d; \{\th^*_i\}_{i=0}^d)$ denote the eigenvalue array of $\Phi$.
Let $\{v_i\}_{i=0}^d$ denote a $\Phi$-standard basis for $V$.
For $X \in \text{End}(V)$
let $X^\natural$ denote the matrix  in $\Mat_{d+1}(\F)$ that represents $X$
with respect to the basis $\{v_i\}_{i=0}^d$.
By construction,
\begin{align}
  (E^*_i)^\natural &= \text{diag} (0,\, \ldots, 0,\, \stackrel{i}{1},\, 0,\, \ldots,\, 0)   
                        \qquad\qquad  (0 \leq i \leq d),                           \label{eq:Esinat}
\\
  (A^*)^\natural &= \text{diag} (\th^*_0,\, \th^*_1, \,\ldots, \,\th^*_d).   \label{eq:Asnat}
\end{align}
Also by construction there exist scalars  $\{c_i\}_{i=1}^d$, $\{b_i\}_{i=0}^{d-1}$ in $\F$ such that
\begin{align}
 A^\natural &=
  \begin{pmatrix}
    0 & b_0 &    & & & \text{\bf 0}                  \\
    c_1 & 0 & b_1   \\
         & c_2  & \cdot & \cdot  \\
         &      & \cdot & \cdot & \cdot \\
         &       &         & \cdot & \cdot & b_{d-1} \\
     \text{\bf 0}   &        &          &         & c_d & 0   \\
  \end{pmatrix}.                                                                          \label{eq:Anat}
\end{align}
By Definition \ref{def:RL},
\begin{align}
R^\natural &=
  \begin{pmatrix}
    0 &  &    & & & \text{\bf 0}                  \\
    c_1 & 0 &    \\
         & c_2  & \cdot &   \\
         &      & \cdot & \cdot &  \\
         &       &         & \cdot & \cdot &  \\
     \text{\bf 0}   &        &          &         & c_d & 0   \\
  \end{pmatrix},
&
 L^\natural &=
  \begin{pmatrix}
    0 & b_0 &    & & & \text{\bf 0}                  \\
       & 0 & b_1   \\
         &    & \cdot & \cdot  \\
         &      &    & \cdot & \cdot \\
         &       &         &    & \cdot & b_{d-1} \\
     \text{\bf 0}   &        &          &         &  & 0   \\
  \end{pmatrix}.                                                                          \label{eq:RLnat}
\end{align}
The scalars $\{c_i\}_{i=1}^d$, $\{b_i\}_{i=0}^{d-1}$ are called the
{\em intersection numbers of $\Phi$}.
The intersection numbers $\{c^*_i\}_{i=1}^d$, $\{b^*_i\}_{i=0}^{d-1}$ of $\Phi^*$
are called the {\em dual intersection numbers of $\Phi$}. 

\begin{lemma}    \label{lem:Avi}    \samepage
\ifDRAFT {\rm lem:Avi}. \fi
Assume that $\Phi$ is nontrivial. Then
\begin{align*}
 A v_0 &= c_1 v_1, \\
 A v_i &= b_{i-1} v_{i-1} + c_{i+1} v_{i+1}  \qquad\qquad (1 \leq i \leq d-1), \\
 A v_d &= b_{d-1} v_{d-1}.
\end{align*}
\end{lemma}

\begin{proof}
By \eqref{eq:Anat}.
\end{proof}

\begin{lemma}   \label{lem:RLvi}   \samepage
\ifDRAFT {\rm lem:RLvi}. \fi
We have
\begin{align*}
\hspace{1cm} && R v_i &= c_{i+1} v_{i+1} \qquad (0 \leq i \leq d-1), &
R v_d &= 0, & \hspace{1cm}
\\
\hspace{1cm} && L v_i &= b_{i-1} v_{i-1} \qquad (1 \leq i \leq d), &
L v_0 &= 0.  &   \hspace{1cm}
\end{align*}
\end{lemma}

\begin{proof}
By \eqref{eq:RLnat}.
\end{proof} 

\begin{lemma}    \label{lem:cibinonzero}    \samepage
\ifDRAFT {\rm lem:cibinonzero}. \fi
The scalars $\{b_i\}_{i=0}^{d-1}$, $\{c_i \}_{i=1}^d$ are all nonzero.
\end{lemma}

\begin{proof}
We first show that $c_i \neq 0$ for $1 \leq i \leq d$.
Let $i$ be given.
We have $E^*_i A E^*_{i-1} \neq 0$ by Lemma \ref{lem:EsiAEsj}.
By \eqref{eq:Esinat} and \eqref{eq:Anat},
the matrix $(E^*_i  A E^*_{i-1})^\natural$ has 
$(i,i-1)$-entry $c_i$ and all other entries $0$.
By these comments $c_i \neq 0$.
One similarly shows that $b_i \neq 0$ for $0 \leq i \leq d-1$.
\end{proof}

\begin{lemma}   \label{lem:cibi}   \samepage
\ifDRAFT {\rm lem:cibi}. \fi
Assume that $\Phi$ is nontrivial. Then  the following hold.
\begin{itemize}
\item[\rm (i)]
The intersection numbers of $\Phi$ satisfy
\begin{align}
c_i &= \frac{\th_1 \th^*_i - \th_0 \th^*_{i+1}}{\th^*_{i-1} - \th^*_{i+1}}  \qquad (1 \leq i \leq d-1),
  &  c_d &= \th_0,                   \label{eq:ci}
\\
b_i &= \frac{\th_1 \th^*_i - \th_0 \th^*_{i-1}}{\th^*_{i+1} - \th^*_{i-1}} \qquad (1 \leq i \leq d-1),
  &  b_0 &= \th_0.                  \label{eq:bi}
\end{align}
\item[\rm (ii)]
The dual intersection numbers of $\Phi$ satisfy
\begin{align}
c^*_i &= \frac{\th^*_1 \th_i - \th^*_0 \th_{i+1}}{\th_{i-1} - \th_{i+1}}  \qquad (1 \leq i \leq d-1),
  &  c^*_d &= \th^*_0,                     \label{eq:csi}
\\
b^*_i &= \frac{\th^*_1 \th_i - \th^*_0 \th_{i-1}}{\th_{i+1} - \th_{i-1}} \qquad (1 \leq i \leq d-1),
  &  b^*_0 &= \th^*_0.                  \label{eq:bsi}
\end{align}
\end{itemize}
\end{lemma}

\begin{proof}
(i)
By Lemmas \ref{lem:RL3} and \ref{lem:RLvi}.

(ii)
Apply (i) above to $\Phi^*$.
\end{proof}

\begin{corollary}   \label{cor:unique}  \samepage
\ifDRAFT {\rm cor:unique}. \fi
A TB tridiagonal system is uniquely determined up to isomorphism by its
eigenvalue array.
\end{corollary}

\begin{proof}
By \eqref{eq:Asnat}, \eqref{eq:Anat} and Lemma \ref{lem:cibi}(i),
the entries of $A^\natural$ and $(A^*)^\natural$ are determined by 
the eigenvalue array.
\end{proof}

\begin{lemma}    \label{lem:hAhsAs}    \samepage
\ifDRAFT {\rm lem:hAhsAs}. \fi
For nonzero $\zeta$, $\zeta^* \in \F$, consider the TB tridiagonal system
\[
    \check{\Phi} =  ( \zeta A; \{E_i\}_{i=0}^d; \zeta^* A^*; \{E^*_i\}_{i=0}^d).
\]
Then the (dual) intersection numbers of $\check{\Phi}$ are
\begin{align*}
  \check{c}_i &= \zeta c_i,   \qquad\qquad  \check{c}^{\,*}_i = \zeta^* c^*_i   && (1 \leq i \leq d), 
\\
 \check{b}_i & = \zeta b_i,   \qquad\qquad     \check{b}^{\, *}_i = \zeta^* b^*_i    && (0 \leq i \leq d-1).
\end{align*}
\end{lemma}

\begin{proof}
Use Lemmas \ref{lem:hhs} and \ref{lem:cibi}.
\end{proof}

\section{Tridiagonal matrices}
\label{sec:trid}

In this section we collect some results about tridiagonal matrices
that will be used later in the paper.
Consider a tridiagonal matrix in $\Mat_{d+1}(\F)$:
\begin{align}                                \label{eq:defA}
 A &=
  \begin{pmatrix}
    a_0 & b_0 &    & & & \text{\bf 0}                  \\
    c_1 & a_1 & b_1   \\
         & c_2  & \cdot & \cdot  \\
         &      & \cdot & \cdot & \cdot \\
         &       &         & \cdot & \cdot & b_{d-1} \\
     \text{\bf 0}   &        &          &         & c_d & a_d   \\
  \end{pmatrix}.
\end{align}
We say that $A$ is {\em irreducible} whenever $c_i b_{i-1} \neq 0$ for $1 \leq i \leq d$.
For the rest of this section, assume that $A$ is irreducible.

We have some remarks about the minimal polynomial of $A$.
For $0 \leq r \leq d$ consider the matrix $A^r$.
For $0 \leq i,j \leq d$ the $(i,j)$-entry of $A^r$ satisfies
\begin{align}                            \label{eq:Arij}
  (A^r)_{i,j} &=
  \begin{cases}
     0  & \text{ if  $\;\;|i-j| > r$},
  \\
  \neq 0 & \text{ if $\;\;|i-j|=r$}.
  \end{cases}
\end{align}
Consequently the matrices $\{A^r\}_{r=0}^d$ are linearly independent.
Therefore the minimal polynomial of $A$ coincides with the characteristic polynomial of $A$.
So each eigenspace of $A$ has dimension one.
We remark that $A$ might not be diagonalizable.

For $0 \leq i \leq d$ define $E^*_i \in \Mat_{d+1}(\F)$ by
\begin{equation}
   E^*_i = \text{diag} ( 0,\, \ldots, \, 0, \, \stackrel{i}{1},\, 0, \, \ldots, \, 0).  \label{eq:defEsi}
\end{equation}
Observe that 
$E^*_i E^*_j = \delta_{i,j} E^*_i$ $(0 \leq i,j \leq d)$ and
$I = \sum_{i=0}^d E^*_i$.

\begin{lemma}    \label{lem:EsiArEsj}    \samepage
\ifDRAFT {\rm lem:EsiArEsj}. \fi
For $0 \leq i,j,r \leq d$,
\begin{align}                                 \label{eq:EsiArEsj}
  E^*_i A^r E^*_j &=
    \begin{cases}
      0 & \text{ if $\;\;|i-j|>r$}, 
    \\
     \neq 0 & \text{ if $\;\;|i-j|=r$}.
   \end{cases}
\end{align}
\end{lemma}

\begin{proof}
This is a reformulation of \eqref{eq:Arij}.
\end{proof}

\begin{lemma}   \label{lem:ArEs0As}    \samepage
\ifDRAFT {\rm lem:ArEs0As}. \fi
The elements 
\begin{equation}                             \label{eq:ArEs0As}
\{A^r E^*_0 A^s \,|\, 0 \leq r,s \leq d\}
\end{equation}
form a basis for the $\F$-vector space $\Mat_{d+1}(\F)$.
\end{lemma}

\begin{proof}
For $0 \leq r,s \leq d$ consider the matrix $A^r E^*_0 A^s$.
For $0 \leq i,j \leq d$ we compute its $(i,j)$-entry using \eqref{eq:defEsi},
and evaluate the result using \eqref{eq:Arij}:
\[
  (A^r E^*_0 A^s)_{i,j} = (A^r)_{i,0} (A^s)_{0,j} 
 = \begin{cases}
      0 & \text{ if $\;\; i>r\;$ or $\;j>s$},
   \\
     \neq 0 & \text{ if $\;\;i=r\;$ and $\;j=s$}.
  \end{cases}
\]
From the pattern of zero/nonzero entries, we see that
the elements \eqref{eq:ArEs0As} are linearly independent.
The set \eqref{eq:ArEs0As} contains $(d+1)^2$ elements,
and this is the dimension of $\Mat_{d+1}(\F)$.
The result follows.
\end{proof}

\begin{corollary}   \label{cor:AEs0}   \samepage
\ifDRAFT {\rm cor:AEs0}. \fi
The elements  $A$, $E^*_0$ generate the $\F$-algebra $\Mat_{d+1}(\F)$.
\end{corollary}

\begin{proof}
By Lemma \ref{lem:ArEs0As}.
\end{proof}

\begin{definition}            \label{def:As}
Let $\{\th^*_i\}_{i=0}^d$ denote scalars in $\F$.
Define $A^* \in \Mat_{d+1}(\F)$ by
\begin{equation}                    \label{eq:defAs}
    A^* = \text{diag}(\th^*_0,\, \th^*_1,\, \ldots, \, \th^*_d).
\end{equation}
\end{definition}

Observe that
\begin{equation}            \label{eq:As}
 A^* = \sum_{i=0}^d \th^*_i E^*_i.
\end{equation}

\begin{lemma}   \label{lem:As}   \samepage
\ifDRAFT {\rm lem:As}. \fi
Assume that $\th^*_0 \neq \th^*_i$ for $1 \leq i \leq d$.
Then the following {\rm (i)--(iii)} hold.
\begin{itemize}
\item[\rm (i)]
The element $E^*_0$ is a polynomial in $A^*$:
\[
   E^*_0 = \prod_{i=1}^d \frac{A^* - \th^*_i I}{\th^*_0 - \th^*_i}.
\]
\item[\rm (ii)]
The elements $A,A^*$ generate the $\F$-algebra $\Mat_{d+1}(\F)$.
\item[\rm (iii)]
There does not exist a subspace $W \subseteq \F^{d+1}$ such that
$W \neq 0$, $W\neq \F^{d+1}$, $AW \subseteq W$, $A^* W \subseteq W$.
\end{itemize}
\end{lemma}

\begin{proof}
(i)
By matrix multiplication.

(ii)
By (i) and Corollary \ref{cor:AEs0}.

(iii)
By (ii) above and since $\F^{d+1}$ is an irreducible $\Mat_{d+1}(\F)$-module.
\end{proof}

\begin{definition}       \label{def:ki}
\ifDRAFT {\rm def:ki}. \fi
Define scalars $\{k_i\}_{i=0}^d$ by
\begin{align}                          \label{eq:defki}
   k_i &= \frac{b_0 b_1 \cdots b_{i-1}}{c_1 c_2 \cdots c_i}   \qquad\qquad (0 \leq i \leq d).
\end{align}
\end{definition}

Observe that $k_0=1$,  and $k_i \neq 0$ for $0 \leq i \leq d$.
Moreover
\begin{align}
   k_{i-1} b_{i-1} &= k_i c_i  \qquad\qquad (1 \leq i \leq d).      \label{eq:kici}
\end{align}

\begin{definition}    \label{def:K}    \samepage
\ifDRAFT {\rm def:K}. \fi
Define $K \in \Mat_{d+1}(\F)$ by
\begin{equation}                          \label{eq:defK}
   K = \text{diag}(k_0,\, k_1, \, \ldots \, , \, k_d).
\end{equation}
\end{definition}

\begin{lemma}   \label{lem:AtK}   \samepage
\ifDRAFT {\rm lem:AtK}. \fi
We have $A^{\rm t} K = K A$.
\end{lemma}

\begin{proof}
Use \eqref{eq:defA} and \eqref{eq:kici}.
\end{proof}

For an $\F$-algebra $\cal A$,
by an {\em antiautomorphism} of $\cal A$ we mean an $\F$-linear
bijection $\xi : {\cal A} \to {\cal A}$ such that
$(xy)^\xi = y^\xi x^\xi$ for $x,y \in {\cal A}$.

\begin{definition}    \label{def:dagger}   \samepage
\ifDRAFT {\rm def:dagger}. \fi
Define the map
\[
 \dagger : \Mat_{d+1}(\F) \to \Mat_{d+1}(\F), \qquad X \mapsto K^{-1} X^{\rm t} K.
\]
\end{definition}

\begin{lemma}    \label{lem:dagger}   \samepage
\ifDRAFT {\rm lem:dagger}. \fi
The map $\dagger$ is an antiautomorphism of $\Mat_{d+1}(\F)$.
Moreover $(X^\dagger)^\dagger = X$ for all $X \in \Mat_{d+1}(\F)$.
\end{lemma}

\begin{proof}
The map $\dagger$ is $\F$-linear.
For $X \in \Mat_{d+1}(\F)$ we have $(X^\dagger)^\dagger=X$, since
\[
 (X^\dagger)^\dagger = (K^{-1} X^\text{t} K)^\dagger
 = K^{-1} (K^{-1} X^\text{t} K)^\text{t} K = K^{-1} K X K^{-1} K = X.
\]
Consequently the map $\dagger$ is bijective.
For $X$, $Y \in \Mat_{d+1}(\F)$ we have
\[
 (XY)^\dagger = K^{-1} (XY)^\text{\rm t} K
  = K^{-1} Y^\text{t} X^\text{t} K
  = K^{-1} Y^\text{t} K K^{-1} X^\text{t} K
 = Y^\dagger X^\dagger.
\]
Thus $\dagger$ is an antiautomorphism.
\end{proof}

The map $\dagger$ is characterized as follows.

\begin{lemma}    \label{lem:daggerunique}   \samepage
\ifDRAFT {\rm lem:daggerunique}. \fi
The map $\dagger$ is the unique antiautomorphism of $\Mat_{d+1}(\F)$
that fixes each of $A$, $E^*_0$, $E^*_1$, \ldots, $E^*_d$.
\end{lemma}

\begin{proof}
Using Lemma \ref{lem:AtK} and Definition \ref{def:dagger},
\[
   A^\dagger = K^{-1} A^\text{t} K = K^{-1} K A = A,
\]
so $\dagger$ fixes $A$.
For $0 \leq i \leq d$ the map $\dagger$ fixes $E^*_i$ by Definition \ref{def:dagger}
and since $E^*_i$, $K$ are diagonal.
Concerning the the uniqueness,
let $\xi$ denote an antiautomorphism of $\Mat_{d+1}(\F)$ that 
fixes each of  $A$, $E^*_0$, $E^*_1$, \ldots, $E^*_d$.
We show that $\xi = \dagger$.
The composition $\xi \dagger$ is an automorphism of $\Mat_{d+1}(\F)$
that fixes $A$, $E^*_0$, $E^*_1$, \ldots, $E^*_d$.
Consequently $\xi \dagger = 1$ in view of Corollary \ref{cor:AEs0}.
So $\xi = \dagger$.
\end{proof}

\section{Recurrent sequences}
\label{sec:rec}

Let $(\{\th_i\}_{i=0}^d; \{\th^*_i\}_{i=0}^d)$ denote the eigenvalue array of 
a TB tridiagonal system over $\F$.
Later in the paper we will show that there exists $\beta \in \F$ such that
\begin{align*}
 \th_{i-1} - \beta \th_i + \th_{i+1} &=0, &
 \th^*_{i-1} - \beta \th^*_i + \th^*_{i+1} &= 0
 &&  (1 \leq i \leq d-1).
\end{align*}
In this section we have some comments about the above recurrence.
For the rest of this section, fix an integer $d \geq 0$ and
let $\{\sigma_i\}_{i=0}^d$ denote a sequence 
of scalars taken from $\F$.

\begin{definition}     \label{def:recurrent}    \samepage
\ifDRAFT {\rm def:recurrent}. \fi
For $\beta \in \F$, the sequence $\{\sigma_i\}_{i=0}^d$ is said to be {\em $\beta$-recurrent}
whenever 
\begin{align}                \label{eq:rec}  
  \sigma_{i-1} - \beta \sigma_i + \sigma_{i+1} &= 0  \qquad\qquad  (1 \leq i \leq d-1).
\end{align}
We say that $\{\sigma_i\}_{i=0}^d$ is {\em recurrent} whenever it is $\beta$-recurrent
for some $\beta \in \F$.
\end{definition}

\begin{lemma}   \label{lem:rec1}   \samepage
\ifDRAFT {\rm lem:rec1}. \fi
For $\beta \in \F$ the following hold.
\begin{itemize}
\item[\rm (i)]
Assume that $\{\sigma_i\}_{i=0}^d$ is $\beta$-recurrent.
Then there exists $\varrho \in \F$ such that
\begin{align}
 \varrho &= \sigma_{i-1}^2 - \beta \sigma_{i-1} \sigma_i + \sigma_i^2
                            \qquad\qquad (1 \leq i \leq d).           \label{eq:rho2}
\end{align}
\item[\rm (ii)]
Assume that there exists $\varrho \in \F$ that satisfies \eqref{eq:rho2}.
Also assume that $\sigma_{i-1} \neq \sigma_{i+1}$ for $1 \leq i \leq d-1$.
Then $\{\sigma_i\}_{i=0}^d$ is $\beta$-recurrent.
\end{itemize}
\end{lemma}

\begin{proof}
For $1 \leq i \leq d$ define
\[
  S_i = \sigma_{i-1}^2 - \beta \sigma_{i-1} \sigma_i + \sigma_i^2. 
\]
Observe that for $1 \leq i \leq d-1$,
\begin{align}
  S_i - S_{i+1} &= (\sigma_{i-1}- \sigma_{i+1})(\sigma_{i-1} - \beta \sigma_i + \sigma_{i+1}).
                      \label{eq:rec1aux}
\end{align}

(i)
In \eqref{eq:rec1aux} the right-hand side is zero by \eqref{eq:rec}.
So $S_i - S_{i+1}=0$.
Thus $S_i$ is independent of $i$ for $1 \leq i \leq d$.

(ii)
In \eqref{eq:rec1aux} the left-hand side is zero.
In the right-hand side, the first factor is nonzero so
the second factor is zero.
Consequently $\{\sigma_i\}_{i=0}^d$ is $\beta$-recurrent.
\end{proof}

\section{The Askey-Wilson relations}
\label{sec:AW}

Let $A,A^*$ denote a TB tridiagonal pair over $\F$.
In this section we show that there exist scalars $\beta$, $\varrho$, $\varrho^*$ in $\F$ such that
both
\begin{align}
 A^2 A^* - \beta A A^* A + A^* A^2 &= \varrho A^*,        \label{eq:AW1}
\\
 {A^*}^2 A - \beta A^* A A^* + A {A^*}^2 &= \varrho^* A.   \label{eq:AW2}
\end{align}
The equations \eqref{eq:AW1}, \eqref{eq:AW2} are  special cases of the
Askey-Wilson relations \cite{TV, Z}.

For the rest of this section, fix an integer $d \geq 0$,
let $V$ denote a vector space over $\F$
with dimension $d+1$,
and let $\Phi = (A; \{E_i\}_{i=0}^d; A^*; \{E^*_i\}_{i=0}^d)$
denote a TB tridiagonal system on $V$ with eigenvalue array
$(\{\th_i\}_{i=0}^d; \{\th^*_i\}_{i=0}^d)$.

First we explain how \eqref{eq:AW1}, \eqref{eq:AW2} are related to the
recurrence discussed in Lemma \ref{lem:rec1}.

\begin{lemma}   \label{lem:rec4}    \samepage
\ifDRAFT {\rm lem:rec4}. \fi
For $\beta$, $\varrho \in \F$ the following {\rm (i), (ii)} are equivalent:
\begin{itemize}
\item[\rm (i)]
$\th_{i-1}^2 - \beta \th_{i-1} \th_i + \th_i^2 = \varrho$ \quad $(1 \leq i \leq d)$;
\item[\rm (ii)]
$A^2 A^* - \beta A A^* A + A^* A^2 = \varrho A^*$.
\end{itemize}
\end{lemma}

\begin{proof}
Define $D = A^2 A^* - \beta A A^* A + A^* A^2 - \varrho A^*$.
Using $I= \sum_{i=0}^d E_i$ we obtain
\[
  D = I D I = \sum_{0 \leq i,j \leq d} E_i D E_j.
\]
Thus $D=0$ if and only if $E_i D E_j=0$  $(0 \leq i,j \leq d)$.
Pick integers $i$, $j$ such that $0 \leq i,j \leq d$.
Using $E_i A = \th_i E_i$ and $A E_j = \th_j E_j$ we find that
\[
  E_i D E_j = E_i A^* E_j (\th_i^2 - \beta \th_i \th_j + \th_j^2 - \varrho).
\]
By \eqref{eq:EiAsEj}, $E_i A^* E_j \neq 0$ if and only if $|i-j| =1$.
The result follows from these comments.
\end{proof}

Our next goal is to prove \eqref{eq:AW1}, \eqref{eq:AW2};
this is accomplished in Proposition \ref{prop:AW}.

\begin{lemma}   \label{lem:comments}   \samepage
\ifDRAFT {\rm lem:comments}. \fi
The following hold.
\begin{itemize}
\item[\rm (i)]
The elements $A$, $A^*$ generate  $\text{\rm End}(V)$.
\item[\rm (ii)]
For $0 \leq i,j,r \leq d$,
\[
   E^*_i A^r E^*_j =
    \begin{cases}
      0 & \text{ if $\;\;|i-j|>r$},
   \\
     \neq 0 & \text{ if $\;\;|i-j|=r$}.
   \end{cases}
\]
\end{itemize}
\end{lemma}

\begin{proof}
Let $\{v_i\}_{i=0}^d$ denote a $\Phi$-standard basis for $V$.
For $X \in \text{End}(V)$ let $X^\natural$  denote the matrix in $\Mat_{d+1}(\F)$
that represents $X$ with respect to $\{v_i\}_{i=0}^d$.
The matrices $(E^*_i)^\natural$,
$(A^*)^\natural$, $A^\natural$ are given in \eqref{eq:Esinat}--\eqref{eq:Anat}.
Now (i) follows from
Lemma  \ref{lem:As}(ii)
and
(ii) follows from Lemma \ref{lem:EsiArEsj}.
\end{proof}

\begin{lemma}   \label{lem:EsAEsAEs}    \samepage
\ifDRAFT {\rm lem:EsAEsAEs}. \fi
For $0 \leq i,j,k,r,s \leq d$ consider the expression
\begin{equation}
 E^*_i A^r E^*_k A^s E^*_j.           \label{eq:EsiArEskAsEsj}
\end{equation}
\begin{itemize}
\item[\rm (i)]
Assume that $|i-j| > r+s$.
Then \eqref{eq:EsiArEskAsEsj} is zero.
\item[\rm (ii)]
Assume that $i-j = r+s$.
Then \eqref{eq:EsiArEskAsEsj} is nonzero if and only if $k = j+s$.
\item[\rm (iii)]
Assume that $j-i = r+s$.
Then \eqref{eq:EsiArEskAsEsj} is nonzero if and only if $k = i+r$.
\end{itemize}
\end{lemma}

\begin{proof}
Routine verification using Lemma \ref{lem:comments}(ii).
\end{proof}

\begin{lemma}   \label{lem:EsAAsAEs}    \samepage
\ifDRAFT {\rm lem:EsAAsAEs}. \fi
For $0 \leq i,j,r,s \leq d$,
\begin{equation}                   \label{eq:EsAAsAEs}
E^*_i A^r A^* A^s E^*_j =
 \begin{cases}
  0 & \text{ if $\;\;|i-j| > r+s$},
  \\
  \th^*_{j+s} E^*_i A^{r+s} E^*_j & \text{ if $\;\;i-j=r+s$},  
 \\
  \th^*_{i+r} E^*_i A^{r+s} E^*_j & \text{ if $\;\;j-i=r+s$}.
 \end{cases}
\end{equation}
\end{lemma}

\begin{proof}
Using  $A^* = \sum_{k=0}^d \th^*_k E^*_k$,
\[
 E^*_i A^r A^* A^s E^*_j   = \sum_{k=0}^d \th^*_k E^*_i A^r E^*_k A^s E^*_j.
\]
Using $I = \sum_{k=0}^d E^*_k$,
\[
  E^*_i A^{r+s} E^*_j = E^*_i A^r I A^s E^*_j
   = \sum_{k=0}^d E^*_i A^r E^*_k A^s E^*_j.
\]
By these comments and Lemma \ref{lem:EsAEsAEs} we obtain the result.
\end{proof}

\begin{definition}    \label{def:M}   \samepage
\ifDRAFT {\rm def:M}. \fi
Let $M$ (resp.\ $M^*$) denote the subalgebra of $\text{End}(V)$
generated by $A$ (resp.\ $A^*$).
Recall that each of $\{A^i\}_{i=0}^d$ and $\{E_i\}_{i=0}^d$ 
(resp.\ $\{{A^*}^i\}_{i=0}^d$ and $\{E^*_i\}_{i=0}^d$)
is a basis for the $\F$-vector space $M$ (resp.\ $M^*$).
\end{definition}

\begin{lemma}    \label{lem:comments2}    \samepage
\ifDRAFT {\rm lem:comments.2}. \fi
There exists a unique antiautomorphism $\dagger$ of $\text{\rm End}(V)$ that
fixes each of $A$ and $A^*$.
The map $\dagger$ fixes each element of $M$ and each element of $M^*$.
In particular, $\dagger$ fixes $E_i$ and $E^*_i$ for $0 \leq i \leq d$.
Moreover $(X^\dagger)^\dagger = X$ for all $X \in \text{\rm End}(V)$.
\end{lemma}

\begin{proof}
Follows from Lemmas \ref{lem:dagger} and \ref{lem:daggerunique}.
\end{proof}

\begin{definition}    \label{def:MAsM}    \samepage
\ifDRAFT {\rm def:MAsM}. \fi
Define a subspace $M A^* M$ of $\text{End}(V)$ by
\[
  M A^* M = \text{Span} \{ X A^* Y \,|\, X,Y \in M \}.
\]
\end{definition}

\begin{lemma}   \label{lem:MAsM}   \samepage
\ifDRAFT {\rm lem:MAsM}. \fi
The following elements form a basis for the $\F$-vector space $M A^* M$:
\begin{equation}
  \{ E_{i-1}A^* E_i, \; E_i A^* E_{i-1} \,|\,  1 \leq i \leq d\}.            \label{eq:basisEiAsEj}
\end{equation}
Moreover $\;\dim (M A^* M ) = 2d$.
\end{lemma}

\begin{proof}
We assume $d \geq 1$; otherwise the result follows from Lemma \ref{lem:trivial}.
The $\F$-vector space $M A^* M$ is spanned by
$\{ E_i A^* E_j \,|\, 0 \leq i,j \leq d\}$.
By \eqref{eq:EiAsEj} we have $E_i A^* E_j = 0$ if $|i-j| \neq 1$  $(0 \leq i,j \leq d)$.
By these comments
the elements \eqref{eq:basisEiAsEj} span $M A^* M$.
We show that the elements \eqref{eq:basisEiAsEj} are linearly independent.
Suppose that there exist scalars  $\{\alpha_i\}_{i=1}^d$, $\{\beta_i\}_{i=1}^d$ in $\F$
such that
\begin{equation}     \label{eq:MAsMaux1}
 0 = \sum_{i=1}^d \alpha_i E_{i-1} A^* E_i + \sum_{i=1}^d \beta_i E_i A^* E_{i-1}.
\end{equation}
Pick any integer $j$ $(1 \leq j \leq d)$.
We show that $\alpha_j = 0$.
In \eqref{eq:MAsMaux1}, multiply each side on the left by $E_{j-1}$ and right by $E_j$.
This gives $\alpha_j E_{j-1} A^* E_j = 0$.
This forces $\alpha_j = 0$ since $E_{j-1} A^* E_j \neq 0$ by \eqref{eq:EiAsEj}.
A similar argument shows that $\beta_j=0$.
We have shown that the elements \eqref{eq:basisEiAsEj} are linearly independent.
By the above comments the elements \eqref{eq:basisEiAsEj} form a basis for 
the $\F$-vector space $M A^* M$.
The set \eqref{eq:basisEiAsEj} has cardinality $2d$ so $\dim (MA^* M) = 2d$.
\end{proof}

\begin{lemma}   \label{lem:EiAs}    \samepage
\ifDRAFT {\rm lem:EiAs}. \fi
For $0 \leq i \leq d$,
\begin{equation}
 E_i A^* = E_i A^* E_{i-1} + E_i A^* E_{i+1}.          \label{eq:EiAs}
\end{equation}
\end{lemma}

\begin{proof}
Using $I = \sum_{j=0}^d E_j$,
\[
  E_i A^* = E_i A^* I = \sum_{j=0}^d E_i A^* E_j.
\]
By \eqref{eq:EiAsEj}, $E_i A^* E_j = 0$ if $|i-j| \neq 1$ $(0 \leq j \leq d)$.
The result follows.
\end{proof}

\begin{lemma}    \label{lem:AsEi}    \samepage
\ifDRAFT {\rm lem:AsEi}. \fi
For $0 \leq i \leq d$,
\[
 A^* E_i = E_{i-1} A^* E_i + E_{i+1} A^* E_i.
\]
\end{lemma}

\begin{proof}
Apply the antiautomorphism $\dagger$ to each side of \eqref{eq:EiAs}.
\end{proof}

\begin{lemma}   \label{lem:Ei-1AsEi}    \samepage
\ifDRAFT {\rm lem:Ei-1AsEi}. \fi
For $1 \leq i \leq d$,
\begin{align*}
 E_{i-1} A^* E_i &= E_{i-1} A^* - A^* E_{i-2} + E_{i-3} A^* - A^* E_{i-4} + \cdots
\\
 E_i A^* E_{i-1} &= A^* E_{i-1} - E_{i-2} A^* + A^* E_{i-3} - E_{i-4} A^* + \cdots
\end{align*}
Moreover
\begin{align}
 \sum_{ \stackrel{0 \leq i \leq d}{\text{\rm $i$ even}}} E_i A^*
 &= \sum_{ \stackrel{0 \leq i \leq d}{\text{\rm $i$ odd}}}    A^* E_i,
&
 \sum_{  \stackrel{0 \leq i \leq d}{\text{\rm $i$ odd}}}   E_i A^*
 &=  \sum_{  \stackrel{0 \leq i \leq d}{\text{\rm $i$ even}}}   A^* E_i.
                                                        \label{eq:EiAsAsEipre}
\end{align}
\end{lemma}

\begin{proof}
Solve the equations in Lemmas \ref{lem:EiAs} and \ref{lem:AsEi}.
\end{proof}

\begin{lemma}    \label{lem:EiAsAsEipre1}    \samepage
\ifDRAFT {\rm lem:EiAsAsEipre1}. \fi
We have
 $\;\; \sum_{i=0}^d (-1)^i (E_i A^* + A^* E_i) = 0$.
\end{lemma}

\begin{proof}
Follows from \eqref{eq:EiAsAsEipre}.
\end{proof}

\begin{lemma}    \label{lem:EiAsAsEipre2}     \samepage
\ifDRAFT {\rm lem:EiAsAsEipre2}. \fi
The element $A^*$ is equal to each of the following sums:
\begin{align*}
& \sum_{ \stackrel{0 \leq i \leq d}{\text{\rm $i$ even}}}  (E_i A^* + A^* E_i),
\qquad\qquad
  \sum_{ \stackrel{0 \leq i \leq d}{\text{\rm $i$ odd}}}  (E_i A^* + A^* E_i).
\end{align*}
\end{lemma}

\begin{proof}
We have $A^* = \sum_{i=0}^d E_i A^*$.
Evaluate this sum using \eqref{eq:EiAsAsEipre}.
\end{proof}

\begin{lemma}    \label{lem:EiAsAsEi}    \samepage
\ifDRAFT {\rm lem:EiAsAsEi}. \fi
The following elements form a basis for the $\F$-vector space $MA^*M$:
\begin{equation}
  \{ E_i A^*, \, A^* E_i \,|\, 0 \leq i \leq d-1\}.         \label{eq:EiAsAsEi}
\end{equation}
\end{lemma}

\begin{proof}
Let $\cal X$ denote the subspace of $\text{End}(V)$ spanned by \eqref{eq:EiAsAsEi}.
By construction ${\cal X} \subseteq MA^*M$.
By Lemmas \ref{lem:MAsM} and \ref{lem:Ei-1AsEi},
$M A^* M \subseteq {\cal X}$.
So $M A^* M = {\cal X}$.
By Lemma \ref{lem:MAsM} the dimension of $MA^*M$ is $2d$.
The set \eqref{eq:EiAsAsEi} has cardinality $2d$.
Therefore the elements \eqref{eq:EiAsAsEi} form a basis for 
the $\F$-vector space $MA^*M$.
\end{proof}

\begin{corollary}   \label{cor:MAsM}  \samepage
\ifDRAFT {\rm cor:MAsM}. \fi
We have 
$M A^* M = M A^* + A^* M$.
\end{corollary}

\begin{proof}
By construction $M A^* + A^* M \subseteq M A^* M$.
By Lemma \ref{lem:EiAsAsEi}, $M A^* M \subseteq M A^* + A^* M$.
\end{proof}

\begin{definition}    \label{def:sym}    \samepage
\ifDRAFT {\rm def:sym}. \fi
Define a subspace $(MA^*M)^\text{sym}$ of $\text{End}(V)$ by
\[
(M A^* M)^{\text{\rm sym}} = \{ Z \in M A^* M \,|\, Z^\dagger = Z\}.
\]
We call $(M A^* M)^{\text{\rm sym}}$ the {\em symmetric part of $M A^* M$}.
\end{definition}

\begin{lemma}   \label{lem:MAsM2}    \samepage
\ifDRAFT {\rm lem:MAsM2}. \fi
The subspace  $(MA^*M)^\text{\rm sym}$ contains $X A^* X$ for all $X$ in $M$.
Moreover $(MA^*M)^\text{\rm sym}$ contains $X A^* Y + Y A^* X$ for all $X$, $Y$ in $M$.
\end{lemma}

\begin{proof}
Observe that $(X A^* Y)^\dagger = Y A^* X$ for all $X$, $Y \in M$.
\end{proof}

\begin{lemma}   \label{lem:sym}   \samepage
\ifDRAFT {\rm lem:sym}. \fi
The following elements form a basis for the $\F$-vector space $(M A^* M)^{\text{\rm sym}}$:
\begin{equation}
   \{ E_{i-1} A^* E_i + E_i A^* E_{i-1} \,|\, 1 \leq i \leq d\}.    \label{eq:sym1}
\end{equation}
Moreover, $\;\; \dim (M A^* M)^{\text{\rm sym}} =d$.
\end{lemma}

\begin{proof}
Let $\cal Z$ denote the subspace of $\text{End}(V)$ spanned by \eqref{eq:sym1}.
We show that ${\cal Z} = (M A^* M)^\text{sym}$.
By Lemma \ref{lem:MAsM2} we have ${\cal Z} \subseteq (MA^*M)^\text{sym}$.
To show the reverse inclusion, 
pick any $Z \in (M A^* M)^\text{sym}$.
By construction $Z^\dagger = Z$.
By Lemma \ref{lem:MAsM} there exist  scalars $\{\alpha_i\}_{i=1}^d$, $\{\beta_i\}_{i=1}^d$ in $\F$
such that
\[
   Z = \sum_{i=1}^d \alpha_i E_{i-1} A^* E_i + \sum_{i=1}^d \beta_i E_i A^* E_{i-1}.
\]
We have
\[
  Z^\dagger = \sum_{i=1}^d \alpha_i E_i A^* E_{i-1} + \sum_{i=1}^d \beta_i E_{i-1} A^* E_i.
\]
By $Z^\dagger = Z$ and the above two equations, we find that $\alpha_i = \beta_i$ for $1 \leq i \leq d$.
Thus
\[
   Z = \sum_{i=1}^d \alpha_i (E_{i-1}A^* E_i + E_i A^* E_{i-1}),
\]
and so $Z \in {\cal Z}$.
We have shown that ${\cal Z} = (M A^* M)^\text{sym}$.
Therefore the elements \eqref{eq:sym1} span $(MA^*M)^\text{sym}$.
Using Lemma \ref{lem:MAsM} one checks that
the elements \eqref{eq:sym1} are linearly independent.
The result follows.
\end{proof}

\begin{lemma}   \label{lem:sym2}    \samepage
\ifDRAFT {\rm lem:sym2}. \fi
The following elements form a basis for the $\F$-vector space $(M A^* M)^{\text{\rm sym}}$:
\begin{equation}
  \{ E_i A^* + A^* E_i \,|\, 0 \leq i \leq d-1\}.               \label{eq:sym2}
\end{equation}
\end{lemma}

\begin{proof}
Each of the elements in \eqref{eq:sym2} is fixed by $\dagger$, and therefore contained in $(MA^*M)^\text{\rm sym}$.
The elements \eqref{eq:sym2} are linearly independent by Lemma \ref{lem:EiAsAsEi}.
The set \eqref{eq:sym2} has cardinality $d$.
The subspace $(MA^*M)^\text{\rm sym}$ has dimension $d$ by Lemma \ref{lem:sym}.
The result follows from these comments.
\end{proof}

\begin{lemma}   \label{lem:sym3}   \samepage
\ifDRAFT {\rm lem:sym3}. \fi
We have
\begin{equation}
(M A^* M)^{\text{\rm sym}} =  \{ XA^* + A^* X \,|\, X \in M\}.  \label{eq:sym3}
\end{equation}
\end{lemma}

\begin{proof}
The inclusion $\subseteq$ follows from Lemma \ref{lem:sym2}.
The inclusion $\supseteq$ follows from Lemma \ref{lem:MAsM2}.
\end{proof}

\begin{lemma}    \label{lem:sym4}    \samepage
\ifDRAFT {\rm lem:sym4}. \fi
The $\F$-vector space $(M A^* M)^{\text{\rm sym}}$ is spanned by
\begin{equation}
 \{A^*, \; AA^*+A^*A, \; A^2 A^* + A^* A^2, \; 
          \ldots, \; A^d A^* + A^* A^d\}.                            \label{eq:symspan}
\end{equation}
\end{lemma}

\begin{proof}
Let $\cal Z$ denote the subspace of $\text{\rm End}(V)$ spanned by \eqref{eq:symspan}.
We show that ${\cal Z} = (M A^* M)^\text{\rm sym}$.
We have ${\cal Z} \subseteq (M A^* M)^\text{\rm sym}$ by Lemma \ref{lem:MAsM2}.
By Lemma \ref{lem:sym3},
\[
   (M A^* M)^\text{\rm sym} = \text{\rm Span} \, \{ A^i A^* + A^* A^i \,|\, 0 \leq i \leq d\}  \subseteq {\cal Z}.
\]
The result follows.
\end{proof}

By Lemma \ref{lem:sym}
the $\F$-vector space $(M A^* M)^{\text{\rm sym}}$ has dimension $d$.
The set \eqref{eq:symspan} contains $d+1$ elements.
So the vectors \eqref{eq:symspan} are linearly dependent.
We now find the dependency.
To avoid trivialities we assume that $d \geq 1$.

\begin{lemma}    \label{lem:sym5}   \samepage
\ifDRAFT {\rm lem:sym5}. \fi
Assume that $d \geq 1$.
Then the following {\rm (i)--(iii)} hold.
\begin{itemize}
\item[\rm (i)]
$\text{\rm Char} (\F) \neq 2$.
\item[\rm (ii)]
There exists a unique integer $n$ $(1 \leq n \leq d)$ such that $\th^*_n = - \th^*_0$.
\item[\rm (iii)]
The element $A^n A^* + A^* A^n$ is contained in the span of
\begin{equation}    \label{eq:sym5aux}
 \{A^*, \;
 A A^* + A^* A, \;
 A^2 A^* + A^* A^2, \; \ldots, \;
 A^{n-1} A^* + A^* A^{n-1}\}.
\end{equation}
\end{itemize}
\end{lemma}

\begin{proof}
Since the vectors \eqref{eq:symspan} are linearly dependent,
there exist scalars $\{\alpha_i\}_{i=0}^d$ in $\F$, not all zero, such that
\[
  \alpha_0 A^* + \sum_{i=1}^d \alpha_i (A^i A^* + A^* A^i) = 0.
\]
Define $n = \text{\rm max}\,\{i \,|\, 0 \leq i \leq d, \; \alpha_i \neq 0\}$.
Note that $n \geq 1$, since $A^* \neq 0$ by Lemma \ref{lem:trivial}.
So
\begin{align}
   \alpha_0 A^* + \sum_{i=1}^n \alpha_i (A^i A^* + A^* A^i) &= 0, 
  &  \alpha_n \neq 0.                                          \label{eq:sym5aux1}
\end{align}
By \eqref{eq:sym5aux1} the element $A^n A^* + A^* A^n$ is contained 
in the span of \eqref{eq:sym5aux}.
In \eqref{eq:sym5aux1}, 
multiply each side on the left by $E^*_0$ and right by $E^*_n$.
Simplify the result using 
$E^*_0 A^* = \th^*_0 E^*_0$, $A^* E^*_n = \th^*_n E^*_n$, $E^*_0 E^*_n =0$ to get
\[
  (\th^*_0 + \th^*_n) \sum_{i=1}^n \alpha_i E^*_0 A^i E^*_n = 0.
\]
By Lemma \ref{lem:comments}(ii), $E^*_0 A^i E^*_n = 0$ for $1 \leq i \leq n-1$.
So the above line  becomes
\[
   \alpha_n (\th^*_0 + \th^*_n) E^*_0 A^n E^*_n = 0.
\]
We have $\alpha_n \neq 0$,
and  $E^*_0 A^n E^*_n  \neq 0$ by Lemma \ref{lem:comments}(ii).
Thus $\th^*_0 + \th^*_n=0$,
and so $\th^*_n = - \th^*_0$.
The $\{\th^*_i\}_{i=0}^d$ are mutually distinct,
so no dual eigenvalue other than $\th^*_n$ is equal to $- \th^*_0$.
We have $\text{Char}(\F) \neq 2$; otherwise $\th^*_n = \th^*_0$.
The result follows.
\end{proof}

\begin{note}
Referring to Lemma \ref{lem:sym5}(ii),
it turns out that $n=d$;
this will be established in Lemma \ref{lem:onlyif}.
\end{note}

\begin{proposition}   \label{prop:AW}   \samepage
\ifDRAFT {\rm prop:AW}. \fi
There exist  $\beta$, $\varrho$, $\varrho^*$ in $\F$ that satisfy both
\eqref{eq:AW1}, \eqref{eq:AW2}.
\end{proposition}

\begin{proof}
We first show that
\begin{equation}
 AA^*A \in \text{\rm Span} \,  \{ A^*, \, A A^* + A^* A, \, A^2 A^* + A^* A^2\}.     \label{eq:AWpreaux}
\end{equation}
We assume that $d \geq 3$; otherwise \eqref{eq:AWpreaux} holds by Lemma \ref{lem:sym4}.
Recall the integer $n$ from Lemma \ref{lem:sym5}.
By Lemmas \ref{lem:sym4} and \ref{lem:sym5}(iii)
there exist scalars $\{\alpha_i\}_{i=0}^d$ in $\F$ with $\alpha_n = 0$ such that
\[
 A A^* A = \alpha_0 A^* + \sum_{i=1}^d \alpha_i (A^i A^* + A^* A^i).
\]
We show that $\alpha_i = 0$ for $ 3 \leq i \leq d$.
Suppose not.
Then there exists an integer $t$ $(3 \leq t \leq d)$ such that $\alpha_t \neq 0$.
We choose $t$ to be maximal.
Then $t \neq n$ and
\[
 A A^* A = \alpha_0 A^* + \sum_{i=1}^t \alpha_i (A^i A^* + A^* A^i).
\]
In the above line, multiply each side on the left by $E^*_0$ and  right by $E^*_t$.
Simplify the result to get
\[
E^*_0 A A^* A E^*_t
 = (\th^*_0 + \th^*_t) \sum_{i=1}^t \alpha_i E^*_0 A^i E^*_t.
\]
In the above line the left-hand side is zero by Lemma  \ref{lem:EsAAsAEs}, and 
$E^*_0 A^i E^*_t = 0$ $(1 \leq i \leq t-1)$ by Lemma \ref{lem:comments}(ii).
Thus
\begin{equation}
   0 =\alpha_t (\th^*_0 + \th^*_t)  E^*_0 A^t E^*_t.      \label{eq:AWpre3}
\end{equation}
We examine the factors in \eqref{eq:AWpre3}.
By construction $\alpha_t \neq 0$.
We have  $\th^*_0 + \th^*_t \neq 0$ by Lemma \ref{lem:sym5}(ii) and $t \neq n$.
Also $E^*_ 0 A^t E^*_t \neq 0$ by Lemma \ref{lem:comments}(ii).
Therefore the right-hand side of \eqref{eq:AWpre3} is nonzero, for a contradiction.
We have shown \eqref{eq:AWpreaux}.

Next we show that
there exists  $\beta \in \F$ such that
\begin{align}
    \th^*_{i-1} - \beta \th^*_i + \th^*_{i+1} &= 0  \qquad\qquad (1 \leq i \leq d-1).    \label{eq:AWpre2}
\end{align}
We assume that $d \geq 2$; otherwise the assertion is vacuous.
By \eqref{eq:AWpreaux} there exist scalars $\alpha_0$, $\alpha_1$, $\alpha_2$ in $\F$ such that
\[
  0 = \alpha_0 A^* + \alpha_1 (AA^* + A^* A) + \alpha_2 (A^2 A^* + A^* A^2) - A A^* A.
\]
For $1 \leq i \leq d-1$, in the above line multiply each side on the left by $E^*_{i-1}$ and right by $E^*_{i+1}$.
Simplify the result using Lemmas \ref{lem:comments}(ii) and \ref{lem:EsAAsAEs}   to get
\begin{align*}
 0 &= E^*_{i-1} A^2 E^*_{i+1} (\alpha_2 \th^*_{i+1} + \alpha_2 \th^*_{i-1} - \th^*_i).
                     \qquad\qquad  (1 \leq i \leq d-1).        
\end{align*}
We have $E^*_{i-1} A^2 E^*_{i+1} \neq 0$ by Lemma \ref{lem:comments}(ii).
Thus
\begin{align*}
  \alpha_2 (\th^*_{i-1} + \th^*_{i+1}) &= \th^*_i  \qquad\qquad  (1 \leq i \leq d-1).
\end{align*}
Assume for the moment that $\alpha_2 \neq 0$.
Then \eqref{eq:AWpre2} holds for $\beta = \alpha_2^{-1}$.
Next assume that $\alpha_2 = 0$.
Then $\th^*_i = 0$ for $1 \leq i \leq d-1$.
This forces $d=2$ and $\th^*_1 = 0$.
In this case  $\th^*_2 = - \th^*_0$ by Lemma \ref{lem:sym5}(ii)
and \eqref{eq:AWpre2} holds for any $\beta \in \F$.
We have shown that there exists $\beta \in \F$ that satisfies \eqref{eq:AWpre2}.

By \eqref{eq:AWpre2} and Lemma \ref{lem:rec1}(i)
there exists $\varrho^* \in \F$ such that
\begin{align*}
  \th^{*2}_{i-1} - \beta \th^*_{i-1} \th^*_i + {\th^*_i}^2 &= \varrho^*   \qquad\qquad (1 \leq i \leq d).
\end{align*}
By this and Lemma \ref{lem:rec4} (applied to $\Phi^*$) we obtain \eqref{eq:AW2}.
For $1 \leq i \leq d-1$,
multiply each side of \eqref{eq:AW2} on the left by $E_{i-1}$ and right by $E_{i+1}$.
Simplify the result using Lemmas  \ref{lem:comments}(ii) and \ref{lem:EsAAsAEs} (applied to $\Phi^*$) to get
\begin{align*}
  0  &=  E_{i-1} {A^*}^2 E_{i+1} (\th_{i-1} - \beta \th_i + \th_{i+1})   \qquad\qquad (1 \leq i \leq d-1).
\end{align*}
Applying Lemma \ref{lem:comments}(ii) to $\Phi^*$ we obtain
$E_{i-1} {A^*}^2 E_{i+1} \neq 0$. 
Thus
\begin{align*}
   \th_{i-1} - \beta \th_i + \th_{i+1} &= 0   \qquad\qquad (1 \leq i \leq d-1).
\end{align*}
By this and Lemma \ref{lem:rec1}(i) there exists  $\varrho \in \F$ such that
\begin{align*}
 \th_{i-1}^2 - \beta \th_{i-1} \th_i + \th_i^2 &= \varrho   \qquad\qquad  (1 \leq i \leq d).
\end{align*}
By this and Lemma \ref{lem:rec4} we obtain \eqref{eq:AW1}.
\end{proof}

We mention some results for later use.

\begin{lemma}    \label{lem:iso}   \samepage
\ifDRAFT {\rm lem:iso}. \fi
For $X \in \text{\rm End}(V)$ the following are equivalent:
\begin{itemize}
\item[\rm (i)]
$X$ commutes with each of $A$, $A^*$;
\item[\rm (ii)]
there exists $\lambda \in \F$ such that $X = \lambda I$.
\end{itemize}
\end{lemma}

\begin{proof}
(i) $\Rightarrow$ (ii)
By Lemma \ref{lem:comments}(i) the element
$X$ is contained in the center of $\text{\rm End}(V)$.
The center of $\text{\rm End}(V)$ is spanned by $I$.
The result follows.

(ii) $\Rightarrow$ (i)
Clear.
\end{proof}

The next two results follow from Lemma \ref{lem:iso}.

\begin{corollary}   \label{cor:iso3}    \samepage
\ifDRAFT {\rm cor:iso3}. \fi
For $\psi \in \text{\rm End}(V)$ the following are equivalent:
\begin{itemize}
\item[\rm (i)]
$\psi$ is an isomorphism of TB tridiagonal pairs
from $A,A^*$ to $A,A^*$;
\item[\rm (ii)]
there exists nonzero $\lambda \in \F$ such that $\psi = \lambda I$.
\end{itemize}
\end{corollary}

\begin{corollary}   \label{cor:iso4}    \samepage
\ifDRAFT {\rm cor:iso4}. \fi
For $\psi \in \text{\rm End}(V)$  the following are equivalent:
\begin{itemize}
\item[\rm (i)]
$\psi$ is an isomorphism of TB tridiagonal systems
from $\Phi$ to $\Phi$;
\item[\rm (ii)]
there exists nonzero $\lambda \in \F$ such that $\psi = \lambda I$.
\end{itemize}
\end{corollary}

\section{The Askey-Wilson sequence and the fundamental parameter}
\label{sec:AWseq}

In this section, we introduce the notion of an Askey-Wilson sequence
and fundamental parameter for TB tridiagonal pairs and systems.

Throughout this section
let $A,A^*$ denote a TB tridiagonal pair over $\F$.
Let $\Phi$ denote an associated TB tridiagonal system.

\begin{definition}    \label{def:AWseq}     \samepage
\ifDRAFT {\rm def:AWseq}. \fi
By an {\em Askey-Wilson sequence for $A,A^*$} we mean
a sequence $\beta,\varrho,\varrho^*$ of scalars in $\F$
that satisfy \eqref{eq:AW1} and \eqref{eq:AW2}.
\end{definition}

\begin{definition}    \label{def:sysAWseq}    \samepage
\ifDRAFT {\rm def:sysAWseq}. \fi
By an {\em Askey-Wilson sequence for $\Phi$} we mean
an Askey-Wilson sequence for $A,A^*$.
\end{definition}

\begin{lemma}   \label{lem:AWseqPhi}   \samepage
\ifDRAFT {\rm lem:AWseqPhi}. \fi
The following are the same:
\begin{itemize}
\item[\rm (i)]
an Askey-Wilson sequence for $\Phi$;
\item[\rm (ii)]
an  Askey-Wilson sequence for $\Phi^\downarrow$;
\item[\rm (iii)]
an  Askey-Wilson sequence for $\Phi^\Downarrow$.
\end{itemize}
\end{lemma}

\begin{proof}
The TB tridiagonal systems $\Phi$, $\Phi^\downarrow$, $\Phi^\Downarrow$
have the same associated TB tridiagonal pair $A,A^*$.
The result follows.
\end{proof}

Let $(\{\th_i\}_{i=0}^d; \{\th^*_i\}_{i=0}^d)$ denote the eigenvalue array of $\Phi$.

\begin{proposition}   \label{prop:recurrence}    \samepage
\ifDRAFT {\rm prop:recurrence}. \fi
Let $\beta, \varrho, \varrho^*$ denote a sequence of scalars taken from $\F$.
This sequence is an Askey-Wilson sequence for $\Phi$
if and only if both
\begin{align}
\th_{i-1}^2 - \beta \th_{i-1} \th_i + \th_i^2 &= \varrho    \qquad\qquad (1 \leq i \leq d),   \label{eq:AWrho}
\\
 \th^{*2}_{i-1} - \beta \th^*_{i-1} \th^*_i + {\th^*_i}^2 &= \varrho^* 
                                       \qquad\qquad  (1 \leq i \leq d).                                    \label{eq:AWrhos}
\end{align}
\end{proposition}

\begin{proof}
Apply Lemma \ref{lem:rec4} to $\Phi$ and $\Phi^*$.
\end{proof}

\begin{lemma}    \label{lem:AWseqdual}    \samepage
\ifDRAFT {\rm lem:AWseqdual}. \fi
Let $\beta,\varrho,\varrho^*$ denote an Askey-Wilson sequence for $A,A^*$.
Then $\beta,\varrho^*,\varrho$ is an Askey-Wilson sequence for $A^*,A$.
\end{lemma}

\begin{proof}
By \eqref{eq:AW1} and \eqref{eq:AW2}.
\end{proof}

\begin{lemma}    \label{lem:AWseqhhs}    \samepage
\ifDRAFT {\rm lem:AWseqhhs}. \fi
Let $\beta, \varrho, \varrho^*$ denote an Askey-Wilson sequence
for $A,A^*$.
Then for nonzero $\zeta$, $\zeta^* \in \F$ the 
TB tridiagonal pair $\zeta A, \zeta^* A^*$
has an Askey-Wilson sequence $\beta$, $\zeta^2 \varrho$, $\zeta^{*2} \varrho^*$.
\end{lemma}

\begin{proof}
One verifies that $\beta$, $\zeta^2 \varrho$, $\zeta^{*2} \varrho^*$ satisfy
the Askey-Wilson relations for $\zeta A, \zeta^* A^*$.
\end{proof}

\begin{definition}    \label{def:pairfund}    \samepage
\ifDRAFT {\rm def:pairfund}. \fi
A scalar $\beta \in \F$ is called a {\em fundamental parameter for $A,A^*$}
whenever there exist $\varrho$, $\varrho^* \in \F$
such that $\beta,\varrho,\varrho^*$ is an Askey-Wilson sequence for $A,A^*$.
\end{definition}

\begin{definition}    \label{def:sysfund}    \samepage
\ifDRAFT {\rm def:sysfund}. \fi
By a  {\em fundamental parameter for $\Phi$} we mean
a fundamental parameter for $A,A^*$.
\end{definition}

\begin{lemma}    \label{lem:sysfund}    \samepage
\ifDRAFT {\rm lem:sysfund}. \fi
Let $\beta$ denote a fundamental parameter for $\Phi$.
Then $\beta$ is a fundamental parameter for each of $\Phi$, $\Phi^\downarrow$, $\Phi^\Downarrow$,
$\Phi^*$.
\end{lemma} 

\begin{proof}
By Lemmas \ref{lem:AWseqPhi} and \ref{lem:AWseqdual}.
\end{proof}

\begin{proposition}    \label{prop:beta}    \samepage
\ifDRAFT {\rm prop:beta}. \fi
A scalar $\beta \in \F$ is a fundamental parameter for $\Phi$
if and only if both
\begin{align}
 \th_{i-1} - \beta \th_i + \th_{i+1} &= 0  \qquad\qquad  (1 \leq i \leq d-1),             \label{eq:AWbeta1}
\\
  \th^*_{i-1} - \beta \th^*_i + \th^*_{i+1} &= 0 \qquad\qquad (1 \leq i \leq d-1).   \label{eq:AWbeta2}
\end{align}
\end{proposition}

\begin{proof}
By Lemma \ref{lem:rec1} and Proposition \ref{prop:recurrence}.
\end{proof}

\begin{lemma}   \label{lem:funddual}   \samepage
\ifDRAFT {\rm lem:funddual}. \fi
Let $\beta$ denote a fundamental parameter for $A,A^*$.
Then $\beta$ is a fundamental parameter for $A^*,A$.
\end{lemma}

\begin{proof}
By Lemma \ref{lem:AWseqdual} or Lemma \ref{lem:sysfund}.
\end{proof}

\begin{lemma}    \label{lem:fundaffine}    \samepage
\ifDRAFT {\rm lem:fundaffine}. \fi
Let $\beta$ denote a fundamental parameter for $A,A^*$.
Then $\beta$ is a fundamental parameter for $\zeta A, \zeta^* A^*$
for all nonzero $\zeta, \zeta^* \in \F$.
\end{lemma}

\begin{proof}
By Lemma \ref{lem:AWseqhhs}.
\end{proof}

The uniqueness of the Askey-Wilson sequence and fundamental parameter
will be discussed in Lemmas \ref{lem:fundunique}, \ref{lem:rhorhosunique}.

\section{More on recurrent sequences}
\label{sec:sym}

Recall the recurrent sequences from Definition \ref{def:recurrent}.
In this section we obtain a detailed description of these sequences.
In our description two special cases come up,
said to be symmetric and antisymmetric.
Throughout this section fix an integer $d \geq 1$.

\begin{definition}     \label{def:symasym}    \samepage
\ifDRAFT {\rm def:symasym}. \fi
Let $\{\sigma_i\}_{i=0}^d$ denote a sequence of scalars taken from $\F$.
This sequence is said to be {\em over $\F$} and have {\em diameter $d$}.
The sequence $\{\sigma_i\}_{i=0}^d$ is said to be {\em symmetric} whenever $\sigma_i = \sigma_{d-i}$
for $0 \leq i \leq d$.
The sequence $\{\sigma_i\}_{i=0}^d$ is said to be {\em antisymmetric} whenever $\sigma_i + \sigma_{d-i} =0$
for $0 \leq i \leq d$.
\end{definition}

For the rest of this section fix $\beta \in \F$.

\begin{definition}    \label{def:R}    \samepage
\ifDRAFT {\rm def:R}. \fi
Let $\R$ denote the set of $\beta$-recurrent sequences over $\F$ that have
diameter $d$.
Let $\R^\text{\rm sym}$ (resp. $\R^\text{\rm asym}$) denote the set of symmetric
(resp.\ antisymmetric) elements in $\R$.
Note that $\R$ is a subspace of the $\F$-vector space $\F^{d+1}$.
Moreover $\R^\text{\rm sym}$ and $\R^\text{\rm asym}$ are subspaces of $\R$.
\end{definition}

\begin{lemma}   \label{lem:directsum}   \samepage
\ifDRAFT {\rm lem:directsum}. \fi
If $\text{\rm Char}(\F) = 2$ then $\R^\text{\rm sym} = \R^\text{\rm asym}$.
If $\text{\rm Char}(\F) \neq 2$ then
\[
\R = \R^\text{\rm sym} + \R^\text{\rm asym}  \qquad\qquad  \text{\rm (direct sum)}.
\]
\end{lemma}

\begin{proof}
The first assertion is clear.
Concerning the second assertion, assume that $\text{\rm Char}(\F) \neq 2$.
For a sequence $\{\sigma_i\}_{i=0}^d$ in $\R$ define 
\begin{align*}
  \sigma^{+}_i &=  \frac{\sigma_i + \sigma_{d-i}}{2},
\qquad\qquad    \sigma^{-}_i =  \frac{\sigma_i - \sigma_{d-i}}{2} && (0 \leq i \leq d).
\end{align*}
We have $\sigma_i = \sigma^+_i + \sigma^-_i$ for $0 \leq i \leq d$.
The sequence $\{\sigma^+_i\}_{i=0}^d$ is $\beta$-recurrent and symmetric.
The sequence $\{\sigma^-_i\}_{i=0}^d$ is $\beta$-recurrent and antisymmetric.
By these comments $\R= \R^\text{\rm sym} + \R^\text{\rm asym}$.
We show that this sum is direct.
Pick an element $\{\sigma_i\}_{i=0}^d$ of $\R^\text{\rm sym} \cap \R^\text{\rm asym}$.
For $0 \leq i \leq d$ we have both $\sigma_{d-i} = \sigma_i$ and  $\sigma_{d-i} = - \sigma_i$,
and so $\sigma_i = 0$.
Therefore $\R^\text{\rm sym} \cap \R^\text{\rm asym} = 0$ and
consequently the sum $\R^\text{\rm sym} + \R^\text{\rm asym}$ is direct.
\end{proof}

Our next goal is to display a basis for  $\R^\text{\rm sym}$ and  $\R^\text{\rm asym}$,
under the assumption that $\text{\rm Char}(\F) \neq 2$.
Our strategy is to first display some nonzero elements in 
 $\R^\text{\rm sym}$ and  $\R^\text{\rm asym}$,
and a bit later show that these elements form a basis.
The cases $\beta=2$, $\beta =-2$ will be handled separately.

\begin{lemma}    \label{lem:symbasis1}    \samepage
\ifDRAFT {\rm lem:symbasis1}. \fi
Assume that $\text{\rm Char}(\F) \neq 2$ and $\beta = \pm 2$.
For $0 \leq i \leq d$ define $\sigma_i$ as follows:
\begin{equation}                               \label{eq:symbasis1}
\begin{array}{c|c}
 \text{\rm Case} & \sigma_i
\\ \hline
\beta = 2 &  1     \rule{0mm}{3ex} 
\\
\beta = -2, \;\; \text{\rm $d$ even}  & (-1)^i   \rule{0mm}{3ex}
\\
\beta=-2, \;\; \text{\rm $d$ odd} & (d-2i)(-1)^i      \rule{0mm}{3ex}
\end{array}
\end{equation}
Then $\{\sigma_i\}_{i=0}^d$ is nonzero and contained in $\R^\text{\rm sym}$.
\end{lemma}

\begin{proof}
One routinely checks that $\{\sigma_i\}_{i=0}^d$ is nonzero, $\beta$-recurrent, and symmetric.
The result follows.
\end{proof}

\begin{lemma}    \label{lem:asymbasis1}    \samepage
\ifDRAFT {\rm lem:asymbasis1}. \fi
Assume that  $\text{\rm Char}(\F) \neq 2$ and $\beta = \pm 2$.
For $0 \leq i \leq d$ define $\sigma_i$ as follows:
\begin{equation}                               \label{eq:asymbasis1}
\begin{array}{c|c}
 \text{\rm Case} & \sigma_i
\\ \hline
\beta = 2    &   d - 2i     \rule{0mm}{3ex} 
\\
\beta = -2, \;\; \text{\rm $d$ even}  &   (d-2i)(-1)^i   \rule{0mm}{3ex}
\\
\beta=-2, \;\; \text{\rm  $d$ odd} & (-1)^i     \rule{0mm}{3ex}
\end{array}
\end{equation}
Then $\{\sigma_i\}_{i=0}^d$ is nonzero and contained in $\R^\text{\rm asym}$.
\end{lemma}

\begin{proof}
One routinely checks that $\{\sigma_i\}_{i=0}^d$ is nonzero, $\beta$-recurrent, and antisymmetric.
The result follows.
\end{proof}

\begin{lemma}    \label{lem:symbasis2}    \samepage
\ifDRAFT {\rm lem:symbasis2}. \fi
Assume that $\text{\rm Char}(\F) \neq 2$ and $\beta \neq \pm 2$.
Let $q \in \overline{\F}$ be such that $\beta = q^2+q^{-2}$.
For $0 \leq i \leq d$ define $\sigma_i$ by
\[
\sigma_i =
 \begin{cases}
    q^{d-2i} + q^{2i-d} & \quad \text{if $d$ is even},
  \\
   \displaystyle \frac{q^{d-2i} + q^{2i-d}}{q+q^{-1}}  & \quad \text{if $d$ is odd}. \rule{0mm}{5ex}
  \end{cases}
\]
Then $\{\sigma_i\}_{i=0}^d$ is nonzero and contained in  $\R^\text{\rm sym}$.
\end{lemma}

\begin{proof}
Clearly $\{\sigma_i\}_{i=0}^d$ is nonzero.
We verify that $\{\sigma_i\}_{i=0}^d$ is contained in $\R^\text{\rm sym}$.
One routinely checks that $\{\sigma_i\}_{i=0}^d$ is $\beta$-recurrent and symmetric.
We show that $\sigma_i \in \F$ for $0 \leq i \leq d$.
First assume that $d$ is even.
Observe that $\sigma_{d/2} = 2$ and $\sigma_{d/2-1} = q^2 + q^{-2} = \beta$.
Therefore each of $\sigma_{d/2}$, $\sigma_{d/2-1}$ is contained in $\F$.
By this and since $\{\sigma_i\}_{i=0}^d$ is $\beta$-recurrent, we obtain $\sigma_i \in \F$
for $0 \leq i \leq d$.
Next assume that $d$ is odd.
Observe that $\sigma_{(d-1)/2} = 1$ and 
$\sigma_{(d+1)/2} = 1$.
Therefore each of $\sigma_{(d-1)/2}$, $\sigma_{(d+1)/2}$ is contained in $\F$.
By this and since  $\{\sigma_i\}_{i=0}^d$ is $\beta$-recurrent, we obtain $\sigma_i \in \F$
for $0 \leq i \leq d$.
Thus $\{\sigma_i\}_{i=0}^d$ is contained in $\R^\text{\rm sym}$.
The result follows.
\end{proof}

\begin{lemma}    \label{lem:asymbasis2}    \samepage
\ifDRAFT {\rm lem:asymbasis2}. \fi
Assume that $\text{\rm Char}(\F) \neq 2$ and $\beta \neq \pm 2$.
Let $q \in \overline{\F}$ be such that $\beta = q^2+q^{-2}$.
For $0 \leq i \leq d$ define $\sigma_i$ by
\[
\sigma_i =
 \begin{cases}
   \displaystyle \frac{q^{d-2i} - q^{2i-d}}{q^2-q^{-2}}    & \quad \text{if $d$ is even},
  \\
   \displaystyle \frac{q^{d-2i} - q^{2i-d}}{q-q^{-1}}  & \quad \text{if $d$ is odd}. \rule{0mm}{5ex}
  \end{cases}
\]
Then $\{\sigma_i\}_{i=0}^d$ is nonzero and contained in $\R^\text{\rm asym}$.
\end{lemma}

\begin{proof}
Clearly $\{\sigma_i\}_{i=0}^d$ is nonzero.
We verify that $\{\sigma_i\}_{i=0}^d$ is contained in $\R^\text{\rm asym}$.
One routinely checks that $\{\sigma_i\}_{i=0}^d$ is $\beta$-recurrent and antisymmetric.
We show that $\sigma_i \in \F$ for $0 \leq i \leq d$.
First assume that $d$ is even.
Observe that $\sigma_{d/2} = 0$ and $\sigma_{d/2-1} = 1$.
Therefore each of $\sigma_{d/2}$, $\sigma_{d/2-1}$ is contained in $\F$.
By this and since $\{\sigma_i\}_{i=0}^d$ is $\beta$-recurrent, we obtain $\sigma_i \in \F$
for $0 \leq i \leq d$.
Next assume that $d$ is odd.
Observe that $\sigma_{(d-1)/2} = 1$ and 
$\sigma_{(d+1)/2} = -1$.
Therefore each of $\sigma_{(d-1)/2}$, $\sigma_{(d+1)/2}$ is contained in $\F$.
By this and since  $\{\sigma_i\}_{i=0}^d$ is $\beta$-recurrent, we obtain $\sigma_i \in \F$
for $0 \leq i \leq d$.
Thus $\{\sigma_i\}_{i=0}^d$ is contained in $\R^\text{\rm asym}$.
The result follows.
\end{proof}

\begin{note}    \label{note:q}   \samepage
\ifDRAFT {\rm note:q}. \fi
The parameter $q$ in Lemmas \ref{lem:symbasis2}, \ref{lem:asymbasis2} is not uniquely determined
by $\{\sigma_i\}_{i=0}^d$.
To clarify the situation,
let $y$ denote an indeterminate,
and consider the equation 
\begin{equation}    \label{eq:betaq}
\beta = y^2 + y^{-2}.
\end{equation}
Let $0 \neq q \in \F$ denote a solution of \eqref{eq:betaq}.
Then the solutions of \eqref{eq:betaq} are
$q$, $-q$, $q^{-1}$, $-q^{-1}$.
In Lemmas \ref{lem:symbasis2} and \ref{lem:asymbasis2},
replacing $q$ by $-q$, $q^{-1}$, $-q^{-1}$ does not
change $\sigma_i$.
Thus the given sequence $\{\sigma_i\}_{i=0}^d$ depends only on $d$ and $\beta$,
and not on the choice of $q$.
\end{note}

In Lemmas \ref{lem:symbasis2} and \ref{lem:asymbasis2},
each scalar $\sigma_i$ can be represented as a polynomial in $\beta$.
The polynomial has Chebychev type, as we now explain.
Assume that $\text{\rm Char}(\F) \neq 2$.
For scalars $a$, $b$, $c$ in $\F$ define polynomials $T_0(x)$, $T_1(x)$, $T_2(x)$, \ldots
by
\begin{align*}
  T_0(x) &= a,
\\
  T_1(x) &= b x + c,
\\
  T_i (x) &= 2 x T_{i-1}(x) - T_{i-2}(x)  \qquad\qquad\qquad i=2,3,\ldots.
\end{align*}
Then $T_i(x)$ is the $i$th Chebychev polynomial with parameters $a$, $b$, $c$
\cite[Remark 2.5.3]{AAR}.
Now consider 
the sequence $\{\sigma_i\}_{i=0}^d$  from Lemma \ref{lem:symbasis2}.
\begin{itemize}
\item[\rm (i)]
Assume that $d$ is even. 
Then for $n=d/2$ and $a=2$, $b=2$, $c=0$,
\begin{align*}
    \sigma_i &= T_{n-i}(\beta/2)  \qquad\qquad (0 \leq i \leq n).
\end{align*}
\item[\rm (ii)]
Assume that $d$ is odd.
Then for $n=(d-1)/2$ and $a=1$, $b=2$, $c=-1$,
\begin{align*}
  \sigma_i = T_{n-i} (\beta/2)  \qquad\qquad  (0 \leq i \leq n).
\end{align*}
\end{itemize}
Next consider the sequence
$\{\sigma_i\}_{i=0}^d$  from Lemma \ref{lem:asymbasis2}.
\begin{itemize}
\item[\rm (i)]
Assume that $d$ is even.
Then for $n=d/2$ and $a=1$, $b=2$, $c=0$,
\begin{align*}
    \sigma_i &= T_{n-1-i}(\beta/2)  \qquad (0 \leq i \leq n-1),  
  \qquad\qquad  \sigma_n = 0.
\end{align*}
\item[\rm (ii)]
Assume that $d$ is odd.
Then for $n=(d-1)/2$ and $a=1$, $b=2$, $c=1$,
\begin{align*}
  \sigma_i = T_{n-i} (\beta/2)  \qquad\qquad  (0 \leq i \leq n).
\end{align*}
\end{itemize}

\begin{lemma}    \label{lem:dimR}    \samepage
\ifDRAFT {\rm lem:dimR}. \fi
Assume that  $\text{\rm Char}(\F) \neq 2$.
Then 
\begin{align*}
\dim \R &= 2, \qquad\qquad
\dim \R^\text{\rm sym} = 1, \qquad\qquad
\dim \R^\text{\rm asym} =1.
\end{align*}
\end{lemma} 

\begin{proof}
Pick a sequence $\{\sigma_i\}_{i=0}^d$ in $\R$.
By the $\beta$-recurrence,
the scalars $\{\sigma_i\}_{i=0}^d$ are uniquely determined by 
$\sigma_0$, $\sigma_1$.
Thus $\dim \R \leq 2$.
The subspace $\R^\text{\rm sym}$ is nonzero by Lemmas \ref{lem:symbasis1} and \ref{lem:symbasis2}.
The subspace  $\R^\text{\rm asym}$ is nonzero
by Lemmas \ref{lem:asymbasis1}  and \ref{lem:asymbasis2}.
By these comments and Lemma \ref{lem:directsum} we obtain the results.
\end{proof}

\begin{corollary}     \label{cor:basis}    \samepage
\ifDRAFT {\rm cor:basis}. \fi
The following hold.
\begin{itemize}
\item[\rm (i)]
Assume that $\beta = \pm 2$.
Then the sequence $\{\sigma_i\}_{i=0}^d$ from Lemma \ref{lem:symbasis1} 
(resp.\ Lemma \ref{lem:asymbasis1})
is a basis for $\R^\text{\rm sym}$ (resp.\ $\R^\text{\rm asym}$).
\item[\rm (ii)]
Assume that $\beta \neq \pm 2$.
Then the sequence $\{\sigma_i\}_{i=0}^d$ from Lemma \ref{lem:symbasis2} 
(resp.\ Lemma \ref{lem:asymbasis2})
is a basis for $\R^\text{\rm sym}$ (resp.\ $\R^\text{\rm asym}$).
\end{itemize}
\end{corollary}

\begin{proof}
By Lemma \ref{lem:dimR}.
\end{proof}

For a sequence $\{\sigma_i\}_{i=0}^d$ in $\R$,
we sometimes consider the case in which $\{\sigma_i\}_{i=0}^d$ are mutually distinct.
Of course this does not happen if $\{\sigma_i\}_{i=0}^d$ is symmetric.
But it could happen if $\{\sigma_i\}_{i=0}^d$ is antisymmetric.
We now give the details.

\begin{definition}    \label{def:MutDist}    \samepage
\ifDRAFT {\rm def:MutDist}. \fi
A sequence $\{\sigma_i\}_{i=0}^d$ in $\R$ is said to be {\em MutDist}
whenever $\sigma_0$, $\sigma_1$, \ldots, $\sigma_d$ are mutually distinct.
\end{definition}

\begin{definition}    \label{def:MutDist2}    \samepage
\ifDRAFT {\rm def:MutDist2}. \fi
A subspace of $\R$ is said to be {\em MutDist}
whenever the subspace is nonzero and every nonzero element of the subspace is MutDist.
\end{definition}

\begin{lemma}   \label{lem:feasible}    \samepage
\ifDRAFT {\rm lem:feasible}. \fi
The following {\rm (i), (ii)} are equivalent:
\begin{itemize}
\item[\rm (i)]
$\R^\text{\rm asym}$ is MutDist;
\item[\rm (ii)]
$\R^{\rm asym}$ contains a MutDist element.
\end{itemize}
Assume that {\rm (i), (ii)} hold.
Then $\text{\rm Char}(\F) \neq 2$.
\end{lemma}

\begin{proof}
(i) $\Rightarrow$ (ii)
Follows from Definition \ref{def:MutDist2}.

(ii) $\Rightarrow$ (i)
Since $\R^\text{\rm asym}$ has dimension $1$ by Lemma \ref{lem:dimR}.

Assume that (i), (ii) hold.
Let $\{\sigma_i\}_{i=0}^d$ denote a MutDist element in $\R^\text{\rm asym}$.
Then $\sigma_d = - \sigma_0$ and $\sigma_d \neq \sigma_0$.
Thus $\text{\rm Char}(\F) \neq 2$.
\end{proof}

\begin{lemma}   \label{lem:deven}    \samepage
\ifDRAFT {\rm lem:deven}. \fi
Assume that $d \geq 2$ and $\beta=-2$.
Assume that $\R^\text{\rm asym}$ is MutDist.
Then $d$ is even.
\end{lemma}

\begin{proof}
By Lemma \ref{lem:feasible}, $\text{\rm Char}(\F) \neq 2$.
Let $\{\sigma_i\}_{i=0}^d$ be from Lemma \ref{lem:asymbasis1}. 
Then $\{\sigma_i\}_{i=0}^d$ is MutDist.
If $d$ is odd, then $\sigma_0 = \sigma_2$, for a contradiction. Thus $d$ is even.
\end{proof}

\begin{lemma}    \label{lem:closedpre4pre}    \samepage
\ifDRAFT {\rm lem:closedpre4pre}. \fi
Let the sequence $\{\sigma_i\}_{i=0}^d$ be from Lemmas 
\ref{lem:asymbasis1} and \ref{lem:asymbasis2}.
Then for $0 \leq i,j \leq d$ the scalar $\sigma_i - \sigma_j$ is given as follows:
\[      
\begin{array}{c|c}
 \text{\rm Case} & \sigma_i - \sigma_j
\\ \hline
 \beta = 2   &   2(j-i)     \rule{0mm}{3ex}
\\
\beta = -2, \;\; \text{\rm $d$ even}  &  \;\;\;
  \begin{array}{c|cc}
     & \text{\rm $j\,$ even}  &  \text{\rm $j\,$ odd}
    \\ \hline
     \text{\rm $i\,$  even} & 2(j-i) & 2(d-i-j)   \rule{0mm}{2.5ex}  \\
     \text{\rm $i\,$  odd} & 2(i+j-d) & 2(i-j)
   \end{array}                          \rule{0mm}{6ex}
\\
\beta \neq \pm 2, \;\; \text{\rm $d$ even} 
  & \displaystyle \frac{(q^{2j} - q^{2i})(1+q^{2d-2i-2j})}{q^d(q^2-q^{-2})}     \rule{0mm}{6ex}
\\
\beta \neq \pm 2, \;\; \text{\rm $d$ odd}
   & \displaystyle \frac{(q^{2j} - q^{2i})(1+q^{2d-2i-2j})}{q^d(q-q^{-1})}    \rule{0mm}{6ex}
\end{array}
\]
In the above table the scalar $q$ is from Lemma \ref{lem:asymbasis2}.
\end{lemma}

\begin{proof}
Routine.
\end{proof}

\begin{lemma}    \label{lem:closedpre4}    \samepage
\ifDRAFT {\rm lem:closedpre4}. \fi
For $d=1$,
$\R^\text{\rm asym}$ is MutDist if and only if
$\text{\rm Char}(\F) \neq 2$.
For  $d \geq 2$,
$\R^\text{\rm asym}$ is MutDist 
if and only if the following conditions hold:
\begin{equation}                               \label{eq:conditions}
\begin{array}{c|c}
\text{\rm Case} & \text{\rm Conditions}
\\ \hline
\beta = 2
&
\text{\rm $\text{\rm Char}(\F)$ is $0$ or greater than $d$}                \rule{0mm}{3ex}
\\
\beta = - 2
&
\text{\rm $d$ is even}, \quad
\text{\rm $\text{\rm Char}(\F)$ is $0$ or greater than $d$}                \rule{0mm}{3ex}
\\
\beta \neq \pm 2 
&
\;\; \text{\rm Char}(\F) \neq 2, \quad
  q^{2i} \neq 1 \;\; (1 \leq i \leq d),   \quad
  q^{2i} \neq -1 \;\; (1 \leq i \leq d-1)                  \rule{0mm}{3ex}
\end{array}
\end{equation}
In the above table the scalar $q$ is from Lemma \ref{lem:asymbasis2}.
\end{lemma}

\begin{proof}
Use Lemmas \ref{lem:feasible}--\ref{lem:closedpre4pre}.
\end{proof}

We mention a lemma for later use.

\begin{lemma}    \label{lem:s1si}    \samepage
\ifDRAFT {\rm lem:s1si}. \fi
Let $\{\sigma_i\}_{i=0}^d$ denote a MutDist element in $\R^\text{\rm asym}$.
Then $\sigma_0$, $\sigma_d$ are nonzero and
\begin{align*}
  \sigma_1 \sigma_i &\neq \sigma_0 \sigma_{i-1}, \qquad\qquad
 \sigma_1 \sigma_i \neq \sigma_0 \sigma_{i+1}
 &&  (1 \leq i \leq d-1).
\end{align*}
\end{lemma}

\begin{proof}
Without loss of generality, we may assume that $\{\sigma_i\}_{i=0}^d$ is
the sequence from Lemmas \ref{lem:asymbasis1}, \ref{lem:asymbasis2}.
One routinely verifies the results using Lemma \ref{lem:closedpre4}.
\end{proof}

Motivated by Lemma \ref{lem:th0thd}(iii) we consider the
relation $\sigma_1 \sigma_d = \sigma_0 \sigma_{d-1}$.
We will need the following result.

\begin{lemma}    \label{lem:constpre}    \samepage
\ifDRAFT {\rm lem:constpre}. \fi
For sequences $\{\sigma_i\}_{i=0}^d$ and $\{\tau_i\}_{i=0}^d$ in $\R$ the following are equivalent:
\begin{itemize}
\item[\rm (i)]
$\sigma_0 \tau_1 \neq \sigma_1 \tau_0$;
\item[\rm (ii)]
$\{\sigma_i\}_{i=0}^d$ and $\{\tau_i\}_{i=0}^d$ are linearly independent.
\end{itemize}
\end{lemma}

\begin{proof}
(i) $\Rightarrow$ (ii)
Clear.

(ii) $\Rightarrow$ (i)
Consider the $d+1$ by $2$ matrix $H$ with row $i$ equal to 
$\sigma_i, \tau_i$ for $0 \leq i \leq d$.
The rank of $H$ is $2$.
By the $\beta$-recurrence, each row of $H$ is a linear combination
of rows $0$ and $1$.
Therefore rows $0$ and $1$ are linearly independent.
Now take the determinant of the submatrix of $H$ consisting of rows $0$ and $1$.
\end{proof}

\begin{proposition}    \label{prop:rec}     \samepage
\ifDRAFT {\rm prop:rec}. \fi
Assume that $\text{\rm Char}(\F) \neq 2$.
Then for a sequence $\{\sigma_i\}_{i=0}^d$ in $\R$ the following are equivalent:
\begin{itemize}
\item[\rm (i)]
$\sigma_1 \sigma_d = \sigma_0 \sigma_{d-1}$;
\item[\rm (ii)]
the sequence $\{\sigma_i\}_{i=0}^d$ is symmetric or antisymmetric.
\end{itemize}
\end{proposition}

\begin{proof}
By Lemma \ref{lem:directsum} there exist $\{\tau_i\}_{i=0}^d$ in $\R^\text{\rm sym}$
and  $\{\mu_i\}_{i=0}^d$ in $\R^\text{\rm asym}$ such that $\sigma_i = \tau_i + \mu_i$ for $0 \leq i \leq d$.
We have
\begin{align*}
  \sigma_0 \sigma_{d-1} - \sigma_1 \sigma_d
   &= (\tau_0 + \mu_0)(\tau_{d-1} + \mu_{d-1}) 
        - (\tau_1 + \mu_1)(\tau_d + \mu_d)  \\
   &= (\tau_0 + \mu_0)(\tau_1 - \mu_1)  
        - (\tau_1 + \mu_1)(\tau_0 - \mu_0)  \\
   &= 2(\tau_0 \mu_1 - \tau_1 \mu_0).
\end{align*}
By this and Lemma \ref{lem:constpre}, $\sigma_0 \sigma_{d-1} = \sigma_1 \sigma_d$
if and only if $\{\tau_i\}_{i=0}^d$ and $\{\mu_i\}_{i=0}^d$ are linearly dependent.
Since the sum $\R^\text{\em sym}+ \R^\text{\rm asym}$ is direct,
the sequences $\{\tau_i\}_{i=0}^d$, $\{\mu_i\}_{i=0}^d$ are linearly dependent if and only if at least one of 
$\{\tau_i\}_{i=0}^d$, $\{\mu_i\}_{i=0}^d$ is zero.
The result follows.
\end{proof}

\begin{corollary}   \label{cor:const}     \samepage
\ifDRAFT {\rm cor:const}. \fi
Assume that $\text{\rm Char}(\F) \neq 2$.
Let $\{\sigma_i\}_{i=0}^d$ denote a MutDist element in $\R$.
Then the following are equivalent:
\begin{itemize}
\item[\rm (i)]
$\{\sigma_i\}_{i=0}^d$ is antisymmetric;
\item[\rm (ii)]
$\sigma_1 \sigma_d = \sigma_0 \sigma_{d-1}$.
\end{itemize}
\end{corollary}

\begin{proof}
By Proposition \ref{prop:rec} and since no element of $\R^\text{\rm sym}$
is MutDist.
\end{proof}

\section{The classification of TB tridiagonal systems}
\label{sec:classify}

In this section we classify up to isomorphism the TB tridiagonal systems.
Fix an integer $d \geq 1$ and consider a sequence of scalars taken from $\F$:
\begin{equation}
   (\{\th_i\}_{i=0}^d; \{\th^*_i\}_{i=0}^d).                     \label{eq:parray1}
\end{equation}
We now state the classification.

\begin{theorem}  \label{thm:main}   \samepage
\ifDRAFT {\rm thm:main}. \fi
There  exists a TB tridiagonal system $\Phi$ over $\F$ that has eigenvalue
array \eqref{eq:parray1} if and only if the following {\rm (i)--(iii)} hold:
\begin{itemize}
\item[\rm (i)]
$\th_i \neq \th_j$, $\;\th^*_i \neq \th^*_j\;$ if $i \neq j$ $\;(0 \leq i,j \leq d)$;
\item[\rm (ii)]
there exists $\beta \in \F$ such that both
\begin{align}
\th_{i-1} - \beta \th_i + \th_{i+1} &= 0, &
\th^*_{i-1}-\beta \th^*_i + \th^*_{i+1} &= 0
&& (1 \leq i \leq d-1);                                   \label{eq:betarec}
\end{align}
\item[\rm (iii)]
$\th_i + \th_{d-i}=0$, $\;\; \th^*_i + \th^*_{d-i}=0\;$ $(0 \leq i \leq d)$. 
\end{itemize}
In this case $\Phi$ is unique up to isomorphism of TB tridiagonal systems.
\end{theorem}

\begin{note}    \label{note:d12}   \samepage
\ifDRAFT {\rm note:d12}. \fi
For $d \leq 2$, Theorem \ref{thm:main} remains valid if condition (ii) is deleted.
Here is the reason.
First assume that $d=1$.
Then \eqref{eq:betarec} is vacuous and so holds for any $\beta \in \F$.
Next assume that $d=2$.
Then condition (ii) follows from conditions (i), (iii).
Indeed by condition (iii) we have $\th_0 + \th_2=0$ and $2 \th_1=0$.
By this and condition (i) we have $\text{\rm Char}(\F) \neq 2$ and $\th_1=0$.
Similarly $\th^*_0 + \th^*_2=0$ and $\th^*_1=0$.
Therefore \eqref{eq:betarec} holds for any $\beta \in \F$.
\end{note}

The proof of Theorem \ref{thm:main} takes up most of this section.

\begin{lemma}   \label{lem:onlyif}    \samepage
\ifDRAFT {\rm lem:onlyif}. \fi
Assume that there exists a TB tridiagonal system $\Phi$ over $\F$ that
has eigenvalue array \eqref{eq:parray1}.
Then this eigenvalue array satisfies  conditions {\rm (i)--(iii)} in Theorem \ref{thm:main}.
Moreover $\Phi$ is unique up to isomorphism.
\end{lemma}

\begin{proof}
Condition (i) holds by construction.
Condition (ii) holds by Proposition \ref{prop:beta}.
Condition (iii) holds by Lemmas \ref{lem:th0thd}(iii), \ref{lem:sym5}(i), and Corollary \ref{cor:const}.
The uniqueness follows from Corollary \ref{cor:unique}.
\end{proof}

For the rest of this section we assume that the sequence \eqref{eq:parray1} satisfies conditions (i)--(iii)
in Theorem \ref{thm:main}.
We will construct a TB tridiagonal system $\Phi$ over $\F$ that has eigenvalue array \eqref{eq:parray1}.
Fix $\beta \in \F$ that satisfies \eqref{eq:betarec}.
We will refer to Section \ref{sec:sym} with this $\beta$.
Observe that each of the sequences 
$\{\th_i\}_{i=0}^d$, $\{\th^*_i\}_{i=0}^d$ is MutDist and contained in $\R^\text{\rm asym}$.
Note that $\text{\rm Char}(\F) \neq 2$ by Lemma \ref{lem:feasible}.

\begin{lemma}    \label{lem:new1}   \samepage
\ifDRAFT {\rm lem:new1}. \fi
There exists $0 \neq \zeta \in \F$ such that
$\th^*_i = \zeta \th_i$ for $0 \leq i \leq d$.
\end{lemma}

\begin{proof}
By Lemma \ref{lem:dimR} the subspace $\R^\text{\rm asym}$ has dimension one.
Each of $\{\th_i\}_{i=0}^d$, $\{\th^*_i\}_{i=0}^d$ is a nonzero element of $\R^\text{\rm asym}$.
The result follows.
\end{proof}

\begin{lemma}   \label{lem:new2}    \samepage
\ifDRAFT {\rm lem:new2}. \fi
The following hold.
\begin{itemize}
\item[\rm (i)]
Each of $\th_0$, $\th_d$, $\th^*_0$, $\th^*_d$ is nonzero.
\item[\rm (ii)]
$\th_1/\th_0 = \th_{d-1}/\th_d = \th^*_1 / \th^*_0 = \th^*_{d-1}/\th^*_d$.
\end{itemize}
\end{lemma}

\begin{proof}
(i) 
By Lemmas \ref{lem:s1si} and \ref{lem:new1}.

(ii)
We have $\th_1 \th_d = \th_0 \th_{d-1}$ by Corollary \ref{cor:const}.
The result follows from this and Lemma \ref{lem:new1}.
\end{proof}

The following definition is motivated by Lemma \ref{lem:cibi}.

\begin{definition}   \label{def:cibi}    \samepage
\ifDRAFT {\rm def:cibi}. \fi
Define scalars $\{c_i\}_{i=1}^d$, $\{b_i\}_{i=0}^{d-1}$ by
\begin{align}
c_i &= \frac{\th_1 \th^*_i - \th_0 \th^*_{i+1}}{\th^*_{i-1} - \th^*_{i+1}}  \qquad (1 \leq i \leq d-1),
  \qquad\qquad  c_d = \th_0,                                                            \label{eq:defci2}
\\
b_i &= \frac{\th_1 \th^*_i - \th_0 \th^*_{i-1}}{\th^*_{i+1} - \th^*_{i-1}} \qquad (1 \leq i \leq d-1),
  \qquad\qquad  b_0 = \th_0.                                                           \label{eq:defbi2}
\end{align}
\end{definition}

\begin{lemma}    \label{lem:bd-i}    \samepage
\ifDRAFT {\rm lem:bd-i}. \fi
We have $c_i = b_{d-i}$ for $1 \leq i \leq d$.
\end{lemma}

\begin{proof}
Compare \eqref{eq:defci2}, \eqref{eq:defbi2} using the fact that the sequence
$\{\th^*_i\}_{i=0}^d$ is antisymmetric.
\end{proof}

\begin{lemma}    \label{lem:cibinonzero2}    \samepage
\ifDRAFT {\rm lem:cibinonzero2}. \fi
The scalars $\{c_i\}_{i=1}^d$, $\{b_i\}_{i=0}^{d-1}$ are all nonzero.
\end{lemma}

\begin{proof}
By Lemma \ref{lem:new1},
$c_i$ is a nonzero scalar multiple of $\th_1 \th_i - \th_0 \th_{i+1}$ 
for $1 \leq i \leq d-1$ and $c_d=\th_0$.
By this and Lemma \ref{lem:s1si}, $c_i \neq 0$ for $1 \leq i \leq d$.
By this and Lemma \ref{lem:bd-i}, $b_i \neq 0$ for $0 \leq i \leq d-1$.
\end{proof}

\begin{lemma}    \label{lem:cibithsi}    \samepage
\ifDRAFT {\rm lem:cibithsi}. \fi
The following hold:
\begin{itemize}
\item[\rm (i)]
$c_d = \th_0$, $\;b_0 = \th_0$ and
\begin{align}
c_i + b_i &= \th_0     \qquad\qquad   (1 \leq i \leq d-1);      \label{eq:ci+bi}
\end{align}
\item[\rm (ii)]
$c_d \th^*_{d-1} = \th_1 \th^*_d$, $\;b_0 \th^*_1 = \th_1 \th^*_0$ and
\begin{align}
 c_i \th^*_{i-1} + b_i \th^*_{i+1} &= \th_1 \th^*_i     \qquad\qquad (1 \leq i \leq d-1).  \label{eq:ciths}
\end{align}
\end{itemize}
\end{lemma}

\begin{proof}
(i)
Use Definition \ref{def:cibi}.

(ii)
To get \eqref{eq:ciths} use Definition \ref{def:cibi}.
Concerning the first two equations, use $c_d=\th_0$, $b_0 = \th_0$,
and Lemma \ref{lem:new2}.
\end{proof}

\begin{definition}   \label{def:AAs}    \samepage
\ifDRAFT {\rm def:AAs}. \fi
Define matrices $A$, $A^*$ in $\Mat_{d+1}(\F)$ by
\begin{align*}
 A &=
  \begin{pmatrix}
    0 & b_0 &    & & & \text{\bf 0}                  \\
    c_1 & 0 & b_1   \\
         & c_2  & \cdot & \cdot  \\
         &      & \cdot & \cdot & \cdot \\
         &       &         & \cdot & \cdot & b_{d-1} \\
     \text{\bf 0}   &        &          &         & c_d & 0   \\
  \end{pmatrix},
&
 A^* &= \text{diag}(\th^*_0, \, \th^*_1, \, \ldots, \, \th^*_d).
\end{align*}
\end{definition}

Recall the vector space $V = \F^{d+1}$.
We are going to show that $A,A^*$ is a TB tridiagonal pair on $V$,
and $\{\th_i\}_{i=0}^d$ (resp.\ $\{\th^*_i\}_{i=0}^d$) is a standard ordering
of the eigenvalues of $A$ (resp.\ $A^*$).

We first consider the eigenspaces of $A^*$, and the action of $A$ on
these eigenspaces.
Define matrices $\{E^*_i\}_{i=0}^d$ in $\Mat_{d+1}(\F)$ by
\begin{align*}
 E^*_i &= \text{diag}(0,\, \ldots, 0, \, \stackrel{i}{1}, \, 0, \, \ldots, \, 0)
    \qquad\qquad  (0 \leq i \leq d).
\end{align*}
For $0 \leq i \leq d$ define $V^*_i = E^*_i V$.
Observe that $V^*_i$ is a subspace of $V$ with dimension $1$.

\begin{lemma}    \label{lem:Vsi}    \samepage
\ifDRAFT {\rm lem:Vsi}. \fi
For $0 \leq i \leq d$, 
$V^*_i$ is the eigenspace of $A^*$ with eigenvalue $\th^*_i$.
Moreover
\begin{align*}
   A V^*_i &\subseteq V^*_{i-1} + V^*_{i+1}   \qquad\qquad (0 \leq i \leq d),
\end{align*}
where $V^*_{-1}=0$, $V^*_{d+1}=0$.
\end{lemma}

\begin{proof}
By the form of the matrices $A$, $A^*$ in Definition \ref{def:AAs}.
\end{proof}

Next we consider the eigenspaces of $A$, and the action of $A^*$ on these
eigenspaces.
Each eigenspace of $A$ has dimension $1$, as we saw below \eqref{eq:Arij}.
But conceivably $A$ is not diagonalizable.
For $0 \leq i \leq d$ define
\[
   V_i = \{v \in V \,|\, A v = \th_i v\}.
\]
The subspace $V_i$ is nonzero if and only if $\th_i$ is an eigenvalue
of $A$, and in this case $V_i$ has dimension $1$.
Our next goal is to show that each of $\{\th_i\}_{i=0}^d$ is an eigenvalue of $A$,
and $A^* V_i \subseteq V_{i-1} + V_{i+1}$
for $0 \leq i \leq d$, where $V_{-1}=0$ and $V_{d+1}=0$.

By Lemma \ref{lem:rec1} there exists $\varrho^* \in \F$ such that
\begin{align}
 \varrho^* &= \th^{*2}_{i-1} - \beta \th^*_{i-1} \th^*_i + \th^{*2}_i   \qquad\qquad  (1\leq i \leq d).
                                  \label{eq:rhos}
\end{align}

\begin{lemma}    \label{lem:paramAW}    \samepage
\ifDRAFT {\rm lem:paramAW}. \fi
We have
\begin{equation}
  A^{*2} A - \beta A^* A A^* + A A^{*2} = \varrho^* A.        \label{eq:aux}
\end{equation}
\end{lemma}

\begin{proof}
Let $Z$ denote the left-hand side of \eqref{eq:aux} minus the right-hand side of \eqref{eq:aux}.
We show that $Z=0$.
To do this, we use Definition \ref{def:AAs} and matrix multiplication to show that
each entry of $Z$ is zero.
For $0 \leq i,j \leq d$ the $(i,j)$-entry of $Z$ is
\begin{equation}
  Z_{i,j} =  (\th^{*2}_i - \beta \th^*_i \th^*_j + \th^{*2}_j - \varrho^*) A_{i,j}.    \label{eq:Zij}
\end{equation}
First assume that $|i-j| \neq 1$. 
Then $A_{i,j}=0$ so $Z_{i,j}=0$.
Next assume that $|i-j|= 1$.
Then in \eqref{eq:Zij} the first factor on the right is zero so $Z_{i,j}=0$.
We have shown $Z=0$, and the result follows.
\end{proof}

\begin{lemma}    \label{lem:AsVi2}    \samepage
\ifDRAFT {\rm lem:AsVi2}. \fi
For $0 \leq i \leq d$ the scalar $\th_i$ is an eigenvalue of $A$.
Moreover
\begin{align}
 A^* V_i &\subseteq V_{i-1} + V_{i+1}   \qquad\qquad   (0 \leq i \leq d),    \label{eq:AsVi2}
\end{align}
where $V_{-1}=0$ and $V_{d+1}=0$.
\end{lemma}

\begin{proof}
We first show that for $0 \leq i \leq d-1$,
\begin{align}
A^* V_i &\subseteq V_0 + V_1 + \cdots + V_{i-1} + V_{i+1},    \label{eq:target1}
\\
A^* V_i &\not\subseteq V_0 + V_1 + \cdots + V_{i-1}.      \label{eq:target2}
\end{align}
We prove this using induction on $i$.
Define vectors $v_0$, $v_1$ in $V$ by
\begin{align}
v_0 &= (1,1,\ldots,1)^\text{t},   \qquad\qquad
v_1 = (\th^*_0, \th^*_1, \ldots, \th^*_d)^\text{t}.    \label{eq:defv0v1}
\end{align} 
The vectors $v_0$, $v_1$ are nonzero.
Using Lemma \ref{lem:cibithsi} one finds that $A v_0 = \th_0 v_0$ and
$A v_1 = \th_1 v_1$. So $v_0 \in V_0$ and $v_1 \in V_1$.
Using the form of $A^*$ in Definition \ref{def:AAs} we find $A^* v_0 = v_1$. 
So \eqref{eq:target1}, \eqref{eq:target2} hold for $i=0$.
Assume that $1 \leq i \leq d-1$.
By induction,
\begin{align}
A^* V_{i-1} &\subseteq V_0 + V_1 + \cdots + V_{i-2} + V_i,    \label{eq:aux1}
\\
A^* V_{i-1} &\not\subseteq V_0 + V_1 +  \cdots + V_{i-2}.     \label{eq:aux2}
\end{align}
By \eqref{eq:aux2},  $V_{i-1} \neq 0$.
Pick $0 \neq v \in V_{i-1}$, and note that $v$ is a basis for $V_{i-1}$.
By \eqref{eq:aux1} there exist $w \in V_0 + V_1 +  \cdots + V_{i-2}$ and $v' \in V_i$
such that $A^* v = w + v'$. 
By \eqref{eq:aux2} the vector $v'$ is nonzero and hence  a basis for $V_i$.
We apply each side of \eqref{eq:aux} to $v$.
Evaluate each term using $A^* v = w+v'$,
and simplify the result using 
$\th_{i-1}-\beta \th_i = - \th_{i+1}$ to get
\begin{equation}
  (A-\th_{i+1} I)A^* v' =
 - \th_{i-1} A^* w + \beta A^* A w - A A^* w + \varrho^* \th_{i-1} v.   \label{eq:aux3}
\end{equation}
By construction $Aw \in V_0  + \cdots + V_{i-2}$.
By induction $A^* (V_0 + \cdots + V_{i-2}) \subseteq V_0  + \cdots + V_{i-1}$.
By these comments, the right-hand side of \eqref{eq:aux3} is contained
in $V_0 + \cdots + V_{i-1}$.
So \eqref{eq:aux3} yields
\[
  (A - \th_{i+1} I) A^* v' \in V_0 + V_1 + \cdots + V_{i-1}.
\]
Thus there exist vectors
$\{v_r\}_{r=0}^{i-1}$ in $V$ such that
$v_r \in V_r$ $(0 \leq r \leq i-1)$ and
$(A-\th_{i+1}I)A^* v' = \sum_{r=0}^{i-1} v_r$.
Define $v'' \in V$ by
\[
  v'' = A^* v' + \sum_{r=0}^{i-1} \frac{v_r}{\th_{i+1} - \th_r}.
\]
One routinely checks $A v'' = \th_{i+1} v''$. So $v'' \in V_{i+1}$.
Therefore \eqref{eq:target1} holds.
We show \eqref{eq:target2}.
Assume by way of contradiction that 
$A^* V_i \subseteq V_0 + V_1 +  \cdots + V_{i-1}$.
Define $W = V_0 + V_1 +  \cdots + V_i$.
Then $A^* W \subseteq W$.
By construction $AW \subseteq W$.
We have $W \neq  0$ since $i \geq 0$,
and $W \neq V$ since $i \leq d-1$.
These comments contradict Lemma \ref{lem:As}(iii).
Thus \eqref{eq:target2} holds.
We have now shown that \eqref{eq:target1} and \eqref{eq:target2}
hold for $0 \leq i \leq d-1$.

Next we show that for $1 \leq i \leq d$,
\begin{align}
A^* V_i &\subseteq V_d + V_{d-1} + \cdots + V_{i+1} + V_{i-1},    \label{eq:target1b}
\\
A^* V_i &\not\subseteq V_d + V_{d-1} + \cdots + V_{i+1}.      \label{eq:target2b}
\end{align}
We prove this using induction on $i=d,d-1,\ldots,1$.
Define vectors $v_d$, $v_{d-1}$ in $V$ by
\begin{align}
v_d &= (1,-1,1, -1, \ldots)^\text{t},
\qquad\qquad
v_{d-1} = (\th^*_0, -\th^*_1, \th^*_2, - \th^*_3, \ldots)^\text{t}.    \label{eq:defvdvd-1}
\end{align} 
Each of $v_d$, $v_{d-1}$ is nonzero.
Using Lemma \ref{lem:cibithsi} and the fact that $\{\th_i\}_{i=0}^d$ is antisymmetric,
we obtain $A v_d = \th_d v_d$ and
$A v_{d-1} = \th_{d-1} v_{d-1}$. So $v_d \in V_d$ and $v_{d-1} \in V_{d-1}$.
Using the form of $A^*$ in Definition \ref{def:AAs} and
\eqref{eq:defvdvd-1},
we obtain  $A^* v_d = v_{d-1}$. 
So \eqref{eq:target1b} and \eqref{eq:target2b} hold for $i=d$.
Assume that $1 \leq i \leq d-1$.
Adjusting the proof of \eqref{eq:target1} and \eqref{eq:target2},
we obtain \eqref{eq:target1b} and \eqref{eq:target2b}.
We have now shown that \eqref{eq:target1b}  and \eqref{eq:target2b}
hold for $1 \leq i \leq d$.

By \eqref{eq:target2} or \eqref{eq:target2b} we see that $V_i$ is nonzero for $0 \leq i \leq d$.
Consequently $\th_i$ is an eigenvalue of $A$ for $0 \leq i \leq d$.
Comparing \eqref{eq:target1} and \eqref{eq:target1b} we obtain \eqref{eq:AsVi2}. 
\end{proof}

\begin{lemma}    \label{lem:TDpair}    \samepage
\ifDRAFT {\rm lem:TDpair}. \fi
The pair $A,A^*$ is a TB tridiagonal pair on $V$.
Moreover,
$\{\th_i\}_{i=0}^d$ (resp.\ $\{\th^*_i\}_{i=0}^d$)
is a standard ordering of the eigenvalues of $A$ (resp.\ $A^*$).
\end{lemma}

\begin{proof}
We verify the conditions (i)--(iv) in Definition \ref{def:TBTDpair}.
By construction $A^*$ is diagonalizable.
By Lemma \ref{lem:AsVi2},  $A$ is diagonalizable.
Thus condition (i) holds.
Conditions (ii) and (iii) hold by Lemmas \ref{lem:AsVi2} and \ref{lem:Vsi}, respectively.
Condition (iv) holds by Lemma \ref{lem:As}(iii).
Thus $A,A^*$ is a TB tridiagonal pair on $V$.
We have shown that $\{V_i\}_{i=0}^d$ (resp.\ $\{V^*_i\}_{i=0}^d$)
is a standard ordering of the eigenspaces of $A$ (resp.\ $A^*$).
The result follows.
\end{proof}

For $0 \leq i \leq d$ let $E_i \in \text{\rm End}(V)$ denote the projection onto $V_i$.
Define
\[
   \Phi = (A; \{E_i\}_{i=0}^d; A^*; \{E^*_i\}_{i=0}^d).
\]

\begin{lemma}   \label{lem:Phi}  \samepage
\ifDRAFT {\rm lem:Phi}. \fi
The above $\Phi$ is a TB tridiagonal system on $V$ with eigenvalue array \eqref{eq:parray1}.
\end{lemma}

\begin{proof}
By Lemma \ref{lem:TDpair} the pair $A,A^*$ is a TB tridiagonal pair on $V$.
By Lemma \ref{lem:TDpair} and the construction,
$\{E_i\}_{i=0}^d$ (resp.\ $\{E^*_i\}_{i=0}^d$) is a standard ordering
of the primitive idempotents of $A$ (resp.\ $A^*$).
The result follows.
\end{proof}

By Lemmas \ref{lem:onlyif}, \ref{lem:Phi}
we have now proved Theorem \ref{thm:main}.

\section{The eigenvalue array}
\label{sec:array}

In this section we introduce the notion of an eigenvalue array over $\F$.
Fix an integer $d \geq 1$ and consider a sequence of
scalars taken from $\F$:
\begin{equation}
 (\{\th_i\}_{i=0}^d; \{\th^*_i\}_{i=0}^d).           \label{eq:parray3}
\end{equation}

\begin{definition}    \label{def:eigenarray}    \samepage
\ifDRAFT {\rm def:eigenarray}. \fi
The sequence \eqref{eq:parray3} is called an {\em eigenvalue array over $\F$} 
whenever it satisfies
conditions {\rm (i)--(iii)} in Theorem \ref{thm:main}.
We call $d$ the {\em diameter} of the eigenvalue array.
\end{definition}

\begin{definition}     \label{def:corresponingPhi}    \samepage
\ifDRAFT {\rm def:correspondingPhi}. \fi
Assume that \eqref{eq:parray3} is an eigenvalue array over $\F$.
By the {\em corresponding TB tridiagonal system} we mean
the one from Theorem \ref{thm:main}.
\end{definition}

\begin{definition}    \label{def:paramAWseq}    \samepage
\ifDRAFT {\rm def:paramAWseq}. \fi
Assume that \eqref{eq:parray3} is an eigenvalue array over $\F$.
By an {\em Askey-Wilson sequence for \eqref{eq:parray3}}
we mean an Askey-Wilson sequence for the corresponding TB tridiagonal system.
\end{definition}

\begin{definition}    \label{def:arrayfund}    \samepage
\ifDRAFT {\rm def:arrayfund}. \fi
Assume that \eqref{eq:parray3} is an eigenvalue array over $\F$.
By a {\em fundamental parameter for \eqref{eq:parray3}} we mean a fundamental
parameter for the corresponding TB tridiagonal system.
\end{definition}

We have some comments.

\begin{lemma}    \label{lem:Char2}   \samepage
\ifDRAFT {\rm lem:Char2}. \fi
Assume that $\text{\rm Char}(\F) =2$.
Then there does not exist an eigenvalue array over $\F$ with diameter $d$.
\end{lemma}

\begin{proof}
Combine conditions (i), (iii) in Theorem \ref{thm:main}.
\end{proof}

\begin{lemma}    \label{lem:fundunique}   \samepage
\ifDRAFT {\rm lem:fundunique}. \fi
Assume that \eqref{eq:parray3} is an eigenvalue array over $\F$.
\begin{itemize}
\item[\rm (i)]
Assume that $d \leq 2$.
Then any scalar $\beta \in \F$ is a fundamental parameter for \eqref{eq:parray3}.
\item[\rm (ii)]
Assume that $d \geq 3$.
Then there exists a unique fundamental parameter $\beta$ for \eqref{eq:parray3}.
\item[\rm (iii)]
Assume that $d \geq 3$ and $d$ is odd.
Then $\beta \neq -2$.
\end{itemize}
\end{lemma}

\begin{proof}
(i)
By Note \ref{note:d12}.

(ii)
By Definition \ref{def:arrayfund} there exists at least one fundamental parameter $\beta$.
This parameter satisfies $\th_0 - \beta \th_1 + \th_2 =0$ and
$\th_1 - \beta \th_2 + \th_3=0$.
At least one of $\th_1$, $\th_2$ is nonzero since $\th_1 \neq \th_2$.
By these comments $\beta$ is uniquely determined by $\th_0$, $\th_1$, $\th_2$,  $\th_3$.
The result follows.

(iii)
By Lemma \ref{lem:deven}.
\end{proof}

\begin{lemma}    \label{lem:rhorhosunique}    \samepage
\ifDRAFT {\rm lem:rhorhosunique}. \fi
Assume that \eqref{eq:parray3} is an eigenvalue array over $\F$.
Let $\beta$ denote a fundamental parameter for \eqref{eq:parray3}.
Then there exists a unique ordered pair $\varrho, \varrho^*$
such that $\beta,\varrho,\varrho^*$ is an Askey-Wilson sequence for \eqref{eq:parray3}.
\end{lemma}

\begin{proof}
The scalars $\varrho$, $\varrho^*$ exist by Proposition \ref{prop:AW} and
Definition \ref{def:AWseq}.
They are unique by \eqref{eq:AWrho}, \eqref{eq:AWrhos}.
\end{proof}

\begin{lemma}   \label{lem:rhorhos}    \samepage
\ifDRAFT {\rm lem:rhorhos}. \fi
Assume that \eqref{eq:parray3} is an eigenvalue array over $\F$.
Let  $\beta, \varrho, \varrho^*$ denote an Askey-Wilson sequence for \eqref{eq:parray3}.
Then the following hold.
\begin{itemize}
\item[\rm (i)]
Assume that $d$ is even. Then
\begin{align*}
  \varrho &= \th^2_r, &
  \varrho^* &= \th^{*2}_r, & & \text{where \quad $r=d/2-1$}.
\end{align*}
\item[\rm (ii)]
Assume that $d$ is odd. Then
\begin{align*}
  \varrho &= (\beta+2) \th^2_r, &
  \varrho^* &= (\beta+2) \th^{*2}_r, & & \text{where \quad $r=(d-1)/2$}.
\end{align*}
\end{itemize}
\end{lemma}

\begin{proof}
(i)
By Theorem \ref{thm:main}(iii),
$\th_{r+1}=0$.
By this and \eqref{eq:AWrho} at $i=r+1$ we get $\varrho = \th_r^2$.
Similar for $\varrho^*$.

(ii)
By Theorem \ref{thm:main}(iii),
$\th_r + \th_{r+1}=0$.
By this and \eqref{eq:AWrho} at $i=r+1$ we get  $\varrho = (\beta+2) \th_r^2$.
Similar for $\varrho^*$.
\end{proof}

\begin{lemma}    \label{lem:th0}    \samepage
\ifDRAFT {\rm lem:th0}. \fi
Assume that \eqref{eq:parray3} is an eigenvalue array over $\F$.
Then the following hold.
\begin{itemize}
\item[\rm (i)]
Assume that $d$ is even.
Then 
\[
    \th_i = 0 \;\; \text{ if and only if } \;\;  i=d/2 \qquad\qquad (0 \leq i \leq d).
\]
\item[\rm (ii)]
Assume that $d$ is odd.
Then
\[
   \th_i \neq 0 \qquad\qquad (0 \leq i \leq d).
\]
\end{itemize}
\end{lemma}

\begin{proof}
We have $\text{\rm Char}(\F) \neq 2$ by Lemma \ref{lem:Char2}.
The result follows in view of Theorem \ref{thm:main}(i), (iii).
\end{proof}

\begin{lemma}    \label{lem:rho0}    \samepage
\ifDRAFT {\rm lem:rho0}. \fi
Assume that \eqref{eq:parray3} is an eigenvalue array over $\F$.
Let $\beta, \varrho, \varrho^*$ denote an Askey-Wilson sequence for \eqref{eq:parray3}.
Then the following are equivalent:

\text{\rm (i) } $\varrho=0$;
\qquad
\text{\rm (ii) } $\varrho^* = 0$;
\qquad
\text{\rm (iii) } $d=1$ and  $\beta = -2$.
\end{lemma}

\begin{proof}
First assume that $d$ is even.
Then each of $\varrho$, $\varrho^*$ is nonzero by 
Lemmas \ref{lem:rhorhos}(i) and \ref{lem:th0}(i).
Next assume that $d$ is odd.
Then (i)--(iii) are equivalent
by Lemmas \ref{lem:fundunique}(iii), \ref{lem:rhorhos}(ii), \ref{lem:th0}(ii).
\end{proof}

\begin{lemma}    \label{lem:rhononzero}    \samepage
\ifDRAFT {\rm lem:rhononzero}. \fi
Assume that \eqref{eq:parray3} is an eigenvalue array over $\F$.
Then \eqref{eq:parray3} has an Askey-Wilson sequence $\beta, \varrho, \varrho^*$
such that $\varrho$, $\varrho^*$ are nonzero.
\end{lemma}

\begin{proof}
By Lemma \ref{lem:rho0} and since the choice of $\beta$ is arbitrary for $d=1$.
\end{proof}

\begin{definition}   \label{def:inter}    \samepage
\ifDRAFT {\rm def:inter}. \fi
Assume that \eqref{eq:parray3} is an eigenvalue array over $\F$.
By the {\em corresponding intersection numbers} (resp.\ {\em dual intersection numbers})
we mean the intersection numbers (resp.\ dual intersection numbers) of
the corresponding TB tridiagonal system.
\end{definition}

\section{The TB tridiagonal systems in closed form}
\label{sec:closedform}

In this section we give in closed form the eigenvalue array and the (dual) intersection numbers
of a TB tridiagonal system.
We also display the Askey-Wilson relations for the associated TB tridiagonal pair.
Fix an integer $d \geq 1$, and consider a sequence of scalars taken from $\F$:
\begin{equation}
   (\{\th_i\}_{i=0}^d; \{\th^*_i\}_{i=0}^d).                \label{eq:parray7}
\end{equation}

As a warmup we will handle the cases $d=1$, $d=2$ separately.

\begin{lemma}    \label{lem:d1}    \samepage
\ifDRAFT {\rm lem:d1}. \fi
For $d=1$ the following {\rm (i), (ii)} are equivalent:
\begin{itemize}
\item[\rm (i)]
the sequence  \eqref{eq:parray7} is an eigenvalue array over $\F$;
\item[\rm (ii)]
$\text{\rm Char}(\F) \neq 2$,
the scalars $\th_0$, $\th^*_0$ are nonzero,
and $\th_1 = - \th_0$, $\th^*_1 = - \th^*_0$.
\end{itemize}
Assume that {\rm (i), (ii)} hold.
Then the corresponding (dual) intersection numbers are
\[
  c_1=\th_0, \quad b_0 = \th_0, \quad c^*_1 = \th^*_0, \quad b^*_0 = \th^*_0.
\]
Moreover, the Askey-Wilson sequences for \eqref{eq:parray7} are
\[
    \beta, \quad (\beta+2)\th_0^2, \quad (\beta+2)\th^{*2}_0   \qquad \qquad\qquad (\beta \in \F).
\]
In addition,
\begin{align*}
 A A^* &= - A^* A,   \qquad\qquad  A^2 = \th_0^2 I, 
  \qquad\qquad A^{*2} = \th^{*2}_0 I.
\end{align*} 
\end{lemma}

\begin{proof}
(i) $\Rightarrow$ (ii)
By Theorem \ref{thm:main}(iii),
$\th_1 = - \th_0$ and $\th^*_1 = - \th^*_0$.
By this and Theorem \ref{thm:main}(i),
$\text{\rm Char}(\F) \neq 2$ and $\th_0 \neq 0$, $\th^*_0 \neq 0$.

(ii) $\Rightarrow$ (i)
One verifies the conditions (i)--(iii) in Theorem \ref{thm:main}.
Therefore \eqref{eq:parray7} is an eigenvalue array over $\F$.

Assume that {\rm (i), (ii)} hold.
Then the intersection numbers and dual intersection numbers
are obtained using Lemma \ref{lem:cibi}.
The Askey-Wilson sequences are obtained by
Lemmas \ref{lem:fundunique}(i) and \ref{lem:rhorhos}(ii).
By (ii) above, $A$ has eigenvalues $\th_0$, $- \th_0$. So $A^2 = \th_0^2 I$.
Similarly $A^{*2} = \th^{*2}_0 I$.
By \eqref{eq:AsVi}, $A^*$ swaps the eigenspaces of $A$. By this we find that $A A^* = - A^* A$.
\end{proof}

\begin{lemma}    \label{lem:d2}    \samepage
\ifDRAFT {\rm lem:d2}. \fi
For $d=2$ the following {\rm (i), (ii)} are equivalent:
\begin{itemize}
\item[\rm (i)]
the sequence  \eqref{eq:parray7} is an eigenvalue array over $\F$;
\item[\rm (ii)]
$\text{\rm Char}(\F) \neq 2$,
the scalars $\th_0$, $\th^*_0$ are nonzero,
and 
\begin{align*}
\th_1 &=0,  \qquad  \th_2 = - \th_0,  \qquad  \th^*_1 =0, \qquad  \th^*_2 = - \th^*_0.
\end{align*}
\end{itemize}
Assume that {\rm (i), (ii)} hold.
Then the corresponding (dual) intersection numbers are
\begin{align*}
c_1 &= \th_0/2, & c_2 &= \th_0, & c^*_1 &= \th^*_0/2, & c^*_2 &= \th^*_0,
\\
b_0 &= \th_0, & b_1 &= \th_0/2,  & b^*_0 &= \th^*_0,  & b^*_1 &= \th^*_0/2.
\end{align*}
Moreover, the Askey-Wilson sequences for \eqref{eq:parray7} are
\[
    \beta, \quad \th_0^2, \quad \th^{*2}_0   \qquad \qquad\qquad (\beta \in \F).
\]
In addition,
\begin{align}
A A^* A &= 0,  \qquad\qquad   A^2 A^* + A^* A^2 = \th_0^2 A^*,     \label{eq:AWd2no1}
\\
A^* A A^* &=0, \qquad\qquad A^{*2} A + A A^{*2} = \th^{*2}_0 A.   \label{eq:AWd2no2}
\end{align}
\end{lemma}

\begin{proof}
For the assertions above \eqref{eq:AWd2no1},
the proof is similar to Lemma \ref{lem:d1}.
To obtain \eqref{eq:AWd2no1},
set $\varrho = \th_0^2$ in \eqref{eq:AW1} and use the fact that
$\beta$ is arbitrary.
Equation \eqref{eq:AWd2no2} is similarly obtained.
\end{proof}

For the rest of this section, assume that $d \geq 3$.
In Examples \ref{ex:1}--\ref{ex:4} below, 
we display all the parameter arrays over $\F$ with diameter $d$.
In each case we display the (dual) intersection numbers from Definition \ref{def:cibi}, 
as well as the corresponding Askey-Wilson sequence $\beta, \varrho, \varrho^*$.
\begin{example}   \label{ex:1}   \samepage
\ifDRAFT {\rm ex:1}. \fi
Assume that $\text{\rm Char}(\F)$ is $0$ or greater than $d$.
Assume that there exist nonzero $h$, $h^* \in \F$ such that
\begin{align*}
  \th_i &= h (d-2i),
\qquad\qquad
  \th^*_i = h^* (d-2i)  
&& (0 \leq i \leq d).
\end{align*}
Then \eqref{eq:parray7} is an eigenvalue array over $\F$ with
fundamental parameter $\beta=2$.
The corresponding (dual) intersection numbers and the scalars $\varrho$, $\varrho^*$ are
\begin{align*}
 c_i &=  h i, &  c^*_i &= h^* i &&  (1 \leq i \leq d),
\\
 b_i &= h(d-i), & b^*_i &= h^*(d-i) && (0 \leq i \leq d-1),
\\
 \varrho &= 4 h^2, & \varrho^* &= 4 h^{*2}.
\end{align*}
\end{example}

\begin{proof}
Define $\beta=2$.
One routinely verifies that each of $\{\th_i\}_{i=0}^d$, $\{\th^*_i\}_{i=0}^d$
is $\beta$-recurrent and antisymmetric, so contained in  $\R^\text{\rm asym}$.
By Lemma \ref{lem:closedpre4} the subspace $\R^\text{\rm asym}$
is MutDist,
so each of $\{\th_i\}_{i=0}^d$, $\{\th^*_i\}_{i=0}^d$ is MutDist.
By these comments \eqref{eq:parray7} satisfies conditions (i)--(iii)
in Theorem \ref{thm:main}.
Thus \eqref{eq:parray7} is an eigenvalue array over $\F$ with fundamental
parameter $\beta$.
To get the (dual) intersection numbers and the scalars
$\varrho$, $\varrho^*$,  use Lemma \ref{lem:cibi} and
\eqref{eq:AWrho}, \eqref{eq:AWrhos}.
\end{proof}

The following Examples can be verified in a similar way.

\begin{example}   \label{ex:2}   \samepage
\ifDRAFT {\rm ex:2}. \fi
Assume that $d$ is even.
Assume that $\text{\rm Char}(\F)$ is $0$ or greater than $d$.
Assume that there exist nonzero $h$, $h^* \in \F$ such that
\begin{align*}
  \th_i &= h (d-2i)(-1)^i,
&
  \th^*_i &= h^* (d-2i)(-1)^i
&& (0 \leq i \leq d).
\end{align*}
Then \eqref{eq:parray7} is an eigenvalue array over $\F$ with
fundamental parameter $\beta=- 2$.
The corresponding (dual) intersection numbers and the scalars $\varrho$, $\varrho^*$ are
\begin{align*}
 c_i &=  h i, &  c^*_i &= h^* i &&  (1 \leq i \leq d),
\\
 b_i &= h(d-i), & b^*_i &= h^*(d-i) && (0 \leq i \leq d-1),
\\
 \varrho &= 4 h^2, & \varrho^* &= 4 h^{*2}.
\end{align*}
\end{example}

\begin{example}   \label{ex:3}   \samepage
\ifDRAFT {\rm ex:3}. \fi
Assume that $d$ is even and $\text{\rm Char}(\F) \neq 2$.
Let $q$ denote a nonzero scalar in $\overline{\F}$ such that
\begin{align*}
q^2+q^{-2} &\in \F,   &
q^{2i} &\neq 1 \quad (1 \leq i \leq d), &
q^{2i} &\neq -1 \quad (1 \leq i \leq d-1).
\end{align*}
Assume that there exist nonzero $h$, $h^* \in \F$ such that
\begin{align*}
  \th_i &= \frac{ h (q^{d-2i} - q^{2i-d}) }{q^2-q^{-2}}, 
&
  \th^*_i &= \frac{ h^* (q^{d-2i} - q^{2i-d}) }{q^2-q^{-2}}
&& (0 \leq i \leq d).
\end{align*}
Then \eqref{eq:parray7} is an eigenvalue array over $\F$ with
fundamental parameter $\beta=q^2+q^{-2}$.
Moreover $\beta \neq \pm 2$.
The corresponding (dual) intersection numbers and the scalars $\varrho$, $\varrho^*$ are
\begin{align*}
 c_i &=  \frac{h (q^{2i} - q^{-2i})}{(q^2-q^{-2})(q^{d-2i} + q^{2i-d})}  \qquad (1 \leq i \leq d-1),
 &   c_d &=  \frac{ h (q^{d} - q^{-d}) }{q^2-q^{-2}}, 
\\
b_i &=  \frac{h (q^{2d-2i} - q^{2i-2d})}{(q^2-q^{-2})(q^{d-2i} + q^{2i-d})}  \qquad (1 \leq i \leq d-1),
&  b_0 &= \frac{ h (q^{d} - q^{-d}) }{q^2-q^{-2}}, 
\\
 \varrho &= h^2.
\end{align*}
To get $\{c^*_i\}_{i=1}^d$, $\{b^*_i\}_{i=0}^{d-1}$, $\varrho^*$,
replace $h$ with $h^*$ in the above.
\end{example}

\begin{example}   \label{ex:4}   \samepage
\ifDRAFT {\rm ex:4}. \fi
Assume that $d$ is odd and  $\text{\rm Char}(\F) \neq 2$.
Let $q$ denote a nonzero scalar in $\overline{\F}$ such that
\begin{align*}
q^2+q^{-2} &\in \F,   &
q^{2i} &\neq 1 \quad (1 \leq i \leq d), &
q^{2i} &\neq -1 \quad (1 \leq i \leq d-1).
\end{align*}
Assume that there exist nonzero $h$, $h^* \in \F$ such that
\begin{align*}
  \th_i &= \frac{ h (q^{d-2i} - q^{2i-d}) }{q-q^{-1}}, 
&
  \th^*_i &= \frac{ h^* (q^{d-2i} - q^{2i-d}) }{q-q^{-1}}
&& (0 \leq i \leq d).
\end{align*}
Then \eqref{eq:parray7} is an eigenvalue array over $\F$ with
fundamental parameter $\beta=q^2+q^{-2}$.
Moreover $\beta \neq \pm 2$.
The corresponding (dual) intersection numbers and the scalars $\varrho$, $\varrho^*$ are
\begin{align*}
 c_i &=  \frac{h (q^{2i} - q^{-2i})}{(q-q^{-1})(q^{d-2i} + q^{2i-d})}  \qquad (1 \leq i \leq d-1),
 &   c_d &=  \frac{ h (q^{d} - q^{-d}) }{q-q^{-1}}, 
\\
b_i &=  \frac{h (q^{2d-2i} - q^{2i-2d})}{(q-q^{-1})(q^{d-2i} + q^{2i-d})}  \qquad (1 \leq i \leq d-1),
&  b_0 &= \frac{ h (q^{d} - q^{-d}) }{q-q^{-1}}, 
\\
 \varrho &= h^2  (q+q^{-1})^2.
\end{align*}
To get $\{c^*_i\}_{i=1}^d$, $\{b^*_i\}_{i=0}^{d-1}$, $\varrho^*$,
replace $h$ with $h^*$ in the above.
\end{example}

\begin{theorem}   \label{thm:thths}    \samepage
\ifDRAFT {\rm thm:thths}. \fi
Every eigenvalue array over $\F$ with diameter $d \geq 3$ is listed
in exactly one of the Examples \ref{ex:1}--\ref{ex:4}.
\end{theorem}

\begin{proof}
Assume that \eqref{eq:parray7} is an eigenvalue array over $\F$,
and let $\beta$ denote its fundamental parameter.
Observe that each of $\{\th_i\}_{i=0}^d$, $\{\th^*_i\}_{i=0}^d$ is contained in $\R^\text{\rm asym}$.
By this and Lemma \ref{lem:feasible},
$\R^\text{\rm asym}$ is MutDist.
By this and Lemma \ref{lem:closedpre4} the conditions \eqref{eq:conditions} hold.
Let the sequence $\{\sigma_i\}_{i=0}^d$ be from Lemmas \ref{lem:asymbasis1}, \ref{lem:asymbasis2}.
By Corollary \ref{cor:basis}, $\{\sigma_i\}_{i=0}^d$ is a basis for $\R^\text{\rm asym}$.
Therefore there exist $h$, $h^* \in \F$ such that
$\th_i = h \sigma_i$, $\th^*_i = h^* \sigma_i$ for $0 \leq i \leq d$.
Now \eqref{eq:parray7} is listed in the following Examples:
\[
\begin{array}{c|c}
\text{\rm Case} & \text{\rm Listed in}
\\ \hline
\beta = 2 & \text{\rm Example \ref{ex:1}}          \rule{0mm}{2.5ex}
\\
\beta=-2 & \text{\rm Example \ref{ex:2}}       \rule{0mm}{2.5ex}
\\
\beta \neq \pm 2, \,\, \text{\rm $d$ is even}  & \text{\rm Example \ref{ex:3}}       \rule{0mm}{2.5ex}
\\
\beta \neq \pm 2, \,\, \text{\rm $d$ is odd}  & \text{\rm Example \ref{ex:4}}       \rule{0mm}{2.5ex}
\end{array}
\]
The result follows.
\end{proof}

\begin{note}
For the eigenvalue array in Example \ref{ex:1} (resp.\ \ref{ex:2}) (resp.\ \ref{ex:3}, \ref{ex:4}),
the corresponding TB tridiagonal system is said to have
Krawtchouk type (resp.\ Bannai/Ito type) (resp.\ $q$-Racah type).
\end{note}

We now display the Askey-Wilson relations.
Let $\Phi$ denote a TB tridiagonal system over $\F$
with eigenvalue array \eqref{eq:parray7}.
Let $A,A^*$ denote the TB tridiagonal pair associated with $\Phi$.

\begin{lemma}   \label{lem:normalizedAWrel}    \samepage
\ifDRAFT {\rm lem:normalizedAWrel}. \fi
Assume that $d \geq 3$.
Then with the notation in Examples  \ref{ex:1}--\ref{ex:4},
the Askey-Wilson relations for $A,A^*$ are given as follows:
\[
\begin{array}{c|c}
\text{\rm Case} & \text{\rm Askey-Wilson relations}
\\ \hline
\beta=2 &                    \rule{0mm}{5ex}
 \begin{array}{c}
   A^2 A^* - 2 A A^* A + A^* A^2 = 4 h^2 A^*   \\
   {A^*}^2 A - 2 A^* A A^* + A {A^*}^2 = 4 h^{*2} A
  \end{array}
\\
\beta=-2 &                    \rule{0mm}{5.5ex}
 \begin{array}{c}
   A^2 A^* + 2 A A^* A + A^* A^2 = 4 h^2 A^*    \\
   {A^*}^2 A + 2 A^* A A^* + A {A^*}^2 = 4 h^{*2} A
  \end{array}
\\
\beta \neq \pm 2, \,\, \text{\rm $d$ even} &         \rule{0mm}{5.5ex}
 \begin{array}{c}
  A^2 A^* - (q^2 + q^{-2})A A^* A + A^* A^2 = h^2 A^*   \\
  {A^*}^2 A - (q^2 + q^{-2})A^* A A^* + A {A^*}^2 = h^{*2} A
 \end{array}
\\
\beta \neq \pm 2, \,\, \text{\rm $d$ odd} &         \rule{0mm}{5.5ex}
 \begin{array}{c}
  A^2 A^* - (q^2 + q^{-2})A A^* A + A^* A^2 =  h^2 (q+q^{-1})^2 A^* \\
  {A^*}^2 A - (q^2 + q^{-2})A^* A A^* + A {A^*}^2 = h^{*2} (q+q^{-1})^2 A.
 \end{array}
\end{array}
\]
\end{lemma}

\begin{proof}
Evaluate \eqref{eq:AW1}, \eqref{eq:AW2} using the data
in Examples \ref{ex:1}--\ref{ex:4}.
\end{proof}

We will return to the Askey-Wilson relations in Section \ref{sec:Z3}.

\section{The relatives of a TB tridiagonal system} 
\label{sec:selfdual}

Let $\Phi$ denote a TB tridiagonal system over $\F$.
Recall from Section \ref{sec:TBTDpair} the relatives 
$\Phi^*$, $\Phi^\downarrow$, $\Phi^\Downarrow$.
In this section we discuss how these relatives are related to $\Phi$
at an algebraic level.

\begin{lemma}   \label{lem:relatives2}   \samepage
\ifDRAFT {\rm lem:relatives2}. \fi
Consider a TB tridiagonal system over $\F$:
\[
 \Phi = (A; \{E_i\}_{i=0}^d; A^*; \{E^*_i\}_{i=0}^d).
\]
Then for $g \in \{\downarrow, \,\Downarrow, \, \downarrow \Downarrow\}$
the TB tridiagonal system $\Phi^g$ is isomorphic to the TB tridiagonal system
shown in the table below:
\[
\begin{array}{c|c}
g & \text{\rm $\Phi^g$ is isomorphic to}
\\ \hline
\downarrow &  (A; \{E_i\}_{i=0}^d; - A^*; \{E^*_i\}_{i=0}^d)    \rule{0mm}{2.8ex}
\\
\Downarrow & (- A; \{E_i\}_{i=0}^d;  A^*; \{E^*_i\}_{i=0}^d)   \rule{0mm}{2.5ex}
\\
\downarrow\Downarrow & (- A; \{E_i\}_{i=0}^d; - A^*; \{E^*_i\}_{i=0}^d)   \rule{0mm}{2.5ex}
\end{array}
\]
\end{lemma}

\begin{proof}
Use Lemmas \ref{lem:hhs}, \ref{lem:relative} and Corollary \ref{cor:unique},
together with Theorem \ref{thm:main}(iii).
\end{proof}

Referring to Lemma \ref{lem:relatives2},
the explicit isomorphisms will be displayed later in this section.

\begin{lemma}    \label{lem:relatives3}    \samepage
\ifDRAFT {\rm lem:relatives3}. \fi
Let $A,A^*$ denote a TB tridiagonal pair over $\F$.
Then the following TB tridiagonal pairs are mutually isomorphic:
\begin{align*}
   A,A^*, \qquad A, -A^*, \qquad  -A, A^*, \qquad -A,-A^*.
\end{align*}
\end{lemma}

\begin{proof}
By Lemma \ref{lem:relatives2}.
\end{proof}

\begin{definition}   \label{def:selfdualpair}    \samepage
\ifDRAFT {\rm def:selfdualpair}. \fi
Let $A,A^*$ denote a TB tridiagonal pair over $\F$.
Then $A,A^*$ is said to be {\em self-dual} whenever $A,A^*$ is isomorphic
to $A^*,A$.
\end{definition}

\begin{definition}    \label{def:selfdualsystem}   \samepage
\ifDRAFT {\rm selfdualsystem}. \fi
Let $\Phi$ denote a TB tridiagonal system over $\F$.
Then $\Phi$ is said to be {\em self-dual} whenever $\Phi$ is isomorphic
to $\Phi^*$.
\end{definition}

\begin{lemma}    \label{lem:selfdualsystem}    \samepage
\ifDRAFT {\rm lem:selfdualsystem}. \fi
Let $\Phi$ denote a TB tridiagonal system over $\F$ with
eigenvalue array $(\{\th_i\}_{i=0}^d; \{\th^*_i\}_{i=0}^d)$.
Then the following {\rm (i), (ii)} are equivalent:
\begin{itemize}
\item[\rm (i)]
$\Phi$ is self-dual;
\item[\rm (ii)]
$\th_i = \th^*_i$ for $0 \leq i \leq d$.
\end{itemize}
\end{lemma}

\begin{proof}
By Lemma \ref{lem:relative}, $\Phi^*$ has eigenvalue array
$(\{\th^*_i\}_{i=0}^d; \{\th_i\}_{i=0}^d)$.
By Corollary \ref{cor:unique}, $\Phi$ and $\Phi^*$ are isomorphic
if and only if they have the same eigenvalue array.
The result follows.
\end{proof}

\begin{lemma}    \label{lem:selfdualsystemhhs}   \samepage
\ifDRAFT {\rm lem:selfdualsystemhhs}. \fi
Assume that $d \geq 3$.
Let $\Phi$ denote a TB tridiagonal system over $\F$ with
eigenvalue array $(\{\th_i\}_{i=0}^d; \{\th^*_i\}_{i=0}^d)$.
Let the scalars $h$, $h^*$ be from Examples \ref{ex:1}--\ref{ex:4}.
Then $\Phi$ is self-dual if and only if $h=h^*$.
\end{lemma}

\begin{proof}
By Lemma \ref{lem:selfdualsystem}.
\end{proof}

\begin{lemma}    \label{lem:selfdualsystem2}    \samepage
\ifDRAFT {\rm lem:selfdualsystem2}. \fi
Consider a TB tridiagonal system over $\F$:
\begin{equation}                  \label{eq:Phi2}
   (A; \{E_i\}_{i=0}^d; A^*; \{E^*_i\}_{i=0}^d).
\end{equation}
Then there exists $0 \neq \zeta \in \F$ such that the TB tridiagonal system
\begin{equation}
   (\zeta A; \{E_i\}_{i=0}^d; A^*; \{E^*_i\}_{i=0}^d)               \label{eq:Phid}
\end{equation}
is self-dual.
\end{lemma}

\begin{proof}
By Lemmas \ref{lem:new1} and  \ref{lem:selfdualsystem}.
\end{proof}

\begin{lemma}    \label{lem:selfdualpair2}    \samepage
\ifDRAFT {\rm lem:selfdualpair2}. \fi
Let $A,A^*$ denote a TB tridiagonal pair over $\F$.
Then there exists $0 \neq \zeta \in \F$ such that
the TB tridiagonal pair $\zeta A, A^*$ is self-dual.
\end{lemma}

\begin{proof}
By Lemma \ref{lem:selfdualsystem2}.
\end{proof}

\begin{theorem}    \label{thm:selfdual}    \samepage
\ifDRAFT {\rm thm:selfdual}. \fi
Let $A,A^*$ denote a self-dual TB tridiagonal pair over $\F$.
Then the following TB tridiagonal pairs are mutually isomorphic:
\begin{align*}
   A,A^*, \qquad A, -A^*, \qquad  -A, A^*, \qquad -A,-A^*,
\\
  A^*,A, \qquad A^*,-A, \qquad -A^*,A, \qquad -A^*,-A.
\end{align*}
\end{theorem}

\begin{proof}
By Lemma \ref{lem:relatives3} and Definition \ref{def:selfdualpair}.
\end{proof}

Our next goal is to display the isomorphisms in Lemma \ref{lem:relatives2}.
For the rest of this section,
fix a TB tridiagonal system over $\F$:
\[
  \Phi = (A; \{E_i\}_{i=0}^d; A^*; \{E^*_i\}_{i=0}^d).
\]
Let $(\{\th_i\}_{i=0}^d; \{\th^*_i\}_{i=0}^d)$ denote the eigenvalue array of $\Phi$.

\begin{definition}    \label{def:S}    \samepage
\ifDRAFT {\rm def:S}. \fi
Define
\begin{align}
  S &= \sum_{i=0}^d (-1)^i E_i, \qquad\qquad
 S^* = \sum_{i=0}^d (-1)^i E^*_i.                \label{eq:defS}
\end{align}
\end{definition}

\begin{lemma}   \label{lem:S2}    \samepage
\ifDRAFT {\rm lem:S2}. \fi
We have 
$S^2 = I$ and $S^{*2} = I$.
Moreover $S$ and $S^*$ are invertible.
\end{lemma}

\begin{proof}
Concerning $S$, use $E_i E_j = \delta_{i,j} E_i$ $(0 \leq i,j \leq d)$
and $I=\sum_{i=0}^d E_i$.
The case of $S^*$ is similar.
\end{proof}

\begin{lemma}    \label{lem:SEsi}    \samepage
\ifDRAFT {\rm lem:SEsi}. \fi
The following hold:
\begin{itemize}
\item[\rm (i)]
$S A = A S$; 
\item[\rm (ii)]
$S E_i = E_i S$ for $0 \leq i \leq d$;
\item[\rm (iii)]
$S A^* = - A^* S$;
\item[\rm (iv)]
$S E^*_i = E^*_{d-i} S$ for $0 \leq i \leq d$.
\end{itemize}
\end{lemma}

\begin{proof}
(i), (ii)
By construction.

(iii)
By Lemma \ref{lem:EiAsAsEipre1} and Definition \ref{def:S}.

(iv)
Using \eqref{eq:Ei} we obtain
\begin{align*}
S E^*_i S^{-1} 
   = \prod_{\stackrel{0 \leq j \leq d}{j \neq i}}
               \frac{S A^* S^{-1} - \th^*_j I }{\th^*_i - \th^*_j}.
\end{align*}
Evaluate the above equation using $SA^* S^{-1} = - A^*$ and Theorem \ref{thm:main}(iii)
to get $S E^*_i S^{-1} = E^*_{d-i}$.
The result follows.
\end{proof}

\begin{lemma}    \label{lem:SsEi}    \samepage
\ifDRAFT {\rm lem:SsEi}. \fi
The following hold:
\begin{itemize}
\item[\rm (i)]
$S^* A^* = A^* S^*$;
\item[\rm (ii)]
$S^* E^*_i = E^*_i S^*$ for $0 \leq i \leq d$;
\item[\rm (iii)]
$S^* A = - A S^*$;
\item[\rm (iv)]
$S^* E_i = E_{d-i} S^*$ for $0 \leq i \leq d$.
\end{itemize}
\end{lemma}

\begin{proof}
Apply Lemma \ref{lem:SEsi} to $\Phi^*$.
\end{proof}

\begin{lemma}    \label{lem:S}    \samepage
\ifDRAFT {\rm lem:S}. \fi
The following hold:
\begin{itemize}
\item[\rm (i)]
$S$ is an isomorphism of TB tridiagonal systems from $\Phi^\downarrow$ to
\[
(A; \{E_i\}_{i=0}^d; - A^*; \{E^*_i\}_{i=0}^d).
\]
\item[\rm (ii)]
$S$ is an isomorphism of TB tridiagonal pairs
from $A,A^*$ to $A,-A^*$.
\end{itemize}
\end{lemma}

\begin{proof}
Use Lemma \ref{lem:SEsi}.
\end{proof}

\begin{lemma}    \label{lem:Ss}    \samepage
\ifDRAFT {\rm lem:Ss}. \fi
The following hold:
\begin{itemize}
\item[\rm (i)]
$S^*$ is an isomorphism of TB tridiagonal systems from $\Phi^\Downarrow$ to
\[
(- A; \{E_i\}_{i=0}^d; A^*; \{E^*_i\}_{i=0}^d).
\]
\item[\rm (ii)]
$S^*$ is an isomorphism of TB tridiagonal pairs
from $A,A^*$ to $-A,A^*$.
\end{itemize}
\end{lemma}

\begin{proof}
Use Lemma \ref{lem:SsEi}.
\end{proof}

\begin{lemma}    \label{lem:SSsSsS}    \samepage
\ifDRAFT {\rm lem:SSsSsS}. \fi
We have $S S^* = (-1)^d S^* S$.
\end{lemma}

\begin{proof}
Using Definition \ref{def:S} and Lemma \ref{lem:SEsi}(iv),
\\ \qquad
$\displaystyle S S^* = \sum_{i=0}^d (-1)^i S E^*_i
  = \sum_{i=0}^d (-1)^i E^*_{d-i} S
  = \sum_{i=0}^d (-1)^{d-i} E^*_i S = (-1)^d S^* S$.
\end{proof}

\begin{lemma}  \label{lem:SSs}    \samepage
\ifDRAFT {\rm lem:SSs}. \fi
The following hold:
\begin{itemize}
\item[\rm (i)]
$S S^*$ is an isomorphism of TB tridiagonal systems from
$\Phi^{\downarrow\Downarrow}$ to 
\[
(- A; \{E_i\}_{i=0}^d; - A^*; \{E^*_i\}_{i=0}^d).
\]
\item[\rm (ii)]
$S S^*$ is an isomorphism of TB tridiagonal pairs
from $A,A^*$ to $-A,-A^*$.
\end{itemize}
\end{lemma}

\begin{proof}
Follows from Lemmas \ref{lem:S} and \ref{lem:Ss}.
\end{proof}

For $0 \leq i \leq d$ we define some polynomials in $\F[x]$:
\begin{align*}
 \tau_i(x) &= (x-\th_0)(x-\th_1)\cdots(x-\th_{i-1}), \\
 \eta_i (x) &= (x-\th_d)(x-\th_{d-1})\cdots(x-\th_{d-i+1}), \\
 \tau^*_i (x)&=  (x-\th^*_0)(x-\th^*_1)\cdots(x-\th^*_{i-1}), \\
 \eta^*_i (x) &= (x-\th^*_d)(x-\th^*_{d-1})\cdots(x-\th^*_{d-i+1}).
\end{align*}
 
\begin{theorem}  {\rm (See \cite{NT:selfdual}.)}  
\label{thm:sdiso}   \samepage
\ifDRAFT {\rm thm:sdiso}. \fi
Assume that $\Phi$ is self-dual.
Then the following four elements are equal, and this common element is
an isomorphism of TB tridiagonal systems from $\Phi$ to $\Phi^*$.
\begin{align*}
& \sum_{i=0}^d \eta_{d-i}(A) E^*_0 E_d \tau^*_i(A^*),
&& \sum_{i=0}^d \eta^*_{d-i}(A^*) E_0 E^*_d \tau_i(A),
\\
& \sum_{i=0}^d \tau^*_i(A^*) E_d E^*_0 \eta_{d-i}(A),
&& \sum_{i=0}^d \tau_i (A) E^*_d E_0 \eta^*_{d-i}(A^*).
\end{align*}
Moreover, the above common element is an isomorphism of
TB tridiagonal pairs from $A,A^*$ to $A^*,A$.
\end{theorem}

\section{The $\mathbb{Z}_3$-symmetric Askey-Wilson relations}
\label{sec:Z3}

For convenience,
we adjust our notation as follows.
\begin{center}
\begin{tabular}{c}
\hline
\hspace{4em} From now on we abbreviate $B=A^*$.  \hspace{4em} \rule{0mm}{3ex}
\\[4pt] \hline
\end{tabular}
\end{center}
Let $A,B$ denote a TB tridiagonal pair over $\F$.
We saw in Section \ref{sec:AW} that $A,B$ satisfy the Askey-Wilson relations
\begin{align}
A^2 B - \beta A B A + B A^2 &= \varrho B,                 \label{eq:ABAW1}
\\
B^2 A - \beta B A B + A B^2 &= \varrho^* A.              \label{eq:ABAW2}
\end{align}
A more detailed version of these relations is given in Lemma \ref{lem:normalizedAWrel}.
In this section we put the Askey-Wilson relations in a form said to be
$\mathbb{Z}_3$-symmetric.
This is done by introducing a third element $C$.

For the rest of this section, the following notation is in effect.
Assume that $\F$ is algebraically closed.
Fix an integer $d \geq 1$, and 
let $V$ denote a vector space over $\F$ with dimension $d+1$.
Let $A,B$ denote a TB tridiagonal pair on $V$.
By Lemma \ref{lem:rhononzero}
there exists an Askey-Wilson sequence $\beta, \varrho, \varrho^*$
for $A,B$ such that $\varrho$, $\varrho^*$ are nonzero.

\begin{definition}   \label{def:z}    \samepage
\ifDRAFT {\rm def:z}. \fi
Let $z$, $z'$, $z''$ denote scalars in $\F$ that satisfy
\begin{equation}                              \label{eq:defz}
\begin{array}{c|cc}
\text{\rm Case} &  z' z''  &   z'' z
\\ \hline
\beta = 2 & -\varrho  &  - \varrho^*         \rule{0mm}{2.5ex}
\\
\beta=-2 &  \varrho &  \varrho^*        \rule{0mm}{2.3ex}
\\
\beta \neq \pm 2 & \;\; \varrho (4-\beta^2)^{-1}  \;\; & \;\; \varrho^* (4-\beta^2)^{-1} \;\;        \rule{0mm}{2.3ex}
\end{array}
\end{equation}
Note that $z z' z'' \neq 0$.
\end{definition}

\begin{proposition}    \label{prop:Z3AW1pre}     \samepage
\ifDRAFT {\rm prop:Z3AW1pre}. \fi
Assume that $\beta = 2$.
Then there exists $C \in \text{\rm End}(V)$ such that
\begin{align}
  B C -  C B &= z A,    \label{eq:AW1pre1}
\\
  C A - A C &= z' B,   \label{eq:AW1pre2}
\\
 A B - B A &= z'' C.    \label{eq:AW1pre3}
\end{align}
\end{proposition}

\begin{proof}
Define $C$ by \eqref{eq:AW1pre3}.
One verifies \eqref{eq:AW1pre1} and \eqref{eq:AW1pre2} using
\eqref{eq:ABAW1}, \eqref{eq:ABAW2}, \eqref{eq:defz}.
\end{proof}

\begin{proposition}    \label{prop:Z3AW2pre}     \samepage
\ifDRAFT {\rm prop:Z3AW2pre}. \fi
Assume that $\beta = - 2$.
Then there exists $C \in \text{\rm End}(V)$ such that
\begin{align}
    B C + C B &= z A,    \label{eq:AW2pre1}
\\
  C A + A C &=  z' B,   \label{eq:AW2pre2}
\\
 A B + B A &=  z'' C.    \label{eq:AW2pre3}
\end{align}
\end{proposition}

\begin{proof}
Similar to the proof of Proposition \ref{prop:Z3AW1pre}.
\end{proof}

\begin{proposition}    \label{prop:Z3AW3pre}     \samepage
\ifDRAFT {\rm prop:Z3AW3pre}. \fi
Assume that $\beta \neq \pm 2$.
Let $0 \neq q \in \F$ be such that $\beta = q^2+q^{-2}$.
Then there exists $C \in \text{\rm End}(V)$ such that
\begin{align}
 \frac{q B C - q^{-1} C B}{q^2-q^{-2}} &= z A ,    \label{eq:AW3pre1}
\\
 \frac{q C A - q^{-1} A C}{q^2-q^{-2}} &= z' B,   \label{eq:AW3pre2}
\\
 \frac{q A B - q^{-1} B A}{q^2-q^{-2}} &= z'' C.    \label{eq:AW3pre3}
\end{align}
\end{proposition}

\begin{proof}
Similar to the proof of Proposition \ref{prop:Z3AW1pre}.
\end{proof}

\begin{lemma}    \label{lem:z}    \samepage
\ifDRAFT {\rm lem:z}. \fi
Assume that $A,B$ is self-dual.
Then in Definition \ref{def:z},
the scalars $z$, $z'$, $z''$ can be chosen such that $z=z' = z''$.
\end{lemma}

\begin{proof}
Since $\varrho = \varrho^*$ and  $\F$ is algebraically closed.
\end{proof}

The next three results follow from Propositions \ref{prop:Z3AW1pre}--\ref{prop:Z3AW3pre}.

\begin{corollary}    \label{cor:Z3AW1}     \samepage
\ifDRAFT {\rm corollary:Z3AW1}. \fi
Assume that $\beta=2$.
Then for  
\[
  \rho = \rho^* = 4,  \qquad\qquad
 z = z' = z'' = 2 \sqrt{-1},
\]
the equations \eqref{eq:AW1pre1}--\eqref{eq:AW1pre3} become
\begin{align}
  B C - C B &= 2 \sqrt{-1} \, A,         \label{eq:AW1no1}
\\
  C A - A C &=  2 \sqrt{-1} \, B,          \label{eq:AW1no2}
\\
  A B - B A &=  2 \sqrt{-1} \, C.        \label{eq:AW1no3}
\end{align}
\end{corollary}

\begin{corollary}    \label{cor:Z3AW2}     \samepage
\ifDRAFT {\rm cor:Z3AW2}. \fi
Assume that $\beta=-2$.
Then for  
\[
\rho = \rho^* = 4,  \qquad \qquad
z = z' = z'' = 2,
\]
the equations \eqref{eq:AW2pre1}--\eqref{eq:AW2pre3} become
\begin{align}
  B C + C B &= 2 A,             \label{eq:AW2no1}
\\
  C A + A C &=  2 B,              \label{eq:AW2no2}
\\
  A B + B A &=  2 C.                  \label{eq:AW2no3}
\end{align}
\end{corollary}

\begin{corollary}    \label{cor:Z3AW3}     \samepage
\ifDRAFT {\rm cor:Z3AW3}. \fi
Assume that $\beta \neq \pm 2$.
Then for  
\[
  \rho = \rho^* = 4-\beta^2,  \qquad \qquad
  z = z' = z'' = 1,
\]
the equations \eqref{eq:AW3pre1}--\eqref{eq:AW3pre3} become
\begin{align}
\frac{q B C - q^{-1} C B}{q^2-q^{-2} } &= A,    \label{eq:AW3no1}
\\
\frac{q C A - q^{-1} A C}{q^2-q^{-2} }  &= B,      \label{eq:AW3no2}
\\
\frac{ q A B - q^{-1} B A}{q^2-q^{-2}}   &=  C.       \label{eq:AW3no3}
\end{align}
\end{corollary}

\begin{note}
The equations \eqref{eq:AW1no1}--\eqref{eq:AW1no3}
are the defining relations
for $\mathfrak{sl}_2$ in the Pauli presentation.
Here are some details.
Assume that $\text{\rm Char}(\F) \neq 2$.
Let $\mathfrak{sl}_2$ denote the Lie algebra over $\F$ consisting of the $2$ by $2$ matrices
that have entries in $\F$ and trace $0$. 
The Lie bracket is $[r,s]=r s - s r$.
Consider the Pauli matrices (see \cite{Pauli}):
\begin{align*}
S_1 &= \begin{pmatrix} 0 & 1 \\ 1 & 0 \end{pmatrix}, &
S_2 &= \begin{pmatrix} 0 & - \sqrt{-1}  \\ \sqrt{-1} & 0 \end{pmatrix}, &
S_3 &= \begin{pmatrix} 1 & 0  \\ 0 & -1 \end{pmatrix}.
\end{align*}
These matrices form a basis for $\mathfrak{sl}_2$,
and satisfy the relations
\begin{align*}
[S_1, S_2]  &= 2 \sqrt{-1} \, S_3,  &
[S_2, S_3] &= 2 \sqrt{-1} \, S_1,   &
[S_3, S_1] &= 2 \sqrt{-1} \, S_2.
\end{align*}
\end{note}

\begin{note}
The equations \eqref{eq:AW2no1}--\eqref{eq:AW2no3} 
 are essentially the defining relations for
the anticommutator spin algebra (see \cite{Arik}).
Here are some details.
In \cite{Arik} Arik and Kayserilioglu introduced an $\F$-algebra
by generators $J_1,J_2,J_3$ and relations
\begin{align*}
 J_2 J_3 + J_3 J_2 &= J_1, &
 J_3 J_1 + J_1 J_3 &= J_2, &
 J_1 J_2 + J_2 J_1 &= J_3.
\end{align*}
This algebra is called the {\em anticommutator spin algebra}.
Observe that the above relations coincide with \eqref{eq:AW2no1}--\eqref{eq:AW2no3}
by setting $J_1 = A/2$, $J_2 = B/2$, $J_3 = C/2$.
\end{note}

\begin{note}
The equations \eqref{eq:AW3no1}--\eqref{eq:AW3no3} 
are essentially the 
defining relations for the quantum algebra $U_q(\mathfrak{so}_3)$.
Here are some details.
The algebra  $U_q(\mathfrak{so}_3)$ has a presentation by generators
$I_1, I_2, I_3$ and relations
\begin{align}
 q^{1/2}  I_2 I_3 - q^{-1/2} I_3 I_2 &= I_1,         \label{eq:I1}
\\
  q^{1/2}  I_3 I_1 - q^{-1/2} I_1 I_3 &= I_2,        \label{eq:I2}
\\
 q^{1/2}  I_1 I_2 - q^{-1/2} I_2 I_1 &= I_3.         \label{eq:I3}
\end{align}
As far as we know,
the equations \eqref{eq:I1}--\eqref{eq:I3} were first considered by 
Santilli \cite{San}.
Later in \cite{Fairlie}, Fairlie discovered that \eqref{eq:I1}--\eqref{eq:I3} show up in $U_q(\mathfrak{so}_3)$.
The fact that \eqref{eq:I1}--\eqref{eq:I3} give a presentation for $U_q(\mathfrak{so}_3)$ 
was proved by Odesski in \cite{odesski}.
See \cite{havlicek} for a more precise history.
The relations \eqref{eq:I1}--\eqref{eq:I3}
become \eqref{eq:AW3no1}--\eqref{eq:AW3no3}
by first setting
$I_1 = A/(q-q^{-1})$, 
$I_2 = B/(q-q^{-1})$, 
$I_3 = C/(q-q^{-1})$,
and then replacing $q$ with $q^2$.
\end{note}

\section{The elements $W$, $W'$, $W''$ and the automorphism $\rho$}
\label{sec:W}

Throughout this section, the following notation is in effect.
Assume that $\F$ is algebraically closed.
Fix an integer $d \geq 1$,
and let $V$ denote a vector space over $\F$ with dimension $d+1$.
Let $A,B$ denote a TB tridiagonal pair on $V$.
In view of Lemma \ref{lem:selfdualpair2}, assume that $A,B$ is self-dual.
Let $\beta,\varrho,\varrho^*$ denote the Askey-Wilson sequence for $A,B$
from above Definition \ref{def:z},
and note that $\varrho = \varrho^*$.
For the case $\beta\neq \pm 2$, fix $0 \neq q \in \F$ as in Proposition \ref{prop:Z3AW3pre}.

\begin{lemma}   \label{lem:th}    \samepage
\ifDRAFT {\rm lem:th}. \fi
Let $\{\th_i\}_{i=0}^d$ denote an eigenvalue sequence of $A,B$.
Then there exists $0 \neq h \in \F$ such that the following holds:
\begin{equation}          \label{eq:ththd}
\begin{array}{c|c}
 \text{\rm Case} &  \th_i 
\\ \hline
\beta=2 &  h(d-2i)   \rule{0mm}{2.7ex}
\\
\beta = -2 & h(d-2i)(-1)^i     \rule{0mm}{2.5ex}
\\
\beta \neq \pm 2 &\;\; h(q^{d-2i} - q^{2i-d}) \;\;    \rule{0mm}{2.5ex}
\end{array}
\end{equation}
\end{lemma}

\begin{proof}
By Lemmas \ref{lem:d1}, \ref{lem:d2} 
and Examples \ref{ex:1}--\ref{ex:4}, together with the assumption 
that $\F$ is algebraically closed.
\end{proof}

\begin{lemma}    \label{lem:rho}    \samepage
\ifDRAFT {\rm lem:rho}. \fi
Referring to Lemma \ref{lem:th},
\begin{equation}          \label{eq:rhorhod}
\begin{array}{c|cc}
 \text{\rm Case} &  \varrho
\\ \hline
\beta=2 &  4 h^2   \rule{0mm}{2.7ex}
\\
\beta = -2 & 4 h^2      \rule{0mm}{2.5ex}
\\
\beta \neq \pm 2 & \;\; h^2 (q^2-q^{-2})^2 \;\; \rule{0mm}{2.5ex}
\end{array}
\end{equation}
\end{lemma}

\begin{proof}
Use Proposition \ref{prop:recurrence} and Lemma \ref{lem:th}.
\end{proof}

In view of Lemma \ref{lem:z},
we choose the scalars $z$, $z'$, $z''$ in Definition \ref{def:z} such that $z=z'=z''$.
Let  $C \in \text{\rm End}(V)$ be from 
Propositions \ref{prop:Z3AW1pre}--\ref{prop:Z3AW3pre}.
Next we define some invertible elements $W$, $W'$ in $\text{\rm End}(V)$,
and use them to construct an automorphism $\rho$ of $\text{\rm End}(V)$
that sends 
\begin{align*}
A &\mapsto B, \qquad\qquad B \mapsto C, \qquad\qquad C \mapsto A.
\end{align*}
For $0 \leq i \leq d$ let $E_i$ (resp.\ $E'_i$) denote the primitive idempotent
of $A$ (resp.\ $B$) for the eigenvalue $\th_i$,
where $\th_i$ is from Lemma \ref{lem:th}.

\begin{definition}    \label{def:WWd}    \samepage
\ifDRAFT {\rm def:WWd}. \fi
Define  $W$, $W' \in \text{\rm End}(V)$ by
\begin{align}
W &= \sum_{i=0}^d t_i E_i, \qquad\qquad
W' = \sum_{i=0}^d t_i E'_i,                 \label{eq:defWWd}
\end{align}
where
\begin{equation}    \label{eq:deft}
\begin{array}{c|c}
\text{\rm Case} & t_i
\\ \hline
\beta=2 & 2^i h^i z^{-i}        \rule{0mm}{2.7ex}
\\
\beta=-2 & (-1)^{\lfloor i/2 \rfloor}     2^i h^i z^{-i}     \rule{0mm}{2.5ex}
\\
\beta \neq \pm 2 & h^i z^{-i} q^{i (d-i)}         \rule{0mm}{2.5ex}
\end{array}
\end{equation}
\end{definition}

\begin{lemma}    \label{lem:Winv}   \samepage
\ifDRAFT {\rm lem:Winv}. \fi
The elements $W$, $W'$ are invertible, with inverse
\begin{align*}
W^{-1} &= \sum_{i=0}^d t_i^{-1} E_i, \qquad\qquad
(W')^{-1} = \sum_{i=0}^d t_i^{-1} E'_i.  
\end{align*}
\end{lemma}

\begin{proof}
By construction.
\end{proof}

Recall the antiautomorphism $\dagger$ of $\text{\rm End}(V)$, from Lemma \ref{lem:comments2}.

\begin{lemma}    \label{lem:Wdag}   \samepage
\ifDRAFT {\rm lem:Wdag}. \fi
The antiautomorphism $\dagger$ fixes each of
$W$, $W'$.
\end{lemma}

\begin{proof}
By \eqref{eq:defWWd} and since $\dagger$ fixes $E_i$ and $E'_i$
for $0 \leq i \leq d$.
\end{proof}

\begin{lemma}    \label{lem:WAWdB}    \samepage
\ifDRAFT {\rm lem:WAWdB}. \fi
We have 
\begin{align}
   A W &=  W A,         
\qquad\qquad
  B W' = W' B.             \label{eq:WAWdB}
\end{align}
\end{lemma}

\begin{proof}
The element $A$ commutes with $W$ by Definition \ref{def:WWd} and since
$A$ commutes with $E_i$ for $0 \leq i \leq d$.
Similarly $B$ commutes with $W'$.
\end{proof}

\begin{lemma}   \label{lem:WBWdC}   \samepage
\ifDRAFT {\rm lem:WBWdC}. \fi
We have
\begin{align}
 B W &=  W C,      
\qquad\qquad
 C W' =  W' A.             \label{eq:WBWdC}
\end{align}
\end{lemma}

\begin{proof}
We first obtain $BW= WC$.
By \eqref{eq:AW1pre3}, \eqref{eq:AW2pre3}, \eqref{eq:AW3pre3} we have
\[
  C = e A B + e' B A,
\]
where
\begin{equation}        \label{eq:eta}
\begin{array}{c|cc}
\text{\rm Case} & e & e'
\\ \hline
\beta = 2 & z^{-1}  & - z^{-1}    \rule{0mm}{2.5ex}
\\
\beta = -2 & z^{-1}  & z^{-1}    \rule{0mm}{2.5ex}
\\
\beta \neq \pm 2 & \;\; \displaystyle \frac{q z^{-1}}{ q^2-q^{-2}}  \;\; 
                    &\displaystyle - \frac{q^{-1}z^{-1} }{q^2-q^{-2}}       \rule{0mm}{3.5ex}
\end{array}
\end{equation}
It suffices to show that
\begin{equation}   \label{eq:WBWinv2}
    e A B + e' B A -  W^{-1} B W = 0.
\end{equation}
To obtain \eqref{eq:WBWinv2}, we show that
\begin{align}
E_i (e A B + e' B A - W^{-1} B W) E_j &=0   \label{eq:tar}
\end{align}
for $0 \leq i,j \leq d$.
Let $i$, $j$ be given.
Using $E_i A = \th_i E_i$, $A E_j = \th_j E_j$, $E_i W^{-1} = t_i^{-1} E_i$, $W E_j = t_j E_j$
 we find that the left-hand side of \eqref{eq:tar}
is equal to
\begin{align}
  (e \th_i + e' \th_j -  t_i^{-1} t_j) E_i B E_j.    \label{eq:titj}
\end{align}
First assume that $|i-j| \neq 1$.
Then $E_i B E_j = 0$ by Lemma \ref{lem:EsiAEsj},
so \eqref{eq:titj} is zero.
Next assume that $|i-j| = 1$.
Using \eqref{eq:defz}, \eqref{eq:ththd}, \eqref{eq:rhorhod}, \eqref{eq:deft}, \eqref{eq:eta}
one routinely finds that 
in \eqref{eq:titj} the coefficient of $E_i B E_j$ is zero.
Therefore \eqref{eq:titj} is zero as desired.
We have obtained $BW= WC$.
The equation $CW'= W' A$ is similarly obtained.
\end{proof}

\begin{definition}    \label{def:P}    \samepage
\ifDRAFT {\rm def:P}. \fi
Define $P = W' W$.
Note that $P^\dagger = W W'$.
\end{definition}

\begin{lemma}    \label{lem:PA}    \samepage
\ifDRAFT {\rm lem:PA}. \fi
We have
\begin{align}
  A P &=   P B,  \qquad\qquad
  B P =  P C,  \qquad\qquad
  C P =  P A.      \label{eq:PA}
\end{align}
\end{lemma}

\begin{proof}
In the equation on the left in \eqref{eq:WBWdC},
multiply each side on the left by $W'$ and use \eqref{eq:WAWdB}
to get $B P =  P C$.
In the equation on the right in \eqref{eq:WBWdC},
 multiply each side on the right by $W$ and use \eqref{eq:WAWdB}
to get $C P = P A$.
Combining the equations in \eqref{eq:WBWdC} we obtain
$B W W' = W W' A$.
By these comments and Definition \ref{def:P} we obtain
$B P^\dagger = P^\dagger A$.
Applying $\dagger$ we get $A P = P B$.
\end{proof}

\begin{definition}    \label{def:mu}    \samepage
\ifDRAFT {\rm def:mu}. \fi
Define the map $\rho : \text{\rm End}(V) \to \text{\rm End}(V)$, $X \mapsto P^{-1} X P$.
Note that $\rho$ is an automorphism of $\text{\rm End}(V)$ that fixes $P$.
\end{definition}

\begin{corollary}    \label{cor:mu}    \samepage
\ifDRAFT {\rm cor:mu}. \fi
The automorphism $\rho$ sends
\[
 A \mapsto B,  \qquad\qquad  B \mapsto C,  \qquad\qquad C \mapsto A.
\]
\end{corollary}

\begin{proof}
By Lemma \ref{lem:PA}.
\end{proof}

\begin{corollary}    \label{cor:mu3}    \samepage
\ifDRAFT {\rm cor:mu3}. \fi
We have $\rho^3 = 1$.
Moreover, there exists $0 \neq \kappa \in \F$ such that  $P^3 = \kappa I$.
\end{corollary}

\begin{proof}
By Corollary \ref{cor:mu} the element $P^3$ commutes with both $A$, $B$.
By this and Lemma \ref{lem:iso} 
there exists $\kappa \in \F$ such that $P^3 = \kappa I$.
We have $\kappa \neq 0$ since $P$ is invertible.
Now $\rho^3=1$ follows by Definition \ref{def:mu}. 
\end{proof}

\begin{note}   {\rm (See \cite[Corollary 14.11]{T:Ltriple}.)}            \label{note:P3}    \samepage
\ifDRAFT {\rm note:P3}. \fi
Referring to Corollary \ref{cor:mu3}, the scalar $\kappa$ is given as follows:
\[
\begin{array}{c|ccc}
\text{\rm Case} & \beta=2 & \beta = -2 & \beta \neq \pm 2
\\ \hline
\kappa  & (-1)^d 2^{-d} h^{-d} z^d & 1 & (-1)^d h^{-d} z^d q^{d(d-1)}  \rule{0mm}{2.7ex}
\end{array}
\]
\end{note}

By construction and Corollary \ref{cor:mu}, the element $C$ is diagonalizable with
eigenvalues $\{\th_i\}_{i=0}^d$.
For $0 \leq i \leq d$ let $E''_i$ denote the primitive idempotent of $C$
for the eigenvalue $\th_i$.

\begin{lemma}   \label{lem:rhoE}    \samepage
\ifDRAFT {\rm lem:rhoE}. \fi
For $0 \leq i \leq d$ the automorphism $\rho$ sends
\[
 E_i \mapsto E'_i,   \qquad\qquad
 E'_i \mapsto E''_i,  \qquad\qquad
 E''_i \mapsto E_i. 
\]
\end{lemma}

\begin{proof}
We first show that $\rho$ sends $E_i \mapsto E'_i$.
In \eqref{eq:Ei} the element $E_i$ is expressed as a polynomial in $A$;
let $f(x) $ denote the polynomial.
Using the equation on the left in \eqref{eq:PA},
\[
 P^{-1} E_i P = P^{-1} f(A) P = f(P^{-1} A P) = f(B) = E'_i.
\]
Thus $\rho$ sends $E_i \mapsto E'_i$.
Similarly, $\rho$ sends $E'_i \mapsto E''_i$ and $E''_i \mapsto E'_i$.
\end{proof}

\begin{definition}    \label{def:Wdd}    \samepage
\ifDRAFT {\rm def:Wdd}. \fi
Define $W''  \in \text{\rm End}(V)$ by
\[
  W'' = \sum_{i=0}^d t_i E''_i,
\]
where $\{t_i\}_{i=0}^d$ are from \eqref{eq:deft}.
\end{definition}

\begin{lemma}    \label{lem:rhoW}    \samepage
\ifDRAFT {\rm lem:rhoW}. \fi
The automorphism $\rho$ sends
\[
  W \mapsto W', \qquad\qquad
  W' \mapsto W'', \qquad\qquad
  W'' \mapsto W.
\]
In other words
\begin{align}
  W P &= P W', \qquad\qquad
  W' P = P W'', \qquad\qquad
  W'' P = P W.            \label{eq:WP}
\end{align}
\end{lemma}

\begin{proof}
The first assertion is by Definitions \ref{def:WWd}, \ref{def:Wdd} and Lemma \ref{lem:rhoE}.
The second assertion follows in view of Definition \ref{def:mu}.
\end{proof}

\begin{lemma}    \label{lem:PWWd}   \samepage
\ifDRAFT {\rm lem:PWWd}. \fi
The element $P$ is equal to each of
\[
  W' W,  \qquad\qquad
  W'' W', \qquad\qquad
  W W''.
\]
\end{lemma}

\begin{proof}
By Definition \ref{def:P}, $P = W' W$.
Apply $\rho$ to this and use Lemma \ref{lem:rhoW} to get
$P = W'' W'$.
Similarly $P = W W''$.
\end{proof}

\begin{definition}    \label{def:braid}    \samepage
\ifDRAFT {\rm def:braid}. \fi
Elements $X$, $Y$ in $\text{\rm End}(V)$ are said to satisfy the
{\em braid relation} whenever $XYX=YXY$.
\end{definition}

\begin{lemma}   \label{lem:braid}    \samepage
\ifDRAFT {\rm lem:braid}. \fi
Any two of $W$, $W'$, $W''$ satisfy the braid relation.
\end{lemma}

\begin{proof}
By \eqref{eq:WP} and Lemma \ref{lem:PWWd}.
\end{proof}

Corollary \ref{cor:mu} and Lemma \ref{lem:rhoE}
show that $A,B,C$ is a Leonard triple in the sense of Curtin \cite{Cur}.
We will say more about Leonard triples in Section \ref{sec:conc}.

\section{Some antiautomorphisms associated with a TB tridiagonal pair}  
\label{sec:anti}

Throughout this section, the following notation is in effect.
Assume that $\F$ is algebraically closed.
Fix an integer $d \geq 1$,
and let $V$ denote a vector space over $\F$ with dimension $d+1$.
Let $A,B$ denote a self-dual TB tridiagonal pair on $V$.
Let $\beta,\varrho,\varrho^*$ denote the Askey-Wilson sequence for $A,B$
from above Definition \ref{def:z}.
Choose the scalars $z$, $z'$, $z''$ in Definition \ref{def:z} such that $z=z'=z''$,
and let  $C \in \text{\rm End}(V)$ be from 
Propositions \ref{prop:Z3AW1pre}--\ref{prop:Z3AW3pre}.
In this section we obtain some antiautomorphisms of $\text{\rm End}(V)$
that act on $A$, $B$, $C$ in an attractive manner.
Recall the antiautomorphism $\dagger$ of $\text{\rm End}(V)$, from Lemma \ref{lem:comments2}.
By construction $\dagger$ fixes $A$, $B$.

\begin{lemma}  \label{lem:antiauto}    \samepage
\ifDRAFT {\rm lem:antiauto}. \fi
For a map $\xi : \text{\rm End}(V) \to \text{\rm End}(V)$
the following are equivalent:
\begin{itemize}
\item[\rm (i)]
$\xi$ is an antiautomorphism of $\text{\rm End}(V)$;
\item[\rm (ii)]
there exists an invertible $T \in \text{\rm End}(V)$ such that
$X^\xi = T^{-1} X^\dagger T$ for all $X \in \text{\rm End}(V)$.
\end{itemize}
\end{lemma}

\begin{proof}
(i) $\Rightarrow$ (ii)
Consider the composition 
\[
\omega  :  \text{\rm End}(V) \xrightarrow{\;\;\; \dagger \;\;\;} \text{\rm End}(V) 
    \xrightarrow{\;\;\; \xi \;\;\;} \text{\rm End}(V).
\]
The map $\omega$ is an automorphism of $\text{\rm End}(V)$.
By the Skolem-Noether theorem,
there exists an invertible $T \in \text{\rm End}(V)$ such that
$X^\omega = T^{-1} X T$ for all $X \in \text{\rm End}(V)$.
Thus $X^\xi = (X^\dagger)^\omega = T^{-1} X^\dagger T$
for all $X \in \text{\rm End}(V)$.

(ii) $\Rightarrow$ (i)
Clear.
\end{proof}

Let  $\{\th_i\}_{i=0}^d$ denote the eigenvalue sequence of $A,B$
from Lemma \ref{lem:th},
and let $E_i$, $E'_i$ be from above Definition \ref{def:WWd}.
Let $W$, $W'$ be from Definition \ref{def:WWd}, 
and let $P$ be from Definition \ref{def:P}.

\begin{definition}   \label{def:dag}    \samepage
\ifDRAFT {\rm def:dag}. \fi
Define the maps $\dagger'$, $\dagger'' : \text{\rm End}(V) \to \text{\rm End}(V)$,
$X \mapsto T^{-1} X^\dagger T$,
where $T$ is from the table below.
\[
\begin{array}{c|cc}
       &    \dagger' & \dagger ''
\\ \hline
T & P^\dagger P  &  (P P^\dagger)^{-1}      \rule{0mm}{2.7ex}
\end{array}
\]
Note that $\dagger'$, $\dagger''$ are antiautomorphisms of $\text{\rm End}(V)$.
\end{definition}

Recall the automorphism $\rho$ of $\text{\rm End}(V)$, from Definition \ref{def:mu}.

\begin{lemma}   \label{lem:dag}    \samepage
\ifDRAFT {\rm lem:dag}. \fi
The map $\dagger'$ is equal to the composition
\[
 \text{\rm End}(V) \xrightarrow{\;\;\; \rho^{-1} \;\;\;} \text{\rm End}(V) 
    \xrightarrow{\;\;\; \dagger \;\;\;} \text{\rm End}(V)
      \xrightarrow{\;\;\; \rho \;\;\;} \text{\rm End}(V).
\]
The map $\dagger''$ is equal to the composition
\[
 \text{\rm End}(V) \xrightarrow{\;\;\; \rho \;\;\;} \text{\rm End}(V) 
    \xrightarrow{\;\;\; \dagger \;\;\;} \text{\rm End}(V)
      \xrightarrow{\;\;\; \rho^{-1} \;\;\;} \text{\rm End}(V).
\]
\end{lemma}

\begin{proof}
Use Definitions \ref{def:mu} and \ref{def:dag}.
\end{proof}

Our next goal is to describe how $\dagger$, $\dagger'$, $\dagger''$ act on $A$, $B$, $C$.
We will treat separately the cases $\beta=2$, $\beta=-2$, $\beta \neq \pm 2$.

\begin{proposition}    \label{prop:dagger2}    \samepage
\ifDRAFT {\rm prop:dagger2}. \fi
Assume that $\beta = 2$.
Then the antiautomorphisms $\dagger$, $\dagger'$, $\dagger''$ act on $A$, $B$, $C$
as follows:
\begin{itemize}
\item[\rm (i)]
$\dagger$ fixes $A$, $B$ and sends $C \mapsto  -C$;
\item[\rm (ii)]
$\dagger'$ fixes $B$, $C$ and sends $A \mapsto -A$;
\item[\rm (iii)]
$\dagger''$ fixes $C$, $A$ and  sends $B \mapsto -B$.
\end{itemize}
\end{proposition}

\begin{proof}
(i)
By construction $A^\dagger = A$ and $B^\dagger = B$.
Applying $\dagger$ to each side of \eqref{eq:AW1pre3},
we obtain $C^\dagger = - C$.

(ii), (iii)
Use Corollary \ref{cor:mu}, 
Lemma \ref{lem:dag}, and (i) above.
\end{proof}

\begin{proposition}    \label{prop:dagger3}    \samepage
\ifDRAFT {\rm prop:dagger3}. \fi
Assume that $\beta = - 2$.
Then the maps $\dagger$, $\dagger'$, $\dagger''$ coincide,
and this map fixes each of $A$, $B$, $C$.
\end{proposition}

\begin{proof}
By Definitions \ref{def:z}, \ref{def:WWd} and Lemma \ref{lem:rho}
we obtain  $t_i^2 = 1$ for $0 \leq i \leq d$.
So $W^2 = I$ and $(W')^2 = I$.
By this and Definition \ref{def:P},
$P^\dagger P = I$ and $P P^\dagger = I$.
By this and Definition \ref{def:dag}
we obtain $\dagger' = \dagger$ and $\dagger'' = \dagger$.
By construction $\dagger$ fixes $A$, $B$.
By this and \eqref{eq:AW2pre3} we obtain $C^\dagger = C$.
The result follows.
\end{proof}

\begin{proposition}    \label{prop:dagger}    \samepage
\ifDRAFT {\rm prop:dagger}. \fi
Assume that $\beta \neq \pm 2$.
Then the antiautomorphisms $\dagger$, $\dagger'$, $\dagger''$ act on $A$, $B$, $C$
as follows:
\begin{itemize}
\item[\rm (i)]
$\dagger$ fixes $A$, $B$ and sends $C \mapsto C - \frac{AB - BA}{z(q-q^{-1})}$;
\item[\rm (ii)]
$\dagger'$ fixes $B$, $C$ and sends $A \mapsto A - \frac{BC- CB}{z(q-q^{-1})}$;
\item[\rm (iii)]
$\dagger''$ fixes $C$, $A$ and sends $B \mapsto B - \frac{CA - AC}{z(q-q^{-1})}$.
\end{itemize}
\end{proposition}

\begin{proof}
Similar to the proof of Proposition \ref{prop:dagger2},
using \eqref{eq:AW3pre3} instead of \eqref{eq:AW1pre3}.
\end{proof}

\begin{definition}   \label{def:ddag}    \samepage
\ifDRAFT {\rm def:ddag}. \fi
Define maps $\ddagger$, $\ddagger'$, $\ddagger'' : \text{\rm End}(V) \to \text{\rm End}(V)$,
$X \mapsto T^{-1} X^\dagger T$,
where $T$ is from the table below.
\[
\begin{array}{c|ccc}
       &    \ddagger & \ddagger' & \ddagger''
\\ \hline
T & W  & (W')^{-1}  &  W W' W          \rule{0mm}{2.7ex}
\end{array}
\]
Note that $\ddagger$, $\ddagger'$, $\ddagger''$ are antiautomorphisms of $\text{\rm End}(V)$.
\end{definition}

\begin{lemma}   \label{lem:ddag}    \samepage
\ifDRAFT {\rm lem:ddag}. \fi
The map $\ddagger'$ is equal to the composition
\begin{equation}     
 \text{\rm End}(V) \xrightarrow{\;\;\; \rho^{-1} \;\;\;} \text{\rm End}(V) 
    \xrightarrow{\;\;\; \ddagger \;\;\;} \text{\rm End}(V)
      \xrightarrow{\;\;\; \rho \;\;\;} \text{\rm End}(V).                \label{eq:ddaggerd}
\end{equation}
The map $\ddagger''$ is equal to the composition
\begin{equation}
 \text{\rm End}(V) \xrightarrow{\;\;\; \rho \;\;\;} \text{\rm End}(V) 
    \xrightarrow{\;\;\; \ddagger \;\;\;} \text{\rm End}(V)
      \xrightarrow{\;\;\; \rho^{-1} \;\;\;} \text{\rm End}(V).          \label{eq:ddaggerdd}
\end{equation}
\end{lemma}

\begin{proof}
Note by Lemma \ref{lem:Wdag} that $\dagger$ fixes each of $W$, $W'$.
We first show that $\ddagger'$ is equal to the composition \eqref{eq:ddaggerd}.
Pick any $X \in \text{\rm End}(V)$.
By Definition \ref{def:ddag}, $\ddagger'$ sends
$X \mapsto W' X^\dagger (W')^{-1}$.
By Definitions \ref{def:P}, \ref{def:mu}, \ref{def:ddag}
the composition \eqref{eq:ddaggerd} sends
$X \mapsto H^{-1} X^\dagger H$,
where $H = W W' W W' W$.
By Definition \ref{def:P} and Corollary \ref{cor:mu3} we obtain
$H = \kappa (W')^{-1}$.
By these comments $\ddagger'$ is equal to the composition \eqref{eq:ddaggerd}.
One similarly shows that $\ddagger''$ is equal to the
composition \eqref{eq:ddaggerdd}.
\end{proof}

\begin{proposition}    \label{prop:ddagger}    \samepage
\ifDRAFT {\rm prop:ddagger}. \fi
The antiautomorphisms $\ddagger$, $\ddagger'$, $\ddagger''$ act on  $A$, $B$, $C$
as follows:
\begin{itemize}
\item[\rm (i)]
$\ddagger$ fixes $A$ and swaps $B$, $C$;
\item[\rm (ii)]
$\ddagger'$ fixes $B$ and swaps $C$, $A$;
\item[\rm (iii)]
$\ddagger''$ fixes $C$ and swaps $A$, $B$.
\end{itemize}
\end{proposition}

\begin{proof}
Note by Lemma \ref{lem:Wdag} that
$\dagger$ fixes each of $W$, $W'$.

(i)
By Lemma \ref{lem:WAWdB} and Definition \ref{def:ddag}, $A^\ddagger = A$.
By Lemma \ref{lem:WBWdC}, 
$B^\ddagger = C$ and
$C^\ddagger = W^{-1} C^\dagger W
   = (W C W^{-1})^\dagger = B$.

(ii), (iii)
Use Corollary \ref{cor:mu}, Lemma \ref{lem:ddag}, and (i) above.
\end{proof}

The existence of the
antiautomorphisms $\ddagger$, $\ddagger'$, $\ddagger''$ shows
that the Leonard triple $A,B,C$ is modular in the sense of Curtin \cite{Cur}.

\begin{lemma}   \label{lem:ddag2}    \samepage
\ifDRAFT {\rm lem:ddag2}. \fi
We have $\ddagger^2=1$, $(\ddagger')^2=1$, $(\ddagger'')^2 = 1$.
\end{lemma}

\begin{proof}
Each of the given squares is an automorphism of $\text{\rm End}(V)$ that
fixes each of $A$, $B$, $C$.
The result follows in view of Lemma \ref{lem:comments}(i).
\end{proof}

Next we explain how $\ddagger$, $\ddagger'$, $\ddagger''$
are related to  the automorphism $\rho$ from Definition \ref{def:mu}.

\begin{proposition}    \label{prop:muddagger}    \samepage
\ifDRAFT {\rm prop:muddagger}. \fi
The automorphism $\rho$ is equal to each of the following compositions:
\begin{align}
 & \text{\rm End}(V) \xrightarrow{\;\; \ddagger' \;\;}
  \text{\rm End}(V) \xrightarrow{\;\;\, \ddagger \;\;\, }
   \text{\rm End}(V),                                         \label{eq:comp1}                    
\\
&  \text{\rm End}(V) \xrightarrow{\;\, \ddagger'' \;\, }
  \text{\rm End}(V) \xrightarrow{\;\; \ddagger' \;\;}
   \text{\rm End}(V),                                        \label{eq:comp2}
\\
&  \text{\rm End}(V) \xrightarrow{\;\;\, \ddagger \;\;\, }
   \text{\rm End}(V) \xrightarrow{\;\, \ddagger'' \;\,}
   \text{\rm End}(V).                                        \label{eq:comp3}
\end{align}
\end{proposition}

\begin{proof}
We first show that $\rho$ is equal to the composition \eqref{eq:comp1}.
Pick any $X \in \text{\rm End}(V)$.
By Definition \ref{def:mu},
$\rho$ sends $X \mapsto P^{-1} X P$.
By Definition \ref{def:ddag},
the composition \eqref{eq:comp1} sends
$X \mapsto W^{-1} (W')^{-1} X W' W$.
By Definition \ref{def:P}, $P = W' W$.
By these comments $\rho$ is equal to the composition \eqref{eq:comp1}.
By this and Lemma \ref{lem:ddag}
we find that $\rho$ is equal to each of
\eqref{eq:comp2}, \eqref{eq:comp3}.
\end{proof}

\section{An action of ${\rm PSL}_2(\mathbb{Z})$ associated with a TB tridiagonal pair}
\label{sec:auto}

Throughout this section, the following notation is in effect.
Assume that $\F$ is algebraically closed.
Fix an integer $d \geq 1$,
and let $V$ denote a vector space over $\F$ with dimension $d+1$.
Let $A,B$ denote a self-dual TB tridiagonal pair on $V$.
Let $\beta,\varrho,\varrho^*$ denote the Askey-Wilson sequence for $A,B$
from above Definition \ref{def:z}.
Choose the scalars $z$, $z'$, $z''$ in Definition \ref{def:z} such that
 $z=z'=z''$,
and let $C \in \text{\rm End}(V)$ be from 
Propositions \ref{prop:Z3AW1pre}--\ref{prop:Z3AW3pre}.

In this section we display an action of $\text{\rm PSL}_2(\mathbb{Z})$ on
$\text{\rm End}(V)$ as a group of automorphisms,
that act on $A$, $B$, $C$ in an attractive manner.
Recall from \cite{Alperin} that  ${\rm PSL}_2(\mathbb{Z})$ 
has a presentation by generators $r$, $s$ 
and relations $r^3=1$, $s^2=1$.
To get the action of $\text{\rm PSL}_2(\mathbb{Z})$,
we need an automorphism of $\text{\rm End}(V)$ that has order $3$,
and one that has order $2$.
In Corollary \ref{cor:mu3}
we obtained an automorphism $\rho$ of $\text{\rm End}(V)$ that has order $3$.
Next we obtain an automorphism of $\text{\rm End}(V)$ that has order $2$.
Recall the elements $W$, $W'$ from Definition \ref{def:WWd}.

\begin{definition}    \label{def:sig}    \samepage
\ifDRAFT {\rm def:sig}. \fi
Define the map $\sigma : \text{\rm End}(V) \to \text{\rm End}(V)$,
$X \mapsto T X T^{-1}$,
where $T=W W' W$.
Note that $\sigma$ is an automorphism of $\text{\rm End}(V)$.
\end{definition}

Recall the antiautomorphism $\dagger$ of $\text{\rm End}(V)$ from Lemma \ref{lem:comments2},
and the antiautomorphism $\ddagger''$ of $\text{\rm End}(V)$ from Definition \ref{def:ddag}.

\begin{proposition}    \label{prop:sig}    \samepage
\ifDRAFT {\rm prop:sig}. \fi
The automorphism $\sigma$ is equal to the composition
\[
    \text{\rm End}(V) \xrightarrow{\;\; \ddagger'' \;\;} \text{\rm End}(V)
     \xrightarrow{\;\;\, \dagger \;\;\,} \text{\rm End}(V).
\]
\end{proposition}

\begin{proof}
Routine verification using Definitions \ref{def:ddag} and \ref{def:sig}.
\end{proof}

\begin{proposition}     \label{prop:sig2}   \samepage
\ifDRAFT {\rm prop:sig2}. \fi
The automorphism $\sigma$ swaps $A$, $B$ and sends
$C \mapsto C^\dagger$.
\end{proposition}

\begin{proof}
By Propositions \ref{prop:ddagger}(iii) and \ref{prop:sig}.
\end{proof}

\begin{corollary}   \label{cor:sig}  \samepage
\ifDRAFT {\rm cor:sig}. \fi
We have $\sigma^2=1$.
\end{corollary}

\begin{proof}
By Proposition \ref{prop:sig2} the automorphism
$\sigma^2$ fixes both $A$, $B$.
The result follows by Lemma \ref{lem:comments}(i).
\end{proof}

\begin{corollary}    \label{cor:PSL}    \samepage
\ifDRAFT {\rm cor:PSL}. \fi
The group ${\rm PSL}_2(\mathbb{Z})$ acts on $\text{\rm End}(V)$
as a group of automorphisms,
such that $r$ acts as $\rho$ and $s$ acts as $\sigma$.
\end{corollary}

\begin{proof}
By Corollaries \ref{cor:mu3} and \ref{cor:sig}.
\end{proof}

\section{Concluding remarks}
\label{sec:conc}

In earlier sections we discussed TB tridiagonal pairs.
In this section we summarize what is known about general tridiagonal pairs,
using the TB case as a guide.
The definition of a tridiagonal pair is given in the Introduction.
The definition of a tridiagonal system is analogous to
Definition  \ref{def:TBTDsystem}.
For the rest of this section,
let $\Phi=(A; \{E_i\}_{i=0}^d; A^*; \{E^*_i\}_{i=0}^\delta)$ denote a tridiagonal system
on $V$.
It is known that $d=\delta$ (see \cite[Lemma 4.5]{ITT}).
By \cite[Theorem 10.1]{ITT}
there exists a sequence of scalars $\beta,\gamma, \gamma^*, \varrho, \varrho^*$ in $\F$
such that both
\begin{align}
 0 &= [ A,\; A^2 A^* - \beta A A^* A + A^* A^2 - \gamma (A A^* + A^* A) - \varrho A^*],
                                       \label{eq:TDrel1}
\\
 0 &= [A^*, \; A^{*2} A - \beta A^* A A^* + A A^{*2} - \gamma^* (A^*A + A A^*) - \varrho^* A],
                                       \label{eq:TDrel2}
\end{align}
where $[r,s]$ means $r s-s r$.
The above equations are called the {\em tridiagonal relations}.
We now describe the eigenvalues.
For $0 \leq i \leq d$ let $\th_i$ (resp.\ $\th^*_i$) denote the eigenvalue of $A$
(resp.\ $A^*$) associated with $E_i$ (resp.\ $E^*_i$).
By \cite[Theorem 11.1]{ITT} the expressions
\[
   \frac{\th_{i-2} - \th_{i+1}}{\th_{i-1}-\th_{i}},  \qquad\qquad
   \frac{\th^*_{i-2} - \th^*_{i+1}}{\th^*_{i-1}-\th^*_{i}}
\]
are equal and independent of $i$ for $2 \leq i \leq d-1$.
Moreover,
\begin{align*}
  \gamma &= \th_{i-1} - \beta \th_i + \th_{i+1}  &&  (1 \leq i \leq d-1),
\\
  \gamma^* &= \th^*_{i-1} - \beta \th^*_i + \th^*_{i+1} &&  (1 \leq i \leq d-1),
\\
  \varrho &= \th^2_{i-1} - \beta \th_{i-1} \th_i + \th_i^2 - \gamma (\th_{i-1} + \th_i)  && (1 \leq i \leq d),
\\
 \varrho^* &= \th^{*2}_{i-1} - \beta \th^*_{i-1} \th^*_i + \th^{*2}_i - \gamma^* (\th^*_{i-1} + \th^*_i)
                       &&  (1 \leq i \leq d).
\end{align*}
Note that $\gamma$, $\gamma^*$ are zero if $\Phi$ is TB.
One topic that we did not discuss in earlier sections is 
the split decomposition.
This is defined as follows.
For $0 \leq i \leq d$ define
\[
  U_i = (E^*_0 V + \cdots + E^*_i V) \cap (E_i V+ \cdots + E_d V).
\]
By \cite[Theorem 4.6]{ITT} we have $V = \sum_{i=0}^d U_i$ (direct sum).
This sum is called the {\em $\Phi$-split decomposition} of $V$.
By \cite[Theorem 4.6]{ITT}  the elements $A$, $A^*$ act on $\{U_i\}_{i=0}^d$ as follows:
\begin{align}
(A-\th_i I) U_i &\subseteq U_{i+1} \qquad (0 \leq i \leq d-1), 
& (A- \th_d I) U_d &= 0,                                               \label{eq:AUi}
\\
(A^* - \th^*_i I) U_i &\subseteq U_{i-1} \qquad (1 \leq i \leq d),
& (A^* - \th^*_0 I) U_0 &= 0.                                         \label{eq:AsUi}
\end{align}
By \cite[Corollary 5.7]{ITT},
for $0 \leq i \leq d$ the dimensions of $U_i$, $E_i V$, $E^*_i V$ are equal;
denote this common dimension by $\rho_i$.
By \cite[Corollary 5.7]{ITT} we have $\rho_i = \rho_{d-i}$.
The sequence $(\rho_0, \rho_1, \ldots, \rho_d)$ is called the {\em shape of $\Phi$}.
By \cite[Theorem 1.3]{NT:shape} the shape satisfies $\rho_i \leq \rho_0 \binom{d}{i}$ for $0 \leq i \leq d$.
Additional results concerning the split decomposition and the shape can be found in
\cite{N:refine, NT:formula, NT:unit, NT:tde, NT:maps, N:height1, T:split, T:LBUB,  Vidar, IT:shape}.
Some miscellaneous results about tridiagonal pairs and systems can be found in
\cite{AC, Bock1, IT:sharpen,  IT:Mock, Al2}.

The tridiagonal system $\Phi$ is said to be {\em sharp} whenever $\rho_0 = 1$.
If $\F$ is algebraically closed then $\Phi$ is sharp \cite[Theorem 1.3]{NT:structure}.
In \cite[Theorem 3.1]{INT} the sharp tridiagonal systems are classified up to isomorphism.
This result makes heavy use of 
\cite{NT:sharp, NT:towards, NT:structure, NT:mu, NT:shape, NT:qRacahmu, IT:qRacah, IT:Drin}.
For the moment, assume that $\Phi$ is sharp.
By \cite[Theorem 11.5]{NT:sharp} and \cite[Theorem 3.1]{INT},
there exists an antiautomorphism $\dagger$ of $\text{\rm End}(V)$
that fixes each of $A$, $A^*$.
By linear algebra,
there exists a nondegenerate symmetric bilinear form  $\b{\;,\;} : V \times V \to \F$
such that
$\b{Xu, v} = \b{u, X^\dagger v}$ for all $u,v \in V$ and $X \in \text{\rm End}(V)$.
See 
\cite{AC3, NT:sharp, Tanaka} for results on the bilinear form.

We now assume that $\rho_i=1$ for $0 \leq i \leq d$.
In this case, $A,A^*$ is called a {\em Leonard pair} and 
$\Phi$ is called a {\em Leonard system}.
For some surveys on this topic, see \cite{NT:Krawt, T:survey, T:intro}.
For a Leonard pair, we can improve on the tridiagonal relations \eqref{eq:TDrel1}, \eqref{eq:TDrel2}
as follows.
By \cite[Theorem 1.5]{TV} there exist scalars $\omega$, $\eta$, $\eta^*$ in $\F$ such that both
\begin{align}
  A^2 A^* - \beta A A^* A + A^* A^2 - \gamma (A A^* + A^* A)- \varrho A^*
   &=  \gamma^* A^2 + \omega A + \eta I,                         \label{eq:AWrelation1}
\\
 A^{*2} A - \beta A^* A A^* + A A^{*2} - \gamma^* (A^* A + A A^*) - \varrho^* A
  &= \gamma A^{*2} + \omega A^* + \eta^* I.                    \label{eq:AWrelation2}
\end{align}
These equations are  called the {\em Askey-Wilson relations} (see \cite{TV, Z}).
Observe that in \eqref{eq:AWrelation1} the right-hand side is a polynomial in $A$,
and therefore commutes with $A$. Thus \eqref{eq:AWrelation1} implies \eqref{eq:TDrel1}.
Similarly \eqref{eq:AWrelation2} implies \eqref{eq:TDrel2}.
If $A,A^*$ is TB then
the scalars $\gamma$, $\gamma^*$, $\omega$, $\eta$, $\eta^*$ are all zero,
and \eqref{eq:AWrelation1}, \eqref{eq:AWrelation2} become
\eqref{eq:AW1}, \eqref{eq:AW2}.
In some cases the Askey-Wilson relations can be put in $\mathbb{Z}_3$-symmetric form
(see \cite[Theorem 10.1]{Huang}).
Additional results concerning the Askey-Wilson relations can be found in
\cite{Vid1, Vid2}.
We recall the $\Phi$-split basis.
Let $\{U_i\}_{i=0}^d$ denote the $\Phi$-split decomposition of $V$.
Pick any nonzero $v \in E^*_0 V$.
For $0 \leq i \leq d$ define $u_i = \tau_i (A) v$,
where $\tau_i$ is from above Theorem \ref{thm:sdiso}.
By \cite[Section 21]{T:survey} we have $0 \neq u_i \in U_i$.
Consequently the vectors $\{u_i\}_{i=0}^d$ form a basis for $V$,
called the {\em $\Phi$-split basis} for $V$.
With respect to this basis, 
the matrices representing $A$, $A^*$ take the form
\begin{align*} 
 A &:
  \begin{pmatrix}
    \th_0 &  &    & & & \text{\bf 0}                  \\
    1 & \th_1    \\
         & 1  & \cdot  \\
         &      & \cdot & \cdot  \\
         &       &         & \cdot & \cdot  \\
     \text{\bf 0}   &        &          &         & 1 & \th_d   \\
  \end{pmatrix},
&
 A^* &:
  \begin{pmatrix}
    \th^*_0 & \vphi_1 &    & & & \text{\bf 0}                  \\
       & \th^*_1 & \vphi_2   \\
         &    & \cdot & \cdot  \\
         &      &  & \cdot & \cdot \\
         &       &         &  & \cdot & \vphi_d \\
     \text{\bf 0}   &        &          &         & & \th^*_d   \\
  \end{pmatrix},
\end{align*}
where $\{\vphi_i\}_{i=1}^d$ are nonzero scalars in $\F$.
The sequence $\{\vphi_i\}_{i=1}^d$ is called the {\em first split sequence} of $\Phi$.
The first split sequence $\{\phi_i\}_{i=1}^d$ of $\Phi^\Downarrow$ is called the
{\em second split sequence} of $\Phi$.
The sequence 
\[
    (\{\th_i\}_{i=0}^d; \{\th^*_i\}_{i=0}^d; \{\vphi_i\}_{i=1}^d; \{\phi_i\}_{i=1}^d)
\]
is called the {\em parameter array} of $\Phi$.
By \cite[Lemma 3.11]{T:Leonard} $\Phi$ is determined up to isomorphism
by its parameter array.
In \cite[Theorem 1.9]{T:Leonard}
the Leonard systems are classified up to isomorphism
in terms of the parameter array.
In the TB case the first and second split sequences look as follows.
By \cite[Theorem 1.5]{NT:balanced} and \cite[Theorem 23.6]{T:survey},
\[
   \vphi_i = -\phi_i = (\th^*_i - \th^*_{i-1})(\th_0 + \th_1 + \cdots + \th_{i-1})
   \qquad \qquad (1 \leq i \leq d).
\] 
In the table below,
for $1 \leq i \leq d$ the scalar $\vphi_i$ is displayed in closed form.
\[
\begin{array}{c|c}
\text{\rm Case}  &  \vphi_i
\\ \hline
\text{\rm Example \ref{ex:1}} &
 2 h h^* i (i-d-1)                          \rule{0mm}{2.7ex}
\\
\text{\rm Example \ref{ex:2}} &
 h h^* (-1)^i (d+1) (d - 2i + 1)- h h^* (d-2i+1)^2       \rule{0mm}{2.7ex}
\\
\text{\rm Example \ref{ex:3}} &                               \rule{0mm}{2.7ex}
 h h^* (q^{i} - q^{-i})(q^{i-d-1}-q^{d-i+1})(q^{d-2i+1}+q^{2i-d-1}) (q^2-q^{-2})^{-2}
\\
\text{\rm Example \ref{ex:4}} &                        \rule{0mm}{2.7ex}
  h h^* (q^{i} - q^{-i})(q^{i-d-1}-q^{d-i+1})(q^{d-2i+1}+q^{2i-d-1}) (q-q^{-1})^{-2}
\end{array}
\]

Another topic  not discussed in earlier sections,
is the connection between Leonard systems and orthogonal polynomials.
For $0 \leq i \leq d$ define a polynomial $u_i (x) \in \F[x]$ by
\[
 u_i (x) = \sum_{\ell=0}^i 
 \frac{(x-\th_0)(x-\th_1)\cdots (x-\th_{\ell-1}) (\th^*_i-\th^*_0)(\th^*_i - \th^*_1)\cdots(\th^*_i - \th^*_{\ell-1})}
        {\vphi_1 \vphi_2 \cdots \vphi_\ell}.
\]
Define $U \in \Mat_{d+1}(\F)$ that has $(i,j)$-entry $u_i (\th_j)$ for $0 \leq i,j \leq d$.
By \cite[Theorem 15.8]{T:qRacah},
$U$ is the transition matrix from an $A^*$-eigenbasis to an $A$-eigenbasis.
By \cite[Section 5]{T:parray}
the polynomials $\{u_i(x)\}_{i=0}^d$ are from the terminating branch of the
Askey scheme, which consists of the following polynomial families:
$q$-Racah,
$q$-Hahn,
dual $q$-Hahn,
$q$-Krawtchouk,
dual $q$-Krawtchouk,
affine $q$-Krawtchouk,
quantum $q$-Krawtchouk,
Racah,
Hahn,
dual Hahn,
Krawtchouk,
Bannai/Ito,
Orphan.
For a discussion of these polynomials, see \cite[pp. 260--300]{BI}.
In the TB case,
Example \ref{ex:1} corresponds to a special case of Krawtchouk polynomials,
Example \ref{ex:2} corresponds to a special case of Bannai/Ito polynomials,
and 
Examples \ref{ex:3}, \ref{ex:4} correspond  to a special case of $q$-Racah polynomials.
Some miscellaneous results about Leonard pairs and systems can be found in
\cite{Cur2, H, NT:affine, N:endparam, N:endentries, NT:det, NT:switch, 
WangHouGao, WangHouGao4, Han1, Han2, Han3,T:24points, Tanaka2}.

At the end of Section \ref{sec:W} and below Proposition \ref{prop:ddagger},
we mentioned the notion of a Leonard triple.
This notion was introduced in \cite[Definition 1.2]{Cur}.
A Leonard triple on $V$ is a $3$-tuple of elements in $\text{\rm End}(V)$
such that for each map,  
there exists a basis for $V$ with respect to which
the matrix representing that map is diagonal and the matrices representing
the other two maps are irreducible tridiagonal. 
To investigate Leonard triples, the following result is useful.
For a Leonard pair $A,A^*$ on $V$ and for $X \in \text{\rm End}(V)$,
consider the matrices that represent $X$ with respect to a standard eigenbasis for $A$ and $A^*$.
If these matrices are both tridiagonal, then $X$ is a linear combination of
$I$, $A$, $A^*$, $AA^*$, $A^* A$ \cite[Theorem 3.2]{NT:span}.
Using this result, the Leonard triples have recently been classified up to isomorphism.
The classification is summarized as follows.
Using the eigenvalues one breaks down the analysis into four cases,
called $q$-Racah, Racah, Krawtchouk, and Bannai/Ito \cite{ChanGaoHou, HouLiuGao}.
The Leonard triples are classified up to isomorphism
in \cite{Huang} (for $q$-Racah type);
\cite{GWH} (for Racah type);
\cite{KangHouGao} (for Krawtchouk type);
\cite{HouXuGao} (for Bannai/Ito type with even diameter);
\cite{HouWangGao} (for Bannai/Ito type with odd diameter).
Additional results on Leonard triples can be found in
\cite{HouGao, KorZ, T:Ltriple, N:LBTD, TZ} (for $q$-Racah type);
\cite{AlCur, LiuHouGao, Post} (for Racah type);
\cite{BalMa, Mik} (for Krawtchouk type);
\cite{Brown1, BGV, GVZ, GVZ2, HouWangGao3, HouZhangGao, TsuVZ} (for Bannai/Ito type);
\cite{N:TDTD} (for general case).

We mention some algebras related to tridiagonal pairs and Leonard pairs.
The {\em tridiagonal algebra} (see \cite[Definition 3.9]{T:Tworelations})
is defined by two generators subject to the tridiagonal relations
\eqref{eq:TDrel1}, \eqref{eq:TDrel2}.
The {\em Onsager algebra} (see \cite{HT:tetra, Tanabe, Davies1, Davies2})
is the tridiagonal algebra for the case $\beta=2$, $\gamma=0$, $\gamma^*=0$,
$\varrho \neq 0$, $\varrho^* \neq 0$.
The {\em $q$-Onsager algebra} (see \cite{Bas1, Bas4,T:positivepart, T:Lusztig, BasK})
is the tridiagonal algebra for the case $\beta=q^2+q^{-2}$, $\gamma=0$, $\gamma^*=0$,
$\varrho \neq 0$, $\varrho^* \neq 0$.
The {\em Askey-Wilson algebra} (see \cite{Z}) 
is defined by two generators subject to the Askey-Wilson relations
\eqref{eq:AWrelation1}, \eqref{eq:AWrelation2}.
This algebra has a central extension (see \cite{T:AWalgebra}) that we now describe.
Fix nonzero $q \in \F$ such that $q^4 \neq 1$.
The {\em universal Askey-Wilson algebra} $\Delta_q$
is defined by generators and relations in the following way.
The generators are $A$, $B$, $C$.
The relations assert that each of
\begin{align*}
& A + \frac{q BC - q^{-1}CB}{q^2-q^{-2}},
&& B + \frac{q CA - q^{-1}AC}{q^2-q^{-2}},
&& C + \frac{q AB - q^{-1}BA}{q^2-q^{-2}}
\end{align*}
is central in $\Delta_q$.
By  \cite[Theorem 3.1, Theorem 3.11]{T:AWalgebra}
the group  ${\rm PSL}_2(\mathbb{Z})$ acts on $\Delta_q$
as a group of automorphisms,
such that the ${\rm PSL}_2(\mathbb{Z})$-generator $r$ (resp.\ $s$)
sends $A \mapsto B \mapsto C \mapsto A$ (resp.\ $A \leftrightarrow B$).
For more information on $\Delta_q$, see
\cite{Huang2, T:AWalgebra, T:AWalgebra2, TZ, Huang4, Huang3}.
The double affine Hecke algebra (DAHA) for a reduced root system was defined by Cherednik (see \cite{Cherednik}),
and the definition was extended to include nonreduced root systems by Sahi (see \cite{Sahi}).
The most general DAHA of rank 1 is said to have type $(C_1^\vee, C_1)$.
In \cite{IT:DAHA} a universal DAHA  of type $(C_1^\vee , C_1)$ was introduced,
and denoted $\hat{H}_q$.
An injective algebra homomorphism $\Delta_q \to \hat{H}_q$ 
is given in \cite[Section 4]{T:AWalgebraDAHA}; see also \cite{IT:DAHA, Koorn1, Koorn2}.
The paper \cite{NT:DAHA} shows how $\hat{H}_q$ is related to Leonard pairs.
Additional results in the literature link tridiagonal pairs and Leonard pairs with
the Lie algebra $\mathfrak{sl}_2$ (see \cite{AlCur, BT:equitable, IT:Krawt, NT:Krawt, Al3, AC5, AC4}), 
the quantum algebras $U_q(\mathfrak{sl}_2)$ 
   (see \cite{Al, Bock2, BockT, Wor, ITW:equitable, T:uqsl2equitable}),
$U_q(\widehat{\mathfrak{sl}}_2)$ 
 (see \cite{AC2, BT:Borel,  F:RL, IT:uqsl2hat, IT:non-nilpotent, IT:aug, T:Bockting}),
the tetrahedron Lie algebra $\boxtimes$ (see \cite{BT:loop, H:tetra,  IT:loop, MP})
and its $q$-deformation $\boxtimes_q$ (see \cite{IT:tetra, F, IRT, IT:inverting, IT:drg, Y2, Y3}).

Tridiagonal pairs have been used to investigate the open $XXZ$ spin chain and related models in
statistical mechanics (see \cite{Bas1, Bas2, Bas3, Bas4, Bas5, Bas6, Bas7}).
The study of tridiagonal pairs has lead to some conceptual advances,
such as 
the bidiagonal pairs/triples (see \cite{F2, F3}),
the billiard arrays (see \cite{T:Billiard,Y}),
Hessenberg pairs (see \cite{God1, God2, God3}),
and the lowering-raising triples (see \cite{N:LRtriple, N:LRpair, T:LRtriple}).

\section{Acknowledgements}
The authors thank Sarah Bockting-Conrad,
Brian Curtin, Jae-ho Lee, and Hajime Tanaka
for giving the paper a close reading and offering many valuable
suggestions.

\bigskip\bigskip\noindent
Kazumasa Nomura\\
Tokyo Medical and Dental University\\
Kohnodai, Ichikawa, 272-0827 Japan\\
email: knomura@pop11.odn.ne.jp

\bigskip\noindent
Paul Terwilliger\\
Department of Mathematics\\
University of Wisconsin\\
480 Lincoln Drive\\ 
Madison, Wisconsin, 53706 USA\\
email: terwilli@math.wisc.edu

\bigskip\noindent
{\bf Keywords.}
Tridiagonal pair, Leonard pair, tridiagonal matrix
\\
\noindent
{\bf 2010 Mathematics Subject Classification.} 17B37, 15A21

\end{document}